\documentclass[10pt]{article}
\usepackage{amsmath}
\usepackage{amssymb}
\usepackage{cite}
\usepackage{lineno}
\usepackage{microtype}
\DisableLigatures[f]{encoding = *, family = * }
\usepackage{changepage}
\usepackage[utf8x]{inputenc}
\usepackage{textcomp,marvosym}
\usepackage{nameref}
\usepackage{hyperref} 
\usepackage{microtype}
\usepackage{graphicx}
\usepackage[toc,page]{appendix}
\usepackage{setspace} 
\usepackage{amsthm}
\usepackage{authblk}
\usepackage[labelfont=bf,labelsep=period,justification=raggedright]{caption}
\usepackage{siunitx}
\usepackage{wrapfig,lipsum}
\makeatletter
\renewcommand{\@biblabel}[1]{\quad#1.}
\makeatother

\newtheorem{thm}{Theorem}
\newtheorem{lem}{Lemma}

\newtheorem{prop}{Proposition}

\newtheorem{defn}{Definition}

\newtheorem{rem}{Remark}

\newtheorem{exmp}{Example}[subsection]
 
\begin{document}
\title{\textbf{ Topological surgery in cosmic phenomena}}
\author[1]{Stathis Antoniou}
\author[2]{Louis H.Kauffman}
\author[3]{Sofia Lambropoulou}
\affil[1]{\small{School of Applied Mathematical and Physical Sciences, National Technical University of Athens, Greece}, \textit{santoniou@math.ntua.gr}}
\affil[2]{\small{Department of Mathematics, Statistics, and Computer Science, University of Illinois at Chicago, Chicago, USA, Department of Mechanics and Mathematics\\ Novosibirsk State University\\
              Novosibirsk, Russia}, \newline \textit{kauffman@uic.edu}}
\affil[3]{\small{School of Applied Mathematical and Physical Sciences, National Technical University of Athens, Greece}, \textit{sofia@math.ntua.gr}}
\date{}
\maketitle

\section*{Abstract}
We connect topological changes that can occur in $3$-space via surgery, with black hole formation, the formation of wormholes and new generalizations of these phenomena, including relationships between quantum entanglement and wormhole formation. By considering the initial manifold as the $3$-dimensional spatial section of spacetime, we describe the changes of topology occurring in these processes by determining the resulting $3$-manifold and its fundamental group. As these global changes are induced by local processes, we use the local form of Morse functions to provide an algebraic formulation of their temporal evolution and propose a potential energy function which, in some cases, could give rise to the local forces related to surgery. We further show how this topological perspective gives new insight for natural phenomena exhibiting surgery, in all dimensions, while emphasizing the $3$-dimensional case, which describes cosmic phenomena. This work makes new bridges between topology and natural sciences and creates a platform for exploring geometrical physics.

\let\thefootnote\relax\footnotetext{

{\noindent}\textit{2010 Mathematics Subject Classification}: 57M25, 57R65, 83F05.

{\noindent}\textit{Keywords}: topological surgery, topological procedure, 3-space, 3-sphere, 3-manifold, handle, topology change, visualization,  stereographic projection, Poincaré dodecahedral space, knot theory, natural phenomena, natural processes, dynamics, continuous process, reconnection, Morse theory, mathematical model, Falaco solitons, black holes, wormholes, Einstein-Rosen bridge, cosmic string, fundamental group, knot group, quantum gravity, cosmology, ER=EPR, entanglement.}

\section{Introduction} \label{Intro}

This work is intended for both mathematicians and physicists. For the mathematician, it can be seen as a collection of examples where topology is applied to natural sciences and especially cosmology while, for the physicist, it covers a large background which is not easily available and provides a clear and concise toolbox of algebraic topology and Morse theory important for understanding natural processes and cosmic phenomena.

The mathematics discussed here falls within the topics of low-dimensional topology. A basic aspect of this branch is the use of cobordisms of $1$, $2$ and $3$-manifolds to understand topological and geometric structure. Such cobordisms can be factored into elementary cobordisms called surgeries, which are elementary steps of topology change. This work characterizes the manifolds resulting from such topology change, it  describes the dynamics of those elementary steps and it directly connects them with physical processes in dimensions $1$, $2$ and $3$. We focus on the formation of Falaco solitons, black holes and wormholes, but our topological perspective can be applied to any phenomena exhibiting such topological change. 

These mathematical descriptions further explain some of the large-scale structures and dynamics found in cosmology. Namely, we present a relation between cosmic phenomena, surgery and the $ER=EPR$ hypothesis, see~\cite{ER_EPR_,ER_EPR}. This hypothesis, due to L. Susskind and J. Maldacena, suggests that the connectivity of space is itself a quantum phenomena and is related to quantum entanglement. By using the surgery viewpoint in a context of cobordism we view a wormhole as a cobordism from empty space to the union of the event horizons of two black holes. In the context of topological quantum field theory, this cobordism is associated with a linear mapping from the complex numbers to the tensor product of spaces associated with the two black holes. The image of unity in the complex numbers in this tensor product is a candidate for an entangled state associated with the wormhole. In this way we provide a topological/geometric context for the $ER=EPR$ hypothesis. 

Further, we show that our surgery hypothesis describes the creation of a cosmic string black hole which does not end up with a singular $3$-manifold, thus proposing a potential solution to the singularity problem. Our hypothesis suggests that a cosmic string that collapses would result in a surgery that could be described in terms of this string and an associated framing. In this viewpoint the string collapses, giving rise to a singularity in the sense of Morse functions, and then the process continues with a new cosmic string expanding from the singularity and filling out a new manifold. The result is that a new $3$-dimensional space arises that can be described by framed surgery applied to the partially collapsed cosmic string, with the application of this surgery on the other side of the standard observer's event horizon.

The paper is organized as follows: in Section~\ref{definitions} we present the formal definition of topological surgery for an arbitrary dimension. In Section~\ref{MorseContinuity} we describe the process of topological surgery using Morse theory. This description extends the work done in \cite{SS0,SS2,SS3,SS1,SS4} and fits the way surgery is exhibited in nature. In Sections~\ref{1d0Morse} and~\ref{2dMorse}  we analyze the descriptions in dimensions 1 and 2 and examine how they can be applied to natural processes of  these dimensions. Further, in Sections~\ref{3dMorse} and~\ref{3DtopoAll} we present and visualize the $4$-dimensional process of $3$-dimensional surgery, we analyze the topology of the resulting manifolds and we connect this process with the lower dimensional cases using rotation. We then use these topological tools to describe the formation of wormholes and black holes in Section~\ref{3d0WH}, where we also discuss the cosmological implications of our topological perspective.

\newpage
\tableofcontents
\newpage


\section{The process of topological surgery} \label{definitions}
Topological surgery is a mathematical technique introduced by A.H. Wallace~\cite{Wal} and J.W. Milnor~\cite{Milsur} which creates new manifolds out of known ones in a controlled way. It has been used in the study and classification of manifolds of dimension greater than three while also being an important topological tool in lower dimensions.

Its key idea is to perform an operation of cutting and gluing by using the fact that, if $X, Y$ are manifolds with boundary, the boundary of their product space   $X \times Y$ is given by $\partial (X \times Y)=(\partial X \times Y) \cup (X \times \partial Y)$. This property implies that $\partial (D^{n+1}\times D^{m-n})=(S^{n} \times D^{m-n}) \cup (D^{n+1} \times S^{m-n-1})$ where $D^n$ is the $n$-dimensional disc and $S^n$ is the $n$-dimensional sphere. Topological surgery describes the process which removes an embedding of $S^{n} \times D^{m-n}$ (a $(m-n)$-thickening of $S^n$) and glues back $D^{n+1} \times S^{m-n-1}$ (a $(n+1)$-thickening of $S^{m-n-1}$) along the common boundary $S^{n} \times S^{m-n-1}$. More precisely, the well-known definition of surgery is:

\begin{defn} \label{surgery} \rm An \textbf{$m$-dimensional $n$-surgery} is the topological process of creating a new $m$-manifold $M'$ out of a given $m$-manifold $M$ by removing a framed $n$-embedding $h:S^n\times D^{m-n}\hookrightarrow  M$, and replacing it with $D^{n+1}\times S^{m-n-1}$, using the `gluing' homeomorphism $h$ along the common boundary $S^n\times S^{m-n-1}$. Namely, and denoting surgery by  $\chi$:
$$M' = \chi(M) = \overline{M\setminus h(S^n\times D^{m-n})} \cup_{h|_{S^n\times S^{m-n-1}}} (D^{n+1}\times S^{m-n-1}). $$
The resulting manifold $M'$ may or may not be homeomorphic to $M$. Note that from the definition, we must have $n+1 \leq m$. Also, the horizontal bar in the above formula indicates the topological closure of the set underneath.

{\noindent}Further, \rm the \textbf{dual $m$-dimensional $(m-n-1)$-surgery} on $M'$ removes a dual framed $(m-n-1)$-embedding  $g:D^{n+1}\times S^{m-n-1}\hookrightarrow  M'$ such that $g|_{S^n\times S^{m-n-1}}=h^{-1}|_{S^n\times S^{m-n-1}}$, and replaces it with $S^n\times D^{m-n}$, using the `gluing' homeomorphism $g$ (or $h^{-1}$) along the common boundary $S^n\times S^{m-n-1}$. That is:
$$M = \chi^{-1}(M') = \overline{M'\setminus g(D^{n+1}\times S^{m-n-1})} \cup_{h^{-1}|_{S^n\times S^{m-n-1}}} (S^n\times D^{m-n}). $$
\end{defn}

Surgery is a local process in $M$ (exchanging $S^n\times D^{m-n}$ for $D^{n+1}\times S^{m-n-1}$) which induces a global change (the transition of $M$ to $M'$). For example, in dimension 1, for $m=1$ and $n=0$,  the local process of $1$-dimensional $0$-surgery cuts out two segments $S^{0}\times D^{1}$ from $M$ and glues back the other two segments $D^{1}\times S^{0}$, see Fig.~\ref{1DEx}. Note that this local process is independent of the initial manifold $M$ on which the two segments $S^{0}\times D^{1}$ are embedded.

\smallbreak
\begin{figure}[!ht]
\begin{center}
\includegraphics[width=6cm]{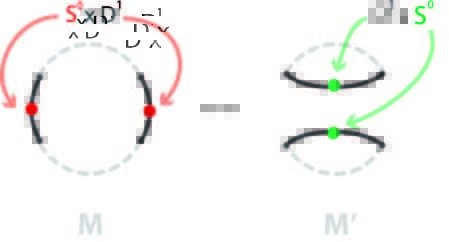}
\caption{$1$-dimensional $0$-surgery}
\label{1DEx}\end{center}
\end{figure}

We will discuss the local process in Section~\ref{local} and the global process in Section~\ref{global}.

\subsection{The local process of surgery} \label{local}
Let us first notice that if we glue together the two $m$-manifolds with boundary  involved in the process of $m$-dimensional $n$-surgery, along their common boundary using the standard mapping $h$, we obtain the $m$-sphere which, in turn, is the boundary of the $(m+1)$-dimensional disc: $(S^n\times D^{m-n}) \cup_{h} (D^{n+1}\times S^{m-n-1}) = (\partial D^{n+1}\times D^{m-n}) \cup_{h} (D^{n+1}\times \partial D^{m-n})  = \partial (D^{n+1}\times D^{m-n})\cong \partial (D^{m+1})=S^{m}$. For example, in dimension 1, $(S^0\times D^{1}) \cup_{h} (D^{1}\times S^{0})=S^{1}$, see Fig.~\ref{HandleExample}~(a).

The $(m+1)$-dimensional disc $D^{m+1} \cong D^{n+1}\times D^{m-n}$ is one dimension higher than the initial manifold $M^{m}$. This extra dimension leaves room for the process of surgery to take place continuously. The disc $D^{m+1}$ considered in its homeomorphic form $D^{n+1}\times D^{m-n}$ is an \textbf{$(m+1)$-dimensional $(n+1)$-handle}. The unique intersection point $D^{n+1} \cap D^{m-n}$ within $D^{n+1}\times D^{m-n}$ is called the \textbf{critical point}. For example, Fig.~\ref{HandleExample}~(b) illustrates the $2$-dimensional $1$-handle $D^1\times D^1$ in which $1$-dimensional $0$-surgery takes place and the corresponding critical point.

The process of surgery is the continuous passage, within the handle $D^{n+1}\times D^{m-n}$, from boundary component $(S^n\times D^{m-n}) \subset  \partial (D^{n+1}\times D^{m-n})$ to its complement $(D^{n+1}\times S^{m-n-1})  \subset  \partial (D^{n+1}\times D^{m-n})$. More precisely, the boundary component $(S^n\times D^{m-n})$ collapse to the critical point $D^{n+1} \cap D^{m-n}$ from which the complement boundary component $(D^{n+1}\times S^{m-n-1})$ emerges. 
 
For the case of $1$-dimensional $0$-surgery, this local process within the handle $D^1\times D^1$ is shown in Fig.~\ref{HandleExample}~(c) where the two segments $S^{0}\times D^{1}$ approach each other, touch at the critical point $D^1 \cap D^1$, where they break, reconnect and become segments $D^{1}\times S^{0}$. 

\smallbreak
\begin{figure}[!ht]
\begin{center}
\includegraphics[width=12cm]{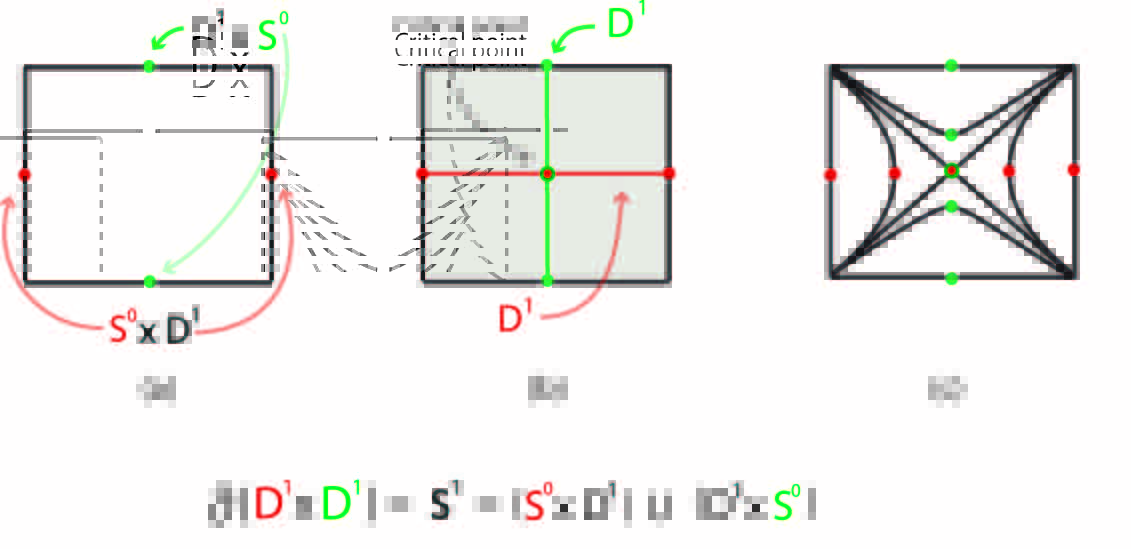}
\caption{The local process of $1$-dimensional $0$-surgery}
\label{HandleExample}\end{center}
\end{figure}

Note that each temporal `slice' of this process is an $m$-dimensional manifold but the evolution of the process requires $m+1$ dimensions in order to be visualized. These local intermediate `slices' will be further analyzed in Section~\ref{MorseContinuity}.

\subsection{The global process of surgery} \label{global}
In order to visualize the global process of surgery which transforms $M$ into $M'$, one also requires $m+1$ dimensions. In fact, surgery on the $m$-manifold $M$ determines a cobordism $(W; M, M')$ called the surgery trace which is made of the temporal `slices' of the global process. More precisely:
 
\begin{defn} \label{cobordism} \rm  An $(m+1)$-dimensional \textbf{cobordism} $(W; M, M')$ is an $(m+1)$-dimensional manifold $W^{m+1}$ with boundary the disjoint union of the closed $m$-manifolds $M,  M'$: $\partial W=M \sqcup M'$. Further, an $(m+1)$-dimensional cobordism $(W; M, M')$ is an $h$-cobordism if the inclusion maps $M \hookrightarrow W$ and $M' \hookrightarrow W$ are homotopy equivalences.
\end{defn}

\begin{defn} \label{trace} \rm  The \textbf{trace} of the surgery removing $S^n\times D^{m-n} \subset  M^{m}$ is the cobordism $(W; M, M')$ obtained by attaching the $(m+1)$-dimensional $(n+1)$-handle $D^{n+1}\times D^{m-n}$ to $ M \times I$ at $S^n\times D^{m-n} \times  {\{1\}} \subset M \times  {\{1\}} $.
\end{defn}

\smallbreak
\begin{figure}[!ht]
\begin{center}
\includegraphics[width=9cm]{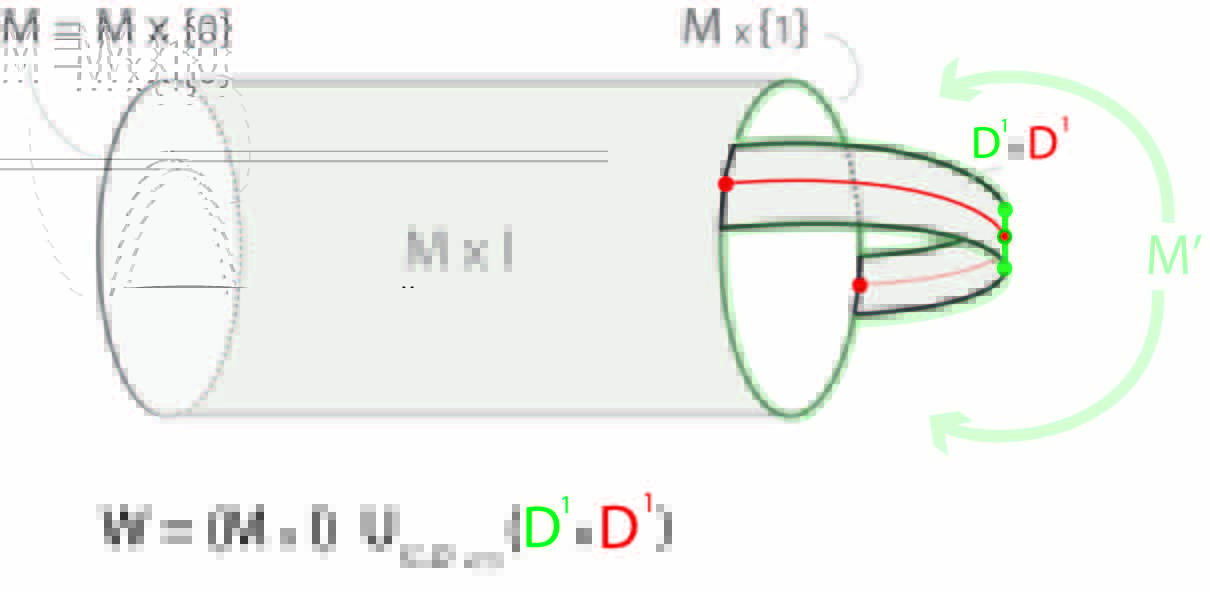}
\caption{The cobordism $(W; M, M')$}
\label{Cobordism}\end{center}
\end{figure}

{\noindent}In fact, two $m$-dimensional manifolds are cobordant if and only if $M'$ can be obtained from $M$ by a finite sequence of surgeries, see \cite{Ra} for details.

The cobordism $(W; M, M')$ of Fig.~\ref{Cobordism} illustrates these definitions for the case of $1$-dimensional $0$-surgery. The local process is part of the global process, hence one can see the handle $D^1\times D^1$ of Fig.~\ref{HandleExample}~(b) in Fig.~\ref{Cobordism}. Further, while not explicitly stated so far, the reader might have already seen from Fig.~\ref{1DEx} that a $1$-dimensional $0$-surgery on $M=S^1$ gives us $M'=S^0 \times S^1$. This is also shown in Fig.~\ref{Cobordism} where we see how the initial manifold $M=S^1 \times  {\{0\}}$ is cobordant with the resulting manifold $M'=S^0 \times S^1 \times  {\{1}\}$, which is shown in Fig.~\ref{Cobordism} in light green. Hence Fig.~\ref{Cobordism} shows $(W; M, M')=(S^1 \times I \cup D^1 \times D^1; S^1, S^0 \times S^1)$.

\smallbreak
\begin{figure}[!ht]
\begin{center}
\includegraphics[width=11cm]{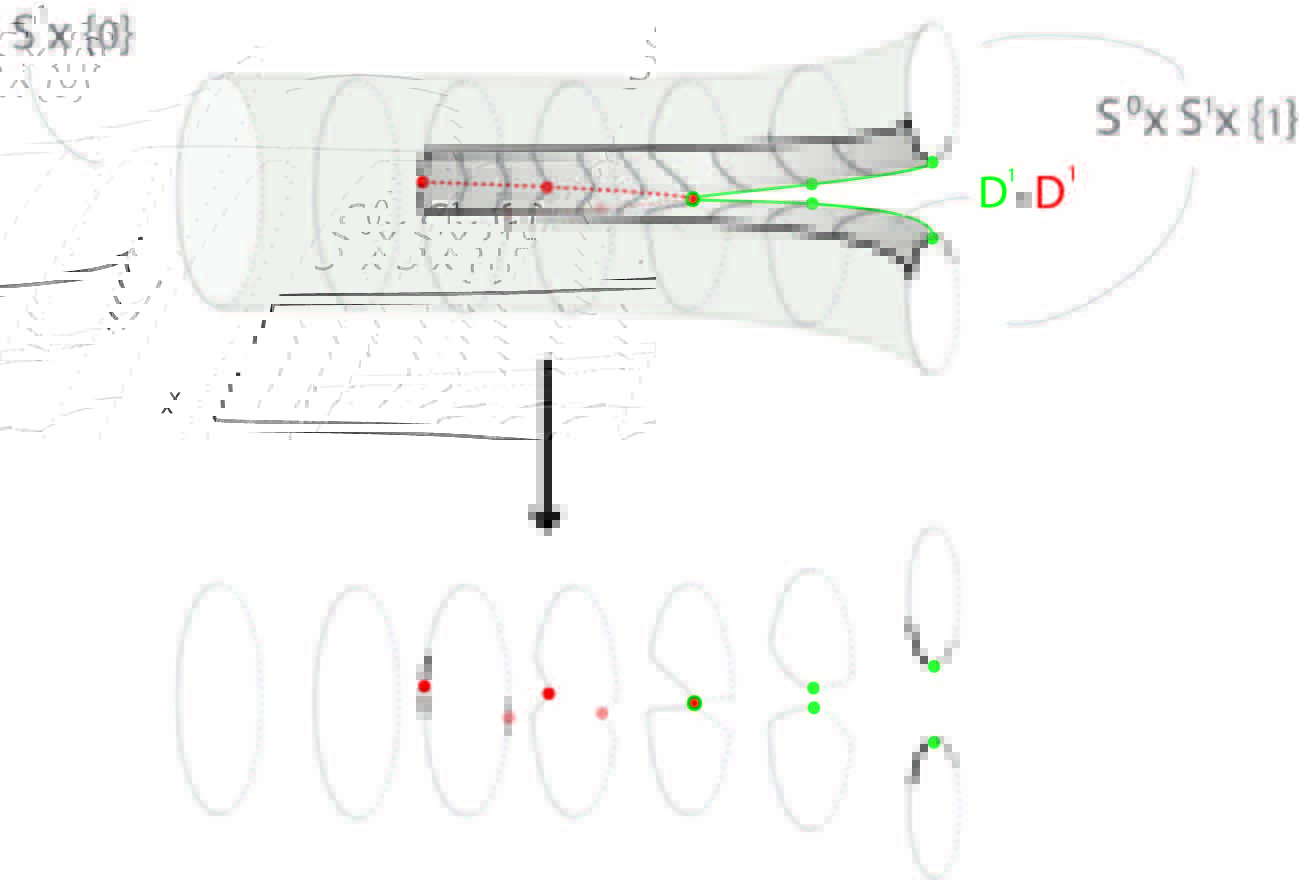}
\caption{The cobordism $(W; S^1, S^0 \times S^1)$ and the process of $1$-dimensional $0$-surgery}
\label{Cobordism2}\end{center}
\end{figure}

However, in order to be able to visualize the temporal `slices' of the global process as perpendicular crossections of the cobordism  of Fig.~\ref{Cobordism}, a homeomorphic representation of $W$ is needed. This is shown in Fig.~\ref{Cobordism2}, where the local process within handle $D^1 \times D^1$ can be seen as part of the the global process of $1$-dimensional $0$-surgery on $S^1$ which, in turn, can be seen as `slices' of $W$.

\section{Morse theory}\label{MorseContinuity}
In this section we will see how Morse theory connects the cobordism of the global process and the $(m+1)$-dimensional $(n+1)$-handle of the local process of surgery. 

\subsection{Definitions}\label{MorsetheoryDef}
We will start by recalling two basic definitions:

\begin{defn} \label{critipoint} \rm  Let $f : M^m \to N^n$ be a differentiable map between two manifolds $M$ and $N$ of dimensions $m$ and $n$ respectively.\\
(i) A \textbf{regular point} of $f$ is a point $x \in M$ where the differential $df(x) : \mathbb{R}^m \to \mathbb{R}^n$ is a linear map of maximal rank, that is, $rank(df (x)) = min(n,m)$.\\
(ii) A \textbf{critical point} of $f$ is a point $x \in M$ which is not regular.\\
(iii) A \textbf{regular value} of $f$ is a point $y \in N$ such that every $x \in f^{−1}(\{y\}) \subseteq M$ is regular (including the empty case $f^{−1}(\{y\}) = \emptyset)$.\\
(iv) A \textbf{critical value} of $f$ is a point $y \in N$ which is not regular. 
\end{defn}

\begin{defn} \label{MorsefuncDef} \rm  Let $f:M^{m} \to \mathbb{R}$ be a differentiable function on an $m$-dimensional manifold.\\
(i) A critical point $x \in M$ of $f$ is \textbf{nondegenerate} if the Hessian matrix $H(x)=(\frac{\partial^2 f}{\partial x_i \partial x_j} )$ is invertible.\\
(ii) The \textbf{index} $Ind(x)$ of a nondegenerate critical point $x \in M$ is the number of negative eigenvalues in $H(x)$, so that with respect to appropriate local coordinates the quadratic term in the Taylor series of $f$ near $x$ is given by 

$$Q(h_1,h_2,...,h_m) = - \sum\limits‎_‎{i=1}^{Ind(x)‎} (h_i)^2 +\sum\limits‎_‎{i=Ind(x)+1}^{m‎} (h_i)^2 \in \mathbb{R}  .$$

\noindent(iii) The function $f$ is \textbf{Morse} if it has only nondegenerate critical points.
\end{defn}

Morse theory studies differentiable manifolds $M$ by considering the critical points of Morse functions  $f:M \to \mathbb{R}$, see \cite{Mil} for details. Among others, Morse theory is used to prove that an $m$-dimensional manifold $M$ can be obtained from $\emptyset$ by successively attaching handles of increasing index $i$:

$$M = \bigcup^{m}_{i=0} (D^{i}\times D^{m-i} \cup D^{i}\times D^{m-i} \cup ... \cup D^{i}\times D^{m-i}) $$ 

\subsection{Connecting Morse theory with the process of surgery}\label{MorsetheoryCon}
The basic connection between Morse theory, cobordisms and topological surgery comes from the following two propositions~\cite{Ra}:

\begin{prop} [\cite{Ra}, Prop. 2.20] \label{MorseCobordisms} \rm  Let $f:W^{m+1} \to I$, where $I$ is the unit interval, be a Morse function on an $(m+1)$-dimensional cobordism $(W; M, M')$ between manifolds $M$ and $M'$ with $$ f^{-1}(\{0\})=M ,\quad f^{-1}(\{1\})=M' $$ and such that all critical points of $f$ are in the interior of $W$.\\
(i) If $f$ has no critical points then $(W; M, M')$ is a trivial $h$-cobordism, with a diffeomorphism $$ (W; M, M') \cong   M \times  (I;{\{0\}},{\{1\}})  $$ which is the identity on $M$.\\
(ii) If $f$ has a single critical point of index $i$ then $W$ is obtained from $ M \times  I$ by attaching an $i$-handle using an embedding $S^{i-1}\times D^{m-i+1} \hookrightarrow  M \times {\{1\}}$, and $(W; M, M')$ is an elementary cobordism of index $i$ with a diffeomorphism $$ (W; M, M') \cong   (M \times  I \cup D^{i}\times D^{m-i+1};M \times {\{0\}},M'). $$
\end{prop}

\begin{prop} [\cite{Ra}, Prop. 2.21] \label{MorseCobordisms2} \rm If an $(m+1)$-dimensional manifold with boundary $(W,\partial W)$ is obtained from $(W_0,\partial W_0)$ by attaching an $i$-handle $$ W = W_0 \cup_{S^{i-1}\times D^{m-i+1}} D^{i}\times D^{m-i+1} $$ then $\partial W$ is obtained from $\partial W_0$ by an $m$-dimensional $(i-1)$-surgery $$ \partial W =  \overline{\partial W_0\setminus (S^{i-1}\times D^{m-i+1})} \cup_{{S^{i-1}\times S^{m-i}}} D^{i}\times S^{m-i}. $$
\end{prop}

The proof of Proposition~\ref{MorseCobordisms} (i) and (ii) can be found in \cite{Mil} and \cite{Ra} respectively, while for Proposition~\ref{MorseCobordisms2} the reader is referred to \cite{Mil}. For example, in the case of $1$-dimensional $0$-surgery, since $m=1$ and $n=0$, the single critical point of index $i=1$ mentioned in Proposition~\ref{MorseCobordisms}~(ii) is $D^1 \cap D^1$ which is in the interior of the handle $D^{1}\times D^{1}$, recall Fig.~\ref{HandleExample}~(b). The corresponding cobordism $W$ referred to in Proposition~\ref{MorseCobordisms2} and shown in Figs.~\ref{Cobordism} and ~\ref{Cobordism2}, is obtained by attaching the handle $D^{1}\times D^{1}$ to $W_0=M \times I = S^1 \times I$ while $\partial W= S^1 \sqcup S^0 \times S^1$ is obtained by a $1$-dimensional $0$-surgery on $\partial W_0 = S^1 \sqcup S^1$.
 
We will now present a theorem and a lemma from~\cite{Ra} which will be used to study the temporal evolution of topological surgery in the following sections. 

\begin{thm} [\cite{Ra}, Thm 2.14] \label{EveryMorse} \rm Every $m$-dimensional manifold $M^m$ admits a Morse function $f:M \to \mathbb{R}$. 
\end{thm}

{\noindent}See \cite{Mil} for the proof.

\begin{lem} [\cite{Ra}, Lemma 2.19]  \label{LocalMorse} \rm For any $0 \leq i \leq m+1$ the Morse function
$$ f:D^{m+1} \to \mathbb{R}; \quad (x_1,x_2,...,x_{m+1}) \mapsto -‎‎\sum\limits‎_‎{j=1}^{i‎} x^2_{j} +‎‎\sum\limits‎_‎{j=i+1}^{m+1‎} x^2_{j}   $$
{\noindent}has a unique interior point $0 \in D^{m+1}$, which is of index $i$. The $(m+1)$-dimensional manifolds with boundary, defined for $0 < \epsilon < 1$ by
$$ W_{- \epsilon}=f^{-1}(-\infty, -\epsilon],\quad   W_{\epsilon}=f^{-1}(-\infty,\epsilon] 	$$
{\noindent}are such that  $W_{\epsilon}$ is obtained from  $W_{- \epsilon}$ by attaching an $i$-handle: $$ W_{\epsilon}=   W_{-\epsilon}  \cup D^{i}\times D^{m-i+1} .	$$
\end{lem}

\smallbreak
\begin{figure}[!ht]
\begin{center}
\includegraphics[width=8cm]{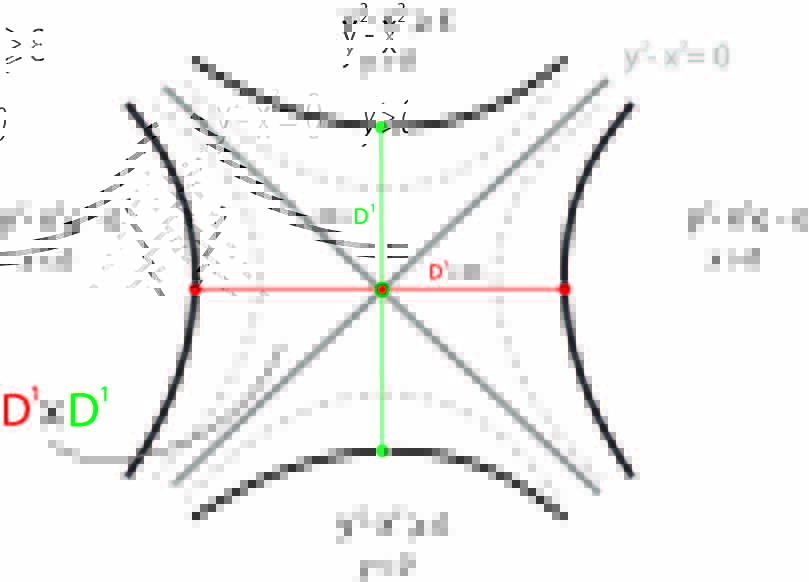}
\caption{The local form of a Morse function for $m=i=1$}
\label{Morse1Example}\end{center}
\end{figure}

For example, the case $m=i=1$ of the lemma is shown in Fig.~\ref{Morse1Example} where: $$ f:D^{2} \to \mathbb{R}; \quad (x,y) \mapsto -x^2+y^2     $$
$$W_{-\epsilon}= {\{(x,y) \in D^2 \mid f(x,y) \leq -\epsilon \}},$$
$$D^{1}\times D^{1}= {\{   (x,y) \in D^2 \mid -\epsilon \leq  f(x,y) \leq \epsilon\}}, \quad  \text{the attached handle},$$
$$W_{\epsilon} = {\{   (x,y) \in D^2 \mid  f(x,y) \leq \epsilon\}} = W_{-\epsilon}  \cup D^{1}\times D^{1}.$$

\smallbreak
Lemma~\ref{LocalMorse} connects Morse functions with both the cobordism of the global process and the handle of the local process. Moreover, the local process of $m$-dimensional $(i-1)$-surgery within the $(m+1)$-dimensional handle, recall Fig.~\ref{HandleExample}~(c), can be parametrized by $\epsilon$. Indeed, comparing  Fig.~\ref{Morse1Example} with Fig.~\ref{HandleExample}~(c), the values $\epsilon<0$ correspond to the two segments $S^0 \times D^1$ approaching each other, $\epsilon=0$ corresponds to the straighten segments which intersect at the critical point $D^1 \cap D^1$, while the values $\epsilon>0$ correspond to the reconnected segments $D^1 \times S^0$.

\section{Local dynamics of $1$-dimensional surgery }\label{1d0Morse}
In this section, we will see how the local form of a Morse function can be used to describe the temporal evolution of natural phenomena exhibiting $1$-dimensional $0$-surgery. Moreover, for phenomena exhibiting this type of surgery, we propose the negative gradient of the local form of a Morse function as a potential energy function giving rise to the local forces related to surgery.

\subsection{Temporal evolution}\label{Time1D}
As mentioned in the end of Section~\ref{MorsetheoryCon}, the Morse function for the case $m=i=1$ of Lemma~\ref{LocalMorse} is: $f(x,y)=-x^2+y^2$. Its plotting is shown in Fig.~\ref{Temp1Example}~(1). Now, parameter $\epsilon$ can be considered as time so we shall denote it by $t$. So, we can describe the process of surgery by varying parameter $t$ of the level curves $-x^2+y^2=t$, illustrated in Fig.~\ref{Temp1Example}~(2), thus providing a continuous analogue of the process illustrated in Fig.~\ref{Morse1Example}. For $-1<t<0$, these hyperbolas are shaded in red. As $t$ gets close to $0$, the two branches of the hyperbolas get close to one another and their color whitens. At $t=0$ the degenerated hyperbola $-x^2+y^2 =0$ consist in two straight white segments along which the reconnection takes place. Finally, as $t$ starts taking positive values in the range $0<t<1$, the two new branches of the hyperbolas start turning to green. 

\smallbreak
\begin{figure}[!ht]
\begin{center}
\includegraphics[width=10cm]{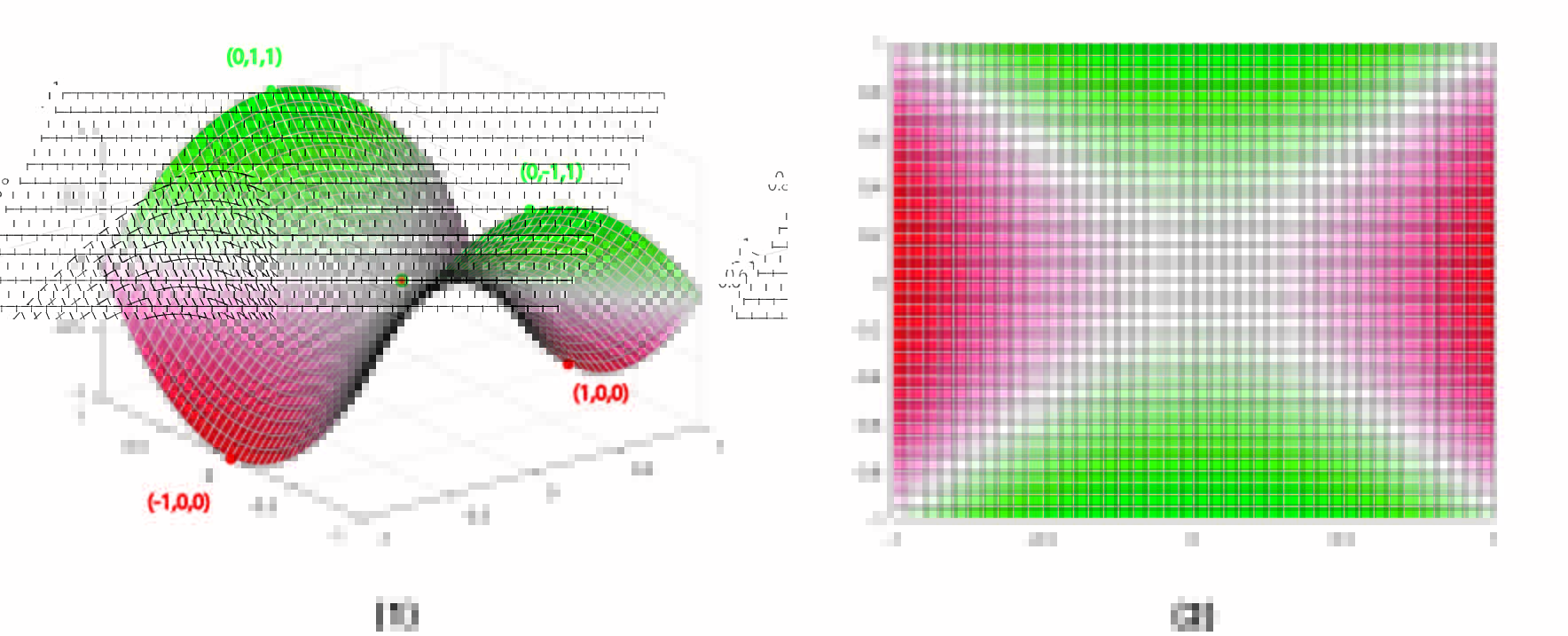}
\caption{\textbf{(1)} $f(x,y)=-x^2+y^2$   \textbf{(2)} The level curves $-x^2+y^2=t; -1<t<1$   }
\label{Temp1Example}\end{center}
\end{figure}

\subsection{Gradient description}\label{Gradient1D}
The gradient vector field $\nabla f=(-2x,2y)$, which is perpendicular to the level curves $-x^2+y^2=t$ and points in the direction of the greatest rate of increase of $f$, is shown in Fig.~\ref{TS_CP_1D_Gradient}. The flow of $S^0$ which is composed of the two red points $(-1,0) \sqcup (1,0)$ in  Fig.~\ref{TS_CP_1D_Gradient}, follows the red vectors along the $x$-axis towards the critical point $(0,0)$. After collapsing to the critical point, two new green points emerge, following the green vectors along the $y$-axis towards the $S^0$ composed of the green points $(0,1) \sqcup (0,-1)$. In other words, the process of $1$-dimensional $0$-surgery can be viewed as the collapsing of the core $S^0$ of segments $S^0 \times D^1$ to the critical point from which the core $S^0$ of segments $D^1 \times S^0$ uncollapses. The two directions followed by the cores are the two perpendicular segments $D^1$ that make up the $2$-dimensional $1$-handle $D^1\times D^1$. These segments were shown in red and green in Fig.~\ref{Morse1Example} and the same color coding has been used to show the vectors acting along them in Fig.~\ref{TS_CP_1D_Gradient}.
 
Moreover, the gradient is closely related to the notion of force. For example, an object starting from a high place (thus having high potential energy) and rolling down to a lower place (of lower potential energy) under the influence of gravity will follow the exact opposite direction of the gradient vectors. Looking at Fig.~\ref{Temp1Example}~(1) and letting two small objects fall from the two highest points $(0,1,1)$ and $(0,-1,1)$, these objects will meet at $(0,0,0)$ and fall down to the two lowest points $(1,0,0)$ and $(-1,0,0)$. Their path projected in $2$-dimensions corresponds to the time-reversed process of Fig.~\ref{TS_CP_1D_Gradient}: two green points $S^0$ collapsing to the critical point from which the two red points $S^0$ emerge. We can think of the Morse function as describing the height (hence related to the potential energy) and of the objects as rolling down the hills described by the Morse function. The gravitational force and the motion of the objects are both in the direction of the negative gradient of the Morse function, perpendicular to its level curves.   

More generally, if the forces acting on a particle are conservative, they are derivable from a scalar potential energy function $V$ as $\vec{F}=-(\nabla V)$. Hence, for phenomena exhibiting such surgery, one can take the local form of the corresponding Morse function multiplied by $-1$: $(-1)*f$ as a potential energy function giving rise to the local forces related to surgery.

\smallbreak
\begin{figure}[!ht]
\begin{center}
\includegraphics[width=7cm]{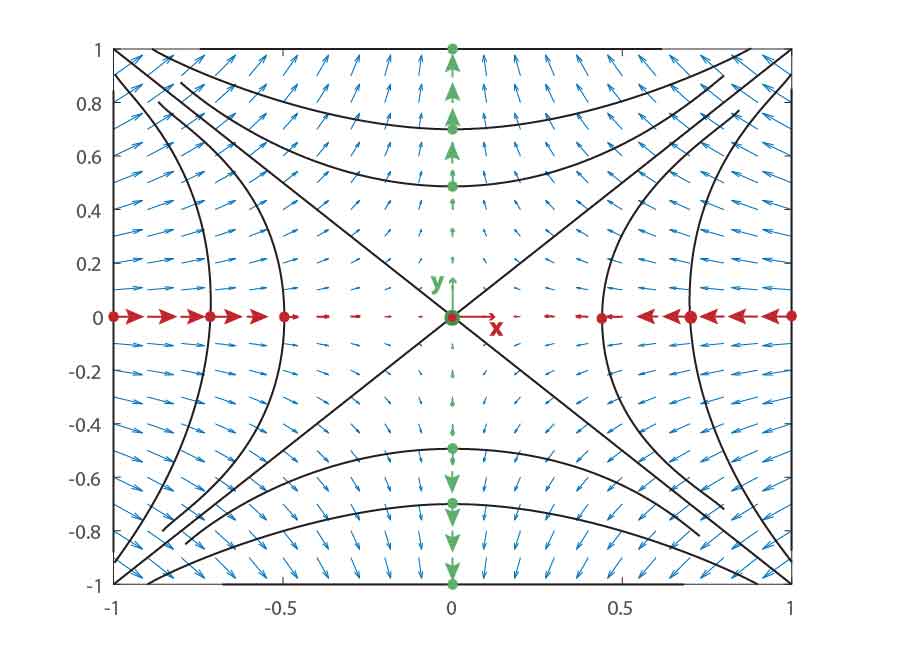}
\caption{The gradient $  \nabla f=(-2x,2y)$}
\label{TS_CP_1D_Gradient}\end{center}
\end{figure}

\subsection{$1$-dimensional phenomena}\label{NatPhe1D}
The above analysis provides a way to describe natural phenomena exhibiting $1$-dimensional $0$-surgery. Such phenomena occur in both micro and macro scales. It can be seen for example during magnetic reconnection (the phenomenon whereby cosmic magnetic field lines from different magnetic domains are spliced to one another,  changing their pattern of conductivity with respect to the sources), during meiosis (when new combinations of genes are produced) and in site-specific DNA recombination (whereby nature alters the genetic code of an organism). These phenomena and their relation to topological surgery have been detailed in \cite{SS0} where we pin down the forces that are present in each process. 
 
Note that this analysis also gives us an algebraic description of the process. More precisely, we can now use equation $-x^2+y^2=t, -1<t<1$, to describe the continuous way the $1$-dimensional splicing and reconnection occurs. Moreover, it generalizes the notion of forces to the negative gradient of the local form of the corresponding Morse function. As a result, if we view the gradient vectors of Fig.~\ref{TS_CP_1D_Gradient} as forces, these act not only on the cores $S^0$  but on the whole segments $S^0 \times D^1$ and $D^1 \times S^0$. Moreover, while the collapse of the core $S^0$ of the initial segments $S^0 \times D^1$ is the effect of attracting forces, we now pin down that the uncollapsing of the core $S^0$ of the final segments $D^1 \times S^0$ is the result of repelling forces. Note that we will keep this color coding throughout the paper. Namely vectors exhibiting attraction and repulsion will be shown in red and green, respectively.

\section{Local dynamics of $2$-dimensional surgery}\label{2dMorse}
In this section, we will see how the local form of a Morse function can be used to describe natural phenomena exhibiting $2$-dimensional surgery. Moreover, for phenomena exhibiting this type of surgery, we propose the negative gradient of the local form of a Morse function as a potential energy function giving rise to the local forces related to surgery.

\subsection{Types of $2$-dimensional surgery}\label{Types2D}
From Definition~\ref{surgery}, we know that there are two types of $2$-dimensional surgery. Namely, starting with a $2$-manifold $M$, one can have $m=2$ and $n=0$ or $m=2$ and $n=1$. The first possibility is the \textit{ $2$-dimensional $0$-surgery} which removes two discs $S^0\times D^2$ from $M$ and replaces them by a cylinder $D^1\times S^1$. This cylinder gets attached along the common boundary $S^0\times S^1$ comprising two copies of $S^1$. For example, if $M=S^2$ the above operation changes its homeomorphism type from the 2-sphere to that of the torus, see Fig.~\ref{2Dex}~(a). The other possibility is the \textit{$2$-dimensional $1$-surgery} where a cylinder (or equivalently an annulus) $S^1 \times D^1$ is removed from $M$ and is replaced by two discs  $D^2 \times S^0$ attached along the common boundary $S^1 \times S^0$. For example, if $M=S^2$ the result is two copies of  $S^2$,  see Fig.~\ref{2Dex}~(b). 

Note now that from Definition~\ref{surgery}, a dual $2$-dimensional $0$-surgery is a $2$-dimensional $1$-surgery and vice versa. Hence,  Fig.~\ref{2Dex}~(a) shows that a $2$-dimensional $0$-surgery on a sphere is the reverse process of a $2$-dimensional $1$-surgery on a torus, while Fig.~\ref{2Dex}~(b) shows that $2$-dimensional $1$-surgery on a sphere is the reverse process of a $2$-dimensional $0$-surgery on two spheres. In the figure, the symbol $\longleftrightarrow $ indicates surgeries from left to right and their corresponding dual surgeries from right to left. 

\smallbreak
\begin{figure}[!ht]
\begin{center}
\includegraphics[width=10cm]{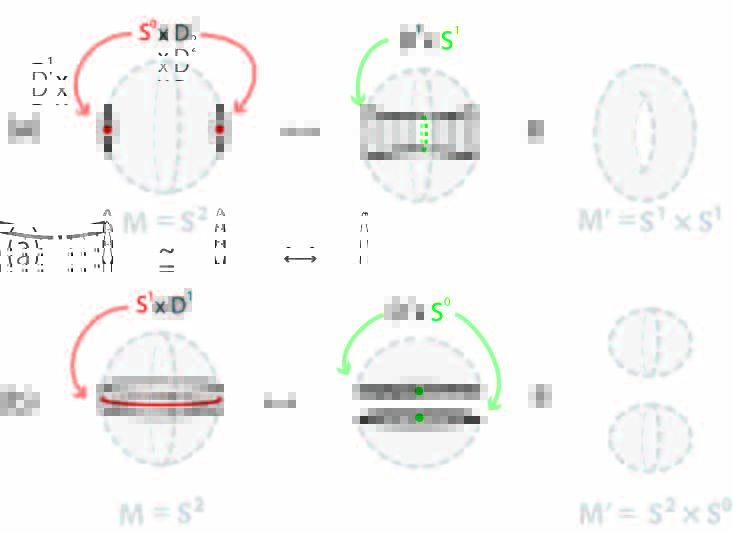}
\caption{ \textbf{(a)} 2-dimensional 0-surgery on $M=S^2$ \textbf{(b)} 2-dimensional 1-surgery on $M=S^2$.}
\label{2Dex}\end{center}
\end{figure}

\subsection{Temporal evolution}\label{Time2D}
Consider now the Morse function of Lemma~\ref{LocalMorse} for the case $m=2$ and $i=1$, namely: $$g:D^{3} \to \mathbb{R};\quad  (x,y,z) \mapsto -x^2+y^2+z^2$$ Applying the line of thought presented in Section~\ref{1d0Morse} one dimension higher, the local process of $2$-dimensional $0$-surgery happens inside handle $D^1\times D^2$ and can be described by varying parameter $t$ of the level surfaces $-x^2+y^2+z^2=t$. For $-1<t<0$, these are two-sheet hyperboloids.  In Fig.~\ref{TS_CP_2D_Gradient_Black_2}, one of these two-sheets hyperboloids is shown intersecting with the $x$-axis at the two antipodal red points. As $t$ gets close to $0$, the two-sheets of the hyperboloids get close to one another. At $t=0$ the two sheets merge and become the conical surface $-x^2+y^2+z^2=0$ centered at $(0,0,0)$, see the red/green point of Fig.~\ref{TS_CP_2D_Gradient_Black_2}, from which, as $t$ takes positive values in the range $0<t<1$, the new one-sheet hyperboloids emerge. One of these one-sheet hyperboloids is shown in Fig.~\ref{TS_CP_2D_Gradient_Black_2} where its intersection with the $(y,z)$-plane is the circle shown in green. Similarly, for $2$-dimensional $1$-surgery, one could consider the Morse function of Lemma~\ref{LocalMorse} for the case $m=2$ and $i=2$. However, one can simply reverse the time of the Morse function of $2$-dimensional $0$-surgery, $g(x,y,z)=-x^2+y^2+z^2$, to obtain the level surfaces 
$-x^2+y^2+z^2=(-t)$, which describe the local process of $2$-dimensional $1$-surgery. In Fig.~\ref{TS_CP_2D_Gradient_Black_2}, this process starts from a one-sheet hyperboloid which is continuously transformed to the ending two-sheets hyperboloid.

\subsection{Gradient description}\label{Gradient2D}
The gradient vector field $\nabla g=(-2x,2y,2z)$ which is perpendicular to the level surfaces $-x^2+y^2+z^2=t$ describing $2$-dimensional $0$-surgery is shown in Fig.~\ref{TS_CP_2D_Gradient_Black_2}. The flow of $S^0$, which is composed of the two red points $(-1,0,0) \sqcup (1,0,0)$ in  Fig.~\ref{TS_CP_2D_Gradient_Black_2}, follows the red vectors along the $x$-axis towards the critical point $(0,0,0)$. After collapsing, the new green circle $S^1$ emerges along the $(y,z)$-plane as a result of the green vectors. In other words, the process of $2$-dimensional $0$-surgery can be seen as the collapsing of the core $S^0$ of discs $S^0 \times D^2$ to the critical point from which the core $S^1$ of cylinder $D^1 \times S^1$ uncollapses. The two directions followed by the cores are along the (red) segment $D^1$ on  the $(x)$-axis and the (green) disc $D^2$ on the $(y,z)$-plane that make up the $3$-dimensional $1$-handle $D^1\times D^2$ in $\mathbb{R}^3$. If we view the gradient vectors of Fig.~\ref{TS_CP_2D_Gradient_Black_2} as forces, the attracting forces acting on the core $S^0$  are fleshed out to the whole $S^0 \times D^2$ until the critical point is reached after which, the repelling forces uncollapsing the core $S^1$ are fleshed out to the cylinder $D^1 \times S^1$. 

Taking the one dimension higher analogue of $1$-dimensional $0$-surgery presented in Section~\ref{Gradient1D}, if the forces acting on a particle are conservative, then the local form of the Morse function $g$ can be used as a potential energy function giving rise to the local forces related to $2$-dimensional $0$-surgery: $\vec{F}=-(\nabla V)=\nabla g$. The gradient vector field perpendicular to the level surfaces describing $2$-dimensional $1$-surgery can be described analogously.

\smallbreak
\begin{figure}[!ht]
\begin{center}
\includegraphics[width=9cm]{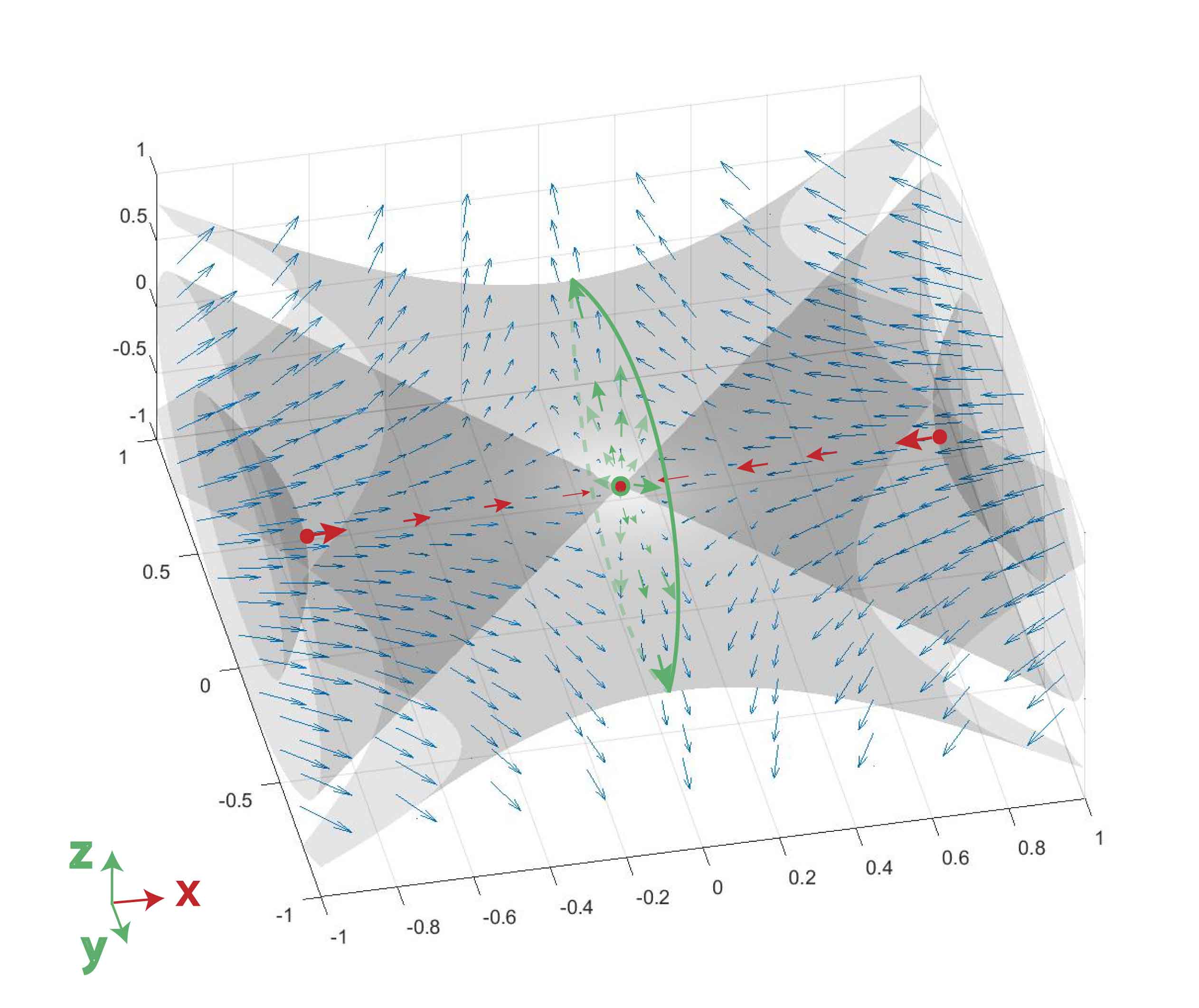}
\caption{The gradient $  \nabla g=(-2x,2y,2z)$}
\label{TS_CP_2D_Gradient_Black_2}\end{center}
\end{figure}

\subsection{$2$-dimensional phenomena}\label{NatPhe2D}
The above analysis provides a way to describe natural phenomena exhibiting $2$-dimensional surgery, that is, phenomena where $2$-dimensional merging and recoupling occurs. Roughly speaking,  $2$-dimensional $0$-surgery can be seen in phenomena where a cylinder is created, while $2$-dimensional $1$-surgery can be seen in phenomena where a cylinder is collapsed.

Examples of $2$-dimensional $0$-surgery comprise the formation of tornadoes, drop coalescence  (the phenomenon where two dispersed drops merge into one), gene transfer in bacteria (where the donor cell produces a connecting tube called a `pilus' which attaches to the recipient cell) and the formation of Falaco solitons, see Fig.~\ref{2D_Falacohomeo2}~(1). Each Falaco soliton consists of a pair of locally unstable but globally stabilized contra-rotating identations in the water-air surface of a swimming pool, see \cite{Ki} for details. The cylinder that is being created can take various forms. For example, it is a tubular vortex of air in the case of tornadoes, a pilus joining the genes during bacterial gene transfer and transverse torsional waves in the case of Falaco solitons, see Fig.~\ref{2D_Falacohomeo2}~(2). 

On the other hand, $2$-dimensional $1$-surgery can be seen during soap bubble splitting (where a soap bubble splits into two smaller bubbles),  when the tension applied on metal specimens by tensile forces results in the phenomena of necking and then fracture and in the biological process of mitosis (where a cell splits into two new cells). These phenomena are characterized by a `necking' occurring in a cylinder $D^1\times S^1$, which degenerates into a point and finally tears apart creating two discs $S^0\times D^2$. The cylinder that is about to collapse can be embedded, for example, in the region of the bubble's  surface where splitting occurs, on the region of metal specimens where necking and fracture occurs, or on the equator of the cell which is about to undergo a mitotic process. These phenomena and their relation to topological surgery have been detailed in \cite{SS0} where we pin down the forces that are present is these processes.

\smallbreak
With this analysis, the local form of the Morse function $g$ can be used to describe algebraically the processes of $2$-dimensional surgeries. Moreover, our analysis provides a novel description of these processes if the gradient vectors $\nabla g$ of Fig.~\ref{TS_CP_2D_Gradient_Black_2} are viewed as forces.

\subsection{Non-trivial embeddings}\label{2Dtwist}
In this section, based on the phenomenon of Falaco solitons creation, we will  examine the local process of topological surgery for non-trivial embeddings (recall Definition~\ref{surgery}).

Let us start by pointing out that, for phenomena exhibiting $2$-dimensional $0$-surgery, the various forms of the attached cylinder are homeomorphic representations of the cylinder $D^1 \times S^1$ shown in Fig.~\ref{TS_CP_2D_Gradient_Black_2}. For example, during the formation of Falaco solitons, the cylinder (and the whole $3$-dimensional $1$-handle $D^1\times D^2$ inside which the local process takes place) is bended and twisted, see Fig.~\ref{2D_Falacohomeo2}~(2). Note that the singular thread shown in Fig.~\ref{2D_Falacohomeo2}~(1) is the segment $D^1$ joining the core $S^0$, which in this case comprises the two central points of the Falaco solitons.

\smallbreak
\begin{figure}[!ht]
\begin{center}
\includegraphics[width=11cm]{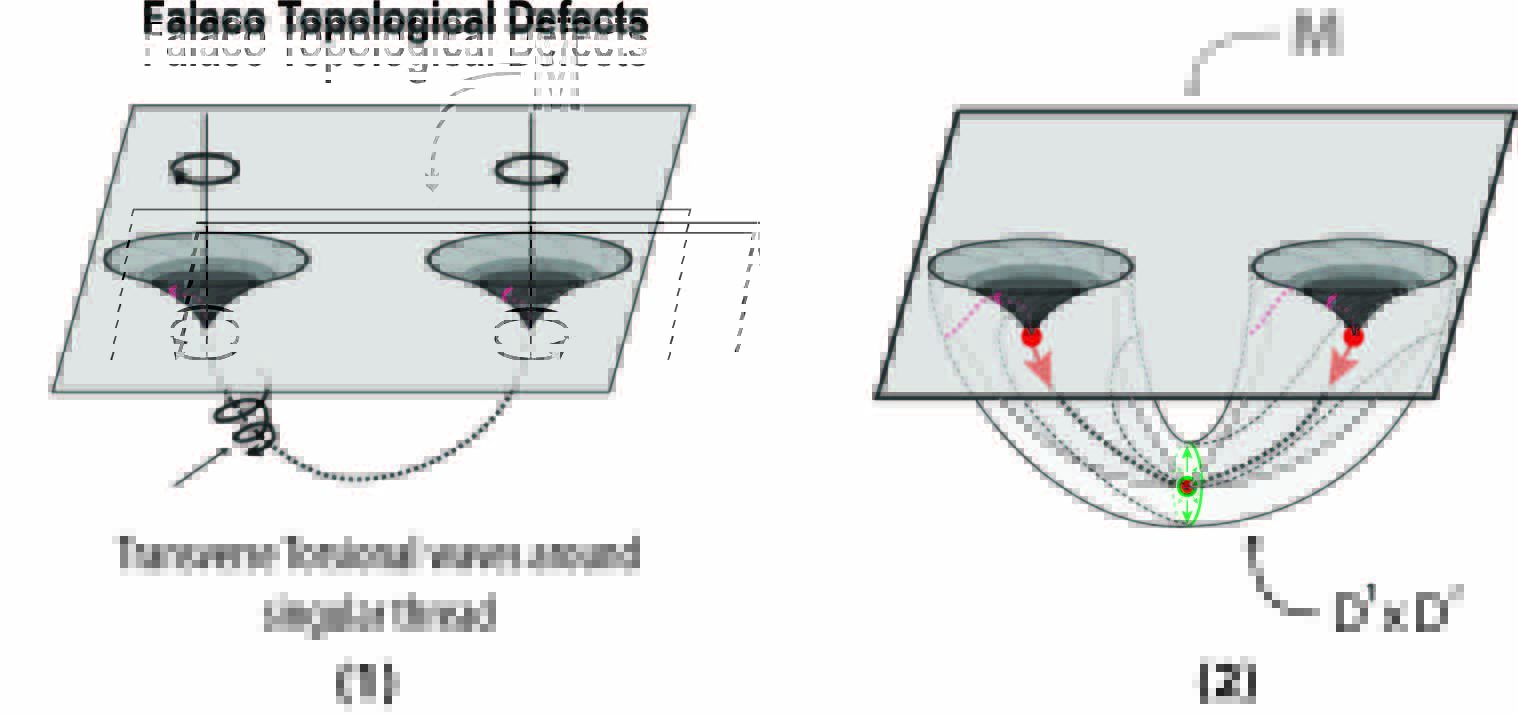}
\caption{\textbf{(1)} Falaco solitons \textbf{(2)} Homeomorphic representation of handle $D^1\times D^2$}
\label{2D_Falacohomeo2}\end{center}
\end{figure}

Up to now, when referring to the embedding $h$ of Definition~\ref{surgery}, we have assumed that the standard (or trivial) embedding, which we will denote by $h_s$, was used. For example, the process of $2$-dimensional $0$-surgery shown in Fig.~\ref{TS_CP_2D_Gradient_Black_2} does not involve twisting. The same process is shown in Fig.~\ref{2D_Twist_Instances}~(1) where the key instances have been discretized for the purpose of clarity. However, many phenomena, including the formation of Falaco solitons, correspond to a non-trivial embedding, say $h_t$, which involves twisting. The two indentations of Fig.~\ref{2D_Falacohomeo2}~(1) can be seen as the first instance of the local process of $2$-dimensional $0$-surgery,  which can be described by an embedding $h_t(S^0\times D^2)$ twisting the two discs. An example of such an embedding can be seen in the leftmost instance of Fig.~\ref{2D_Twist_Instances}~(2). The cylindrical vortex $D^1 \times S^1$ made from the propagation of the torsional waves around the singular thread seen in Fig.~\ref{2D_Falacohomeo2}~(2) can be considered as the final instance of the process, corresponding to the rightmost instance of Fig.~\ref{2D_Twist_Instances}~(2).

The difference between the two embeddings $h=h_s$ and $h=h_t$ is shown in Fig.~\ref{2D_Twist_Instances}~(1) and (2) respectively. More precisely, if we consider counterclockwise rotations as positive, embedding $h_t$ rotates the two initial discs by $-3\pi/4$ and $3\pi/4$ respectively, see the passage from the leftmost instance of Fig.~\ref{2D_Twist_Instances}~(1) to the leftmost instance of Fig.~\ref{2D_Twist_Instances}~(2). If we define the homeomorphisms $\omega_1,\omega_2:D^{2}\to D^{2}$ to be rotations by $-3\pi/4$ and $3\pi/4$ respectively, then $h_t$ is defined as the composition $h_t:S^{0}\times D^2 \xrightarrow{\omega_1 \amalg  \omega_2}  S^0\times D^{2}  \xrightarrow{h} M $. This rotation induces the twisting $g_t$ of angle $-3\pi/2$ of the final cylinder, see the rightmost instances of Fig.~\ref{2D_Twist_Instances}~(1) and~(2).

\smallbreak
\begin{figure}[!ht]
\begin{center}
\includegraphics[width=10cm]{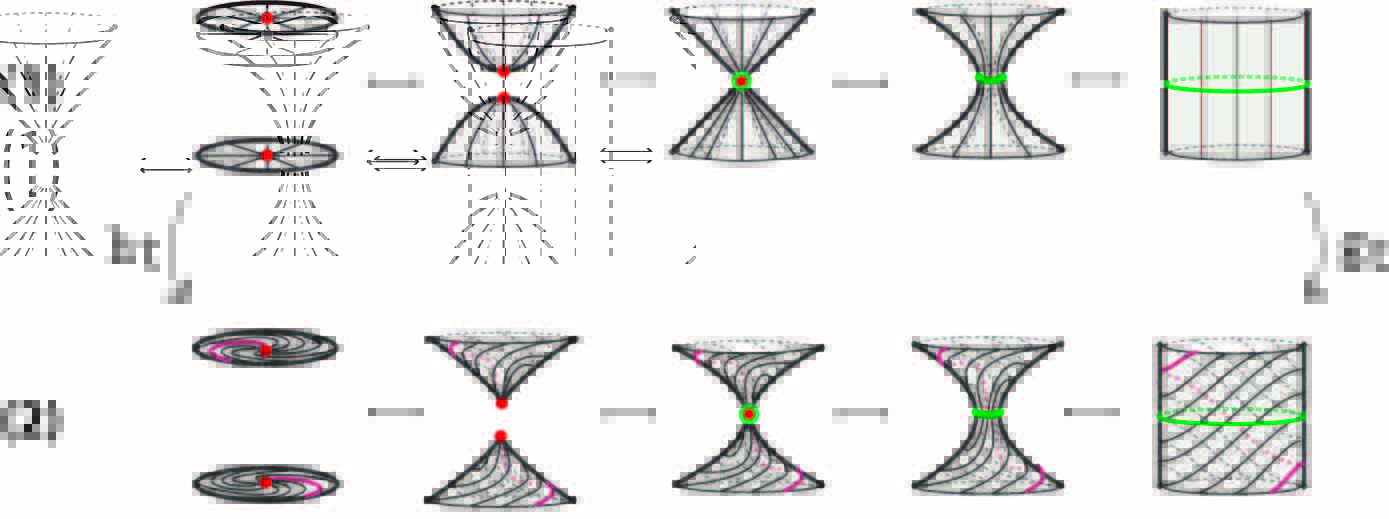}
\caption{$2$-dimensional surgery with \textbf{(1)} the standard embedding  $h=h_s$  \textbf{(2)} the non-trivial embedding $h=h_t$}
\label{2D_Twist_Instances}\end{center}
\end{figure}

When the topological thread is cut, for example when the Falaco solitons hit an obstacle perpendicular to their displacement, the cylindrical vortex tears apart and slowly degenerates to the two discs until they both stop spinning and vanish. Note that since the dissipation of Falaco solitons  is slower than their creation, the intermediate instances of this process can be visualized in real time in experiments such as \cite{PhyG}. This reverse process is the passage from Fig.~\ref{2D_Falacohomeo2}~(2) to Fig.~\ref{2D_Falacohomeo2}~(1) and corresponds to the $2$-dimensional $1$-surgery shown from right to left in Fig.~\ref{2D_Twist_Instances}~(2). 

In this case, the initial cylinder is twisted. In our example, homemorphism $g_t$ rotates the top and bottom of the cylinder by $-3\pi/4$ and $3\pi/4$ respectively, see the passage from the rightmost instance of Fig.~\ref{2D_Twist_Instances}~(1) to the rightmost instance of Fig.~\ref{2D_Twist_Instances}~(2). This rotation induces the twisting of the two final discs, as in the leftmost instance of Fig.~\ref{2D_Twist_Instances}~(2).  

Let us finally conclude that the homeomorphisms $h_t$ and $g_t$ as illustrated in Fig.~\ref{2D_Twist_Instances} provide a better description of the creation (or dissipation) of Falaco solitons and, more generally, of phenomena involving twisting with `drilling' (or twisting with `necking').

\section{Local dynamics of $3$-dimensional surgery}\label{3dMorse}
In this section, we describe locally $3$-dimensional surgery using the local form of a Morse function, we propose ways to visualize this $4$-dimensional process and we  connect the processes of surgery in dimensions $1,2$ and $3$ via rotation. This section together with the next one set the ground for the analysis of cosmic phenomena exhibiting $3$-dimensional surgery, which will be discussed in Section~\ref{3d0WH}.

\subsection{Types of $3$-dimensional surgery}\label{Types3D}
From Definition~\ref{surgery}, we know that there are three types of surgery in dimension $3$. Namely, starting with a 3-manifold $M$, for $m=3$ and $n=0$, we have the \textit{ $3$-dimensional $0$-surgery}, whereby two 3-balls $S^0\times D^3$ are removed from  $M$ and are replaced in the closure of the remaining manifold by a thickened sphere $D^1\times S^2$:

\begin{samepage} 
 \begin{center}
$\chi(M) = \overline{M\setminus h(S^0\times D^{3})} \cup_{h} (D^{1}\times S^{2})$
\end{center}  
\end{samepage} 

Next, for $m=3$ and $n=2$, we have the \textit{ $3$-dimensional $2$-surgery}, which is the reverse (dual) process of $3$-dimensional $0$-surgery. 

Finally, for $m=3$ and $n=1$, we have the \textit{ $3$-dimensional $1$-surgery}, whereby a solid torus $S^1\times D^2$ is removed from $M$ and is replaced by another solid torus $D^2\times S^1$ (with the factors now reversed) via a homeomorphism $h$ of the common boundary:

\begin{samepage} 
 \begin{center}
$\chi(M) = \overline{M\setminus h(S^1\times D^{2})} \cup_{h} (D^{2}\times S^{1})$
\end{center}  
\end{samepage} 

This type of surgery is clearly self-dual.

\subsection{Temporal evolution, gradient and core description}\label{3DTemp}
Consider the Morse function of Lemma~\ref{LocalMorse} for the case $m=3$ and $i=1$: $$f:D^{4} \to \mathbb{R}; \quad (x,y,z,w) \mapsto -x^2+y^2+z^2+w^2$$ Applying the line of thought presented in Sections~\ref{1d0Morse} and~\ref{2dMorse}, the local process of \textit{$3$-dimensional $0$-surgery} happens inside the $4$-dimensional handle  $D^1\times D^3$ and can be described by varying parameter $t$ of the level hypersurfaces $-x^2+y^2+z^2+w^2=t$.  These hypersufaces and the perpendicular gradient vector field $\nabla f=(-2x,2y,2z,2w)$ require four  dimensions in order to be visualized. However, we can describe and visualize the behaviors of the cores and the gradient along their direction of movement. We will refer to this visualization as the `core view' of $3$-dimensional $0$-surgery. The process starts with the core $S^0$ of $S^0\times D^3$, see the two red points in the leftmost instance of Fig.~\ref{3D_Cores}~(1). These two points are attracted towards $(0,0,0,0)$ under the influence of the gradient $\nabla f=(-2x,2y,2z,2w)$ which is negative along the horizontal axis $x$. Along $x$, the local form of the corresponding Morse function is $-x^2=t$ for $-1<t<0$. The two points touch at the critical point $(0,0,0,0)$ which is the intersection $D^1 \cap D^3$ (within the $4$-dimensional handle $D^1\times D^3$), see the middle instance of Fig.~\ref{3D_Cores}~(1). Then, the core $S^2$ of $D^1\times S^2$ uncollapses along the axes $y,z,w$ under the influence of the gradient $\nabla f=(-2x,2y,2z,2w)$ which is positive along axes $y,z,w$, see the  rightmost instance of Fig.~\ref{3D_Cores}~(1). The local form of the corresponding Morse function along $y,z,w$ is $w^2+z^2+y^2=t$ for $0<t<1$. Note that the core $S^0$ (respectively the core $S^2$) bounds the disc $D^1$ (respectively the $3$-ball $D^3$) of  the $4$-dimensional handle $D^1\times D^3$ in which the process takes place. If we view the gradient of Fig.~\ref{3D_Cores}~(1) as a force, one can imagine the $4$-dimensional process by following the line of thought presented in Section~\ref{Gradient2D}. Namely, the attracting forces acting on the cores $S^0$ are fleshed out to the two $3$-balls $S^0 \times D^3$ until the critical point is reached, after which the repelling forces uncollapsing the core $S^2$ are fleshed out to the thickened sphere $D^1 \times S^2$.

\begin{figure}[ht!]
\begin{center}
\includegraphics[width=11cm]{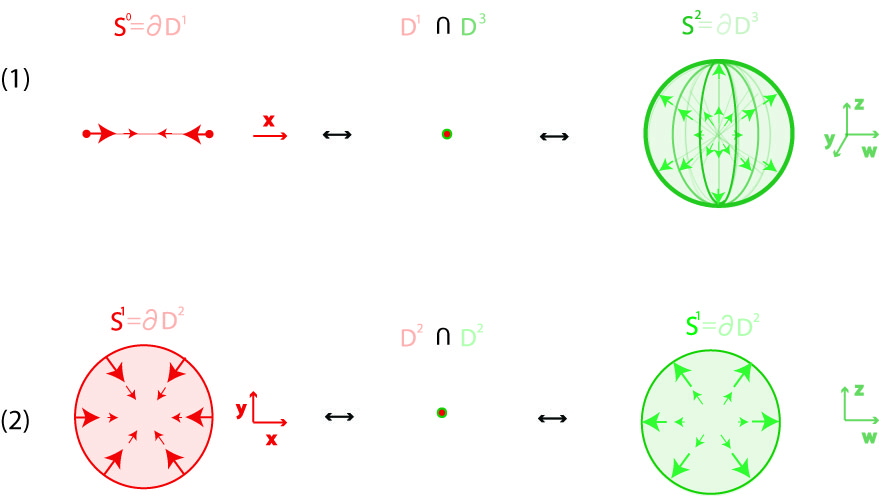}
\caption{ Core view of \textbf{(1)} $3$-dimensional $0$-surgery \textbf{(2)} $3$-dimensional $1$-surgery}
\label{3D_Cores}
\end{center}
\end{figure}  
\smallbreak

Similarly, for \textit{$3$-dimensional $1$-surgery}, we consider the Morse function of Lemma~\ref{LocalMorse} for the case $m=3$ and $i=2$: $$g:D^{4} \to \mathbb{R}; \quad (x,y,z,w) \mapsto -x^2-y^2+z^2+w^2$$ In this case, the local process of $3$-dimensional $1$-surgery happens inside the handle $D^2\times D^2$ and can be described by varying parameter $t$ of the level hypersurfaces $-x^2-y^2+z^2+w^2=t$. In Fig.~\ref{3D_Cores}~(2), the `core view' of $3$-dimensional $1$-surgery is presented. The process starts with the core $S^1$ (shown in red) of the solid torus $S^1\times D^2$. The points of this circle are attracted towards $(0,0,0,0)$ under the influence of the gradient $\nabla g=(-2x,-2y,2z,2w)$ which is negative along axes $x,y$, see the leftmost instance of Fig.~\ref{3D_Cores}~(2). Along these axes, the local form of the corresponding Morse function is $-x^2-y^2=t$ for $-1<t<0$. The circle collapses at the critical point $(0,0,0,0)= D^2 \cap D^2$, see the middle instance of Fig.~\ref{3D_Cores}~(2). Then, the core $S^1$ (shown in green in the rightmost instance of Fig.~\ref{3D_Cores}~(2)) of the solid torus with the factors reversed, $D^2\times S^1$, uncollapses along axes $z,w$ under the influence of the gradient $\nabla g=(-2x,-2y,2z,2w)$ which is positive along axes $z,w$. The local from of the corresponding Morse function along $z,w$ is $z^2+w^2=t$ for $0<t<1$. Note that each core $S^1$ bounds a disc $D^2$ of  the $4$-dimensional handle $D^2\times D^2$ in which the process takes place. If we view the gradient of Fig.~\ref{3D_Cores}~(2) as a force, the $4$-dimensional process can be imagined as follows: the attracting forces acting on the core $S^1$  are fleshed out to the solid torus $S^1 \times D^2$ until the critical point is reached, after which the repelling forces uncollapsing the other core $S^1$ are fleshed out to the other solid torus $D^2 \times S^1$.

\subsection{$3$-dimensional surgery via rotation}\label{3DRot}
In~\cite{SS1}, it is was remarked that $2$-dimensional surgery can be obtained from $1$-dimensional surgery by rotation. Here, we will prove this fact using Morse functions. Let us start by remarking that the level surfaces of $2$-dimensional surgery can be obtained by rotating the level curves of $1$-dimensional $0$-surgery. Indeed, for any given time $t \in (0,1)$, rotating the hyperbolas $-x^2+y^2=t$ around the $x$-axis creates the surfaces $-x^2+y^2+z^2=t$ which describe $2$-dimensional $0$-surgery, see passage of Fig.~\ref{TS_CP_1D_Gradient} to Fig.~\ref{TS_CP_2D_Gradient_Black_2}. Note that, instead of rotating each such temporal slice, one can rotate the whole handle which is comprised of them. For example, rotating the handle $D^1 \times D^1$ made of the parametrized hyperbolas of $1$-dimensional $0$-surgery gives us the handle $D^1 \times D^2$  made of the parametrized surfaces of $2$-dimensional $0$-surgery. The rotation happens around the $x$-axis in the $(y,z)$-plane thus turning $D^1 \times D^1$ to $D^1 \times D^2$ by creating the new repelling direction in the $z$-axis, see the passage from Fig.~\ref{123Rotation}~(1) to~(2). As also shown in Fig.~\ref{123Rotation}~(1) to~(2), the collapsing segments $S^0 \times D^1$ are expanded to $S^0 \times D^2$ while the rotation of core $S^0$ of the uncollapsing segments $D^1 \times S^0$ turns into core $S^1$ of the uncollapsing cylinder  $D^1 \times S^1$. Note that the reverse process of Fig.~\ref{123Rotation}~(2) results in a necking of the cylinder $D^1 \times S^1$, collapsing to the center and recoupling, thus, it describes, $2$-dimensional $1$-surgery via rotation.

\smallbreak
\begin{figure}[ht!]
\begin{center}
\includegraphics[width=13.5cm]{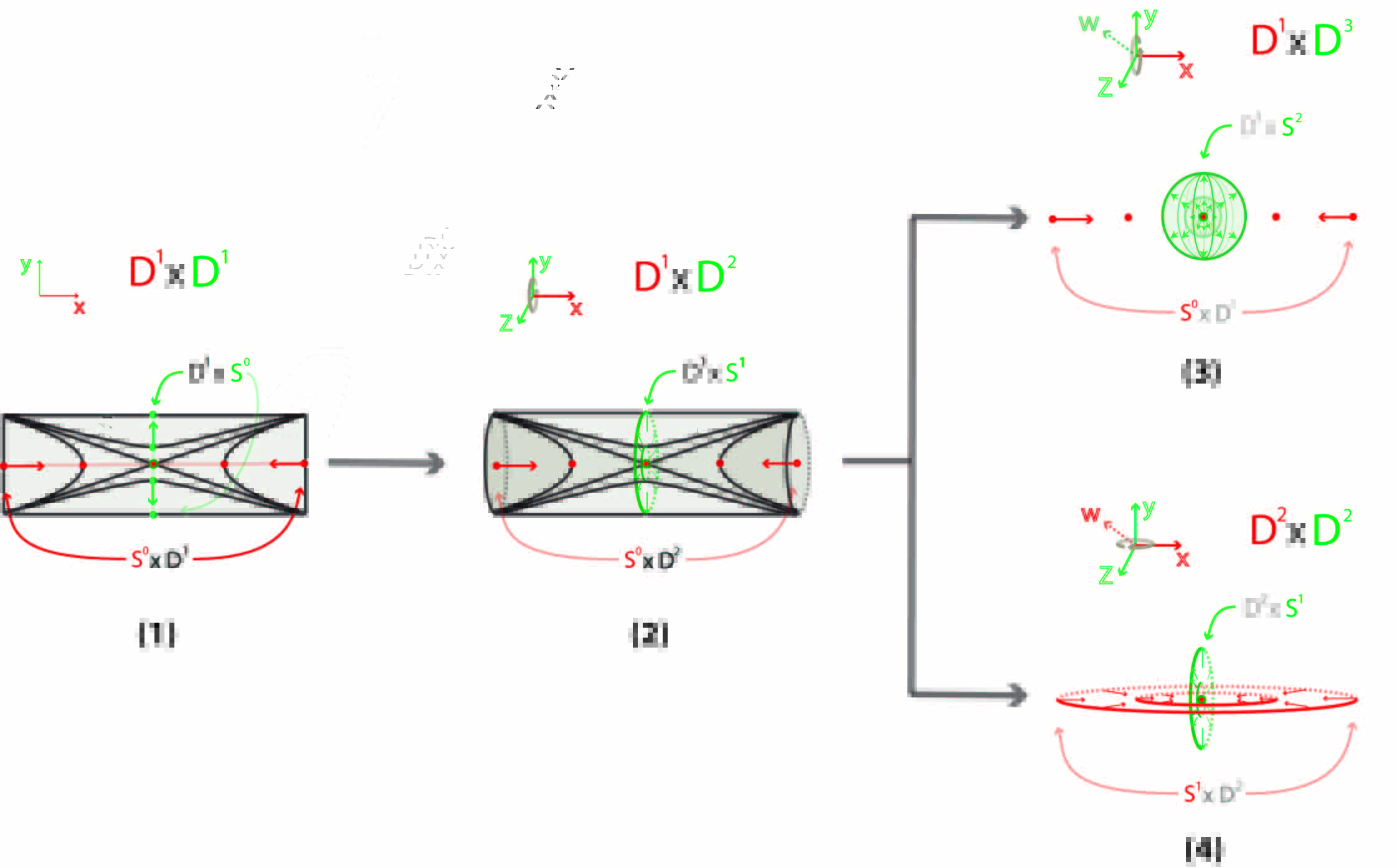}
\caption{ \textbf{(1)} $1$-dimensional $0$-surgery \textbf{(2)} $2$-dimensional $0$-surgery via rotation \textbf{(3)} $3$-dimensional $0$-surgery via rotation \textbf{(4)} $3$-dimensional $1$-surgery via rotation}
\label{123Rotation}
\end{center}
\end{figure}

Moving one dimension up, the instances of both types of $3$-dimensional surgery can be seen as rotations of the instances of $2$-dimensional $0$-surgery taking place in $D^1 \times D^2$. More precisely, for \textit{$3$-dimensional $0$-surgery}, a rotation of $D^1 \times D^2$ around the $x$-axis in the $(y,z,w)$-hyperplane turns it to $D^1 \times D^3$ by creating the new repelling direction in the $w$-axis, see  Fig.~\ref{123Rotation}~(3). The resulting handle $D^1 \times D^3$ is made of the layering of the hypersurfaces $-x^2+y^2+z^2+w^2=t$, $-1<t<1$. In this case, the collapsing discs $S^0 \times D^2$ are thickened to collapsing $3$-balls $S^0 \times D^3$ while the core $S^1$ of the uncollapsing cylinder $D^1 \times S^1$ turns into the core $S^2$ of the uncollapsing thickened sphere $D^1 \times S^2$, see the passage from Fig.~\ref{123Rotation}~(2) to~(3). Note that as handle $D^1 \times D^3$ is $4$-dimensional, only the core view is shown in Fig.~\ref{123Rotation}~(3).

Similarly, for \textit{$3$-dimensional $1$-surgery}, a rotation around the $y$-axis in the $(x,z,w)$-hyperplane turns $D^1 \times D^2$ to $D^2 \times D^2$ by creating the new attracting direction in the $w$-axis, see Fig.~\ref{123Rotation}~(4). In this case, the resulting handle $D^2 \times D^2$ is made of the  hypersurfaces $-x^2+y^2+z^2-w^2=t$, $-1<t<1$. Here, the rotation of the core $S^0$ of the collapsing discs $S^0 \times D^2$ creates the core $S^1$ of the collapsing solid torus $S^1 \times D^2$, while the uncollapsing of the cylinder $D^1  \times S^1$ creates via rotation the uncollapsing solid torus $D^2  \times S^1$, see the passage from Fig.~\ref{123Rotation}~(2) to~(4) where only the core view is shown. Note that, in the local form of the Morse function presented here, directions $y$ and $w$ are interchanged compared to the Morse function $g$ presented in the previous section. This is just a matter of convention and is due to the fact that, in Lemma~\ref{LocalMorse}, Morse functions sum up the negative coordinates first,  hence considering that directions $x$ and $y$ are attracting, whereas here the two attracting directions are $x$ and $w$ because we rotated the predefined coordinates of the local form of the Morse function of $2$-dimensional $0$-surgery.

\begin{rem} \rm
The rotations creating the handles comprised of the instances of both types of $3$-dimensional surgery correspond to thickenings of the core views of Fig.~\ref{123Rotation}~(3) and~(4). However, as already mentioned, this requires the fourth dimension in order to be visualized. Yet, one can visualize the initial and the final instance of both processes of $3$-dimensional surgery in $\mathbb{R}^{3}$ by using stereographic projection. This visualization is presented in the Appendix \ref{2D3DStereo}.
\end{rem}

\subsection{$m$-dimensional surgery via rotation}\label{ND}

In the previous section we showed how one can obtain the instances of surgery and the  local forms of the corresponding Morse functions for dimensions $1,2$ and $3$. In this section we generalize this idea for an arbitrary surgery dimension $m$.

\smallbreak
\begin{figure}[ht!]
\begin{center}
\includegraphics[width=4cm]{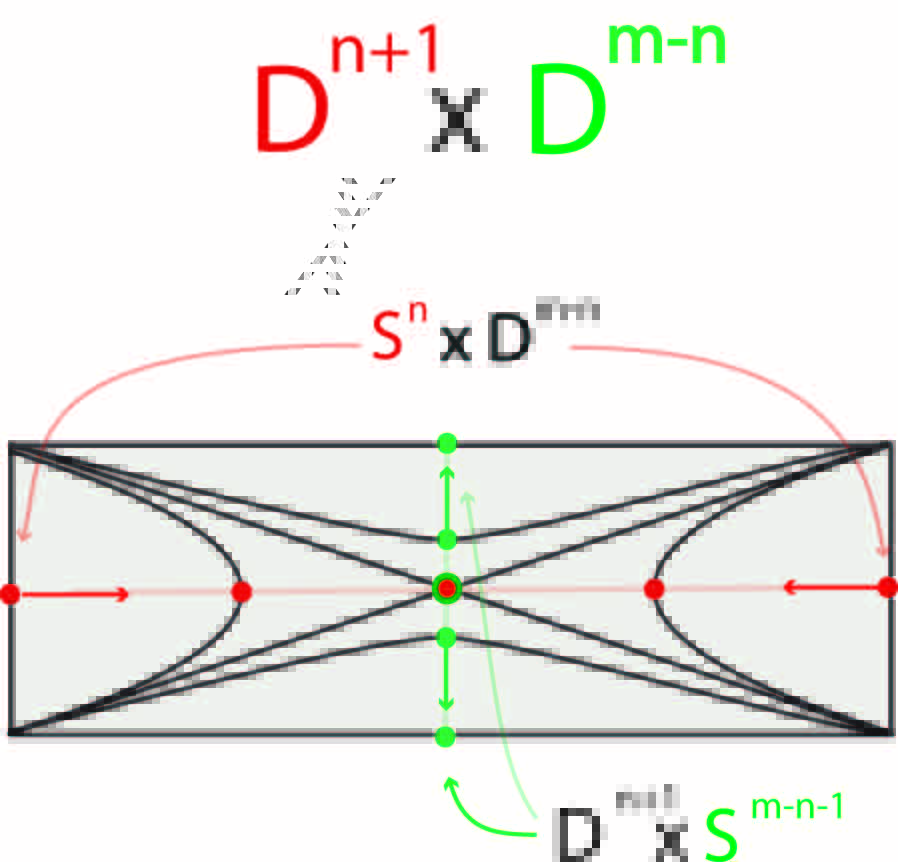}
\caption{ $m$-dimensional $n$-surgery within $D^{n+1} \times D^{m-n}$}
\label{HigherDRotation}
\end{center}
\end{figure}

The local process of an $m$-dimensional $n$-surgery is abstracted in Fig.~\ref{HigherDRotation}. Its instances are made of hypersurfaces given by:

$$   -‎‎\sum\limits‎_‎{j=1}^{n+1‎} x^2_{j} +‎‎\sum\limits‎_‎{j=n+2}^{m+1‎} x^2_{j}=t   , \quad -1<t<1    $$

By varying parameter $t$, one continuously collapses the core $S^n$ of the thickened sphere $S^n \times D^{m-n}$ to the critical point $D^{n+1} \cap D^{m-n}$ from which the core $S^{m-n-1}$ of the thickened sphere $D^{n+1} \times S^{m-n-1}$ uncollapses. The handle $D^{n+1} \times D^{m-n}$ made of these instances can be obtained by $(m-1)$ successive rotations in increasingly higher dimensions of the initial handle $D^{1} \times D^{1}$ made of the instances of $1$-dimensional $0$-surgery.

\subsection{Outlining the $4$-dimensional process }\label{4Doutline}

One can provide an outline of the $4$-dimensional process of $3$-dimensional surgery by analogy to what happens in one dimension lower. We start by illustrating in Fig.~\ref{3D_ThickenedCores}~(1) the $3$-dimensional process of $2$-dimensional $0$-surgery. In the figure we deliberately choose a homeomorphic representation, as the one exhibited by Falaco solitons in Fig.~\ref{2D_Falacohomeo2}, where the two discs $S^0\times D^2$ start embedded in the plane $\mathbb{R}^2$, see instance (a) of Fig.~\ref{3D_ThickenedCores}~(1), but the rest of the process happens in $\mathbb{R}^3$, see instances (b)-(e) of Fig.~\ref{3D_ThickenedCores}~(1).

\smallbreak
\begin{figure}[ht!]
\begin{center}
\includegraphics[width=15cm]{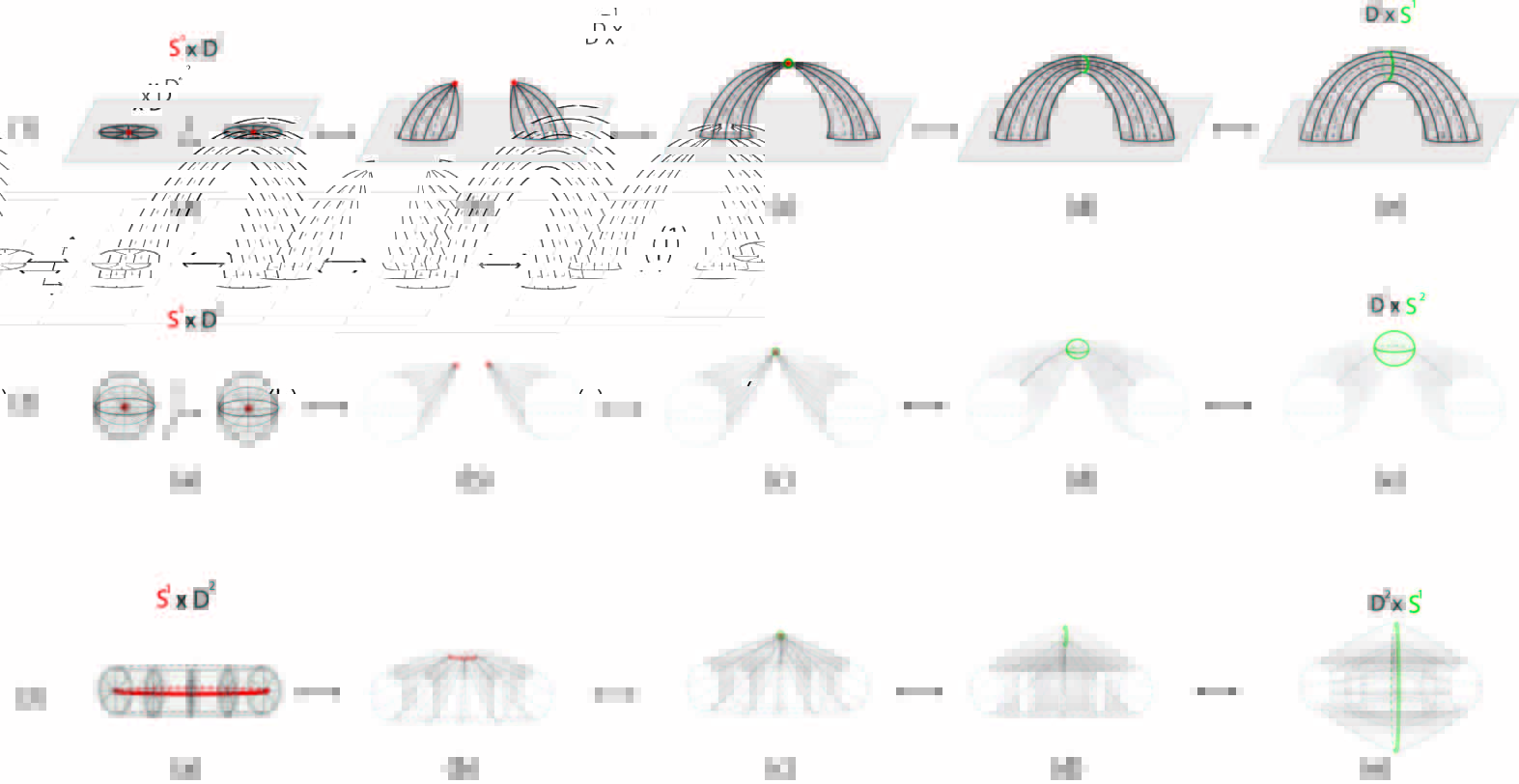}
\caption{ (1) $2$-dimensional $0$-surgery (2) Outline of $3$-dimensional $0$-surgery (3) Outline of $3$-dimensional $1$-surgery}
\label{3D_ThickenedCores}
\end{center}
\end{figure}

In analogy, if a \textit{$3$-dimensional $0$-surgery} starts with two $3$-balls  $S^0\times D^3$ embedded in $\mathbb{R}^3$, see instance (a) of Fig.~\ref{3D_ThickenedCores}~(2), then the rest of the process takes place in $\mathbb{R}^4$, see instances (b)-(e) of Fig.~\ref{3D_ThickenedCores}~(2). Instances (b) and (c) of Fig.~\ref{3D_ThickenedCores}~(2) illustrate the fact that the two $3$-balls  $S^0\times D^3$ `bend' and touch in the fourth dimension while instances (d) and (e) of the same figure illustrate the emerging of the thickened sphere $D^1 \times S^2$. Note that instances (b)-(e) are deliberately shown with increased transparency to depict the fact that the higher dimensional merging and recoupling is not visible in $\mathbb{R}^3$. 

Similarly, if a \textit{$3$-dimensional $1$-surgery} starts with a solid torus $S^1\times D^2$ embedded in $\mathbb{R}^3$, see instance (a) of Fig.~\ref{3D_ThickenedCores}~(3), then the rest of the process in $\mathbb{R}^4$ is outlined in instances (b)-(e) of the same figure. More precisely, instances (b) and (c) sketch the higher dimensional collapse of the solid torus  $S^1\times D^2$ while instances (d) and (e) of Fig.~\ref{3D_ThickenedCores}~(3) sketch the emerging of the solid torus $D^2 \times S^1$ (with the factors reversed).

\section{Global topology and $3$-dimensional surgery}\label{3DtopoAll}
In this section we discuss the global effect of both types of $3$-dimensional surgery on a $3$-manifold $M$ and present some examples and visualizations. As we will see in Section~\ref{3D0Topo}, the result of $3$-dimensional $0$-surgery on a $3$-manifold $M$ is homeomorphic to $M \# ({S^{1}\times S^2})$. On the other hand, $3$-dimensional $1$-surgery is a much more powerful topological tool. Indeed, as explained in Section~\ref{3D1Topo}, starting with $M=S^3$, this type of surgery can create the whole class of closed, connected, orientable $3$-manifolds.
 
\subsection{$3$-dimensional $0$-surgery}\label{3D0Topo}
In Section~\ref{3D0S3}, we present the process of $3$-dimensional $0$-surgery on $M=S^3$. In Section~\ref{ConnSum3d0}, we define the connected sum of two manifolds and show the result of $3$-dimensional $0$-surgery on a $3$-manifold. Finally, in Section~\ref{Fundamen3d0}, we characterize the effect of surgery by determining the fundamental group of the resulting manifold.

\subsubsection{$3$-dimensional $0$-surgery in \texorpdfstring{$S^3$}{S3}}\label{3D0S3}
Let us start by recalling that the $3$-sphere $S^3$ is made by gluing two $3$-balls along their common boundary. Hence, $S^3=B_1^{3} \cup_\theta B_2^{3}$, via a homeomorphism $\theta$ along the boundary $S^2=\partial B_1^{3}=\partial B_2^{3}$. By the Alexander Lemma (see for example \cite{Ro}), any such homeomorphism $\theta$ extends to a homeomorphism between the two $3$-balls, so the result of this gluing will always be homeomorphic to $S^3$.

This decomposition is very helpful in examining the effect of $3$-dimensional $0$-surgery on $M=S^3$ as we can consider that one of the two $3$-balls to be removed, $S^0 \times D^3$, is $B_1^{3}$ while the other one is embedded inside $B_2^{3}$, see Fig.~\ref{3D0_S3}~(1) where the curved vectors in grey represent `gluing along the common boundary'.

\smallbreak
\begin{figure}[ht!]
\begin{center}
\includegraphics[width=16cm]{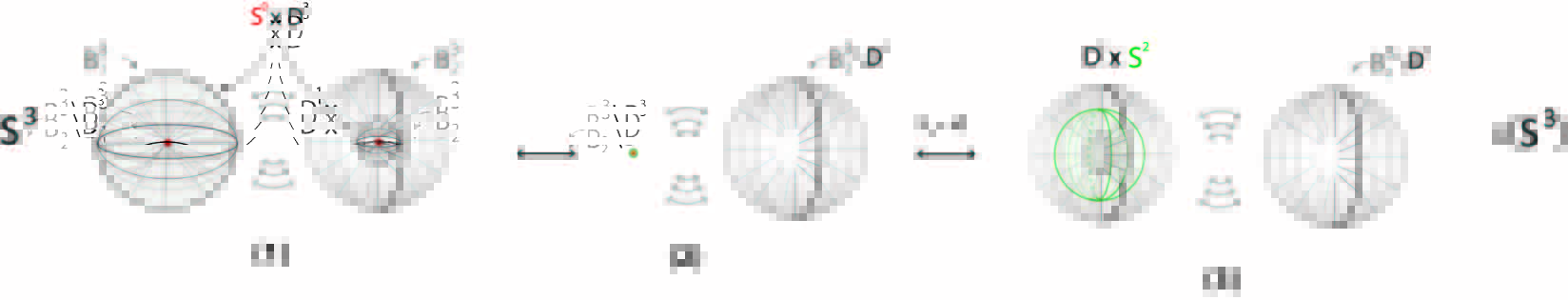}
\caption{ \textbf{(1)} $S^3 = B_1^{3} \cup_\theta B_2^{3}$ \textbf{(2)} $S^0 \times D^3$ collapses  \textbf{(3)} $D^1 \times S^2$ emerges}
\label{3D0_S3}
\end{center}
\end{figure}

The process of \textit{$3$-dimensional $0$-surgery in $S^3$} collapses two $3$-balls $S^0\times D^3$ to a singular point and we are left with $\overline{S^3 \setminus (S^0 \times D^3)}= \overline{B_2^{3}  \setminus D^3}$, a thickened sphere, see Fig.~\ref{3D0_S3}~(2). Then, another thickened sphere $D^1 \times S^2$ (which is a $3$-dimensional tube) uncollapses and is glued with $\overline{B_2^{3}  \setminus D^3}$ along the two common spherical boundaries, see Fig.~\ref{3D0_S3}~(3) and the resulting manifold is:

\begin{equation} 
\begin{aligned}
\chi(S^3) = \chi(B_1^{3} \cup_\theta B_2^{3}) = \overline{(B_1^{3} \cup_\theta B_2^{3})\setminus h_s(S^0\times D^3)} \cup_{h_s|_{S^0\times S^{2}}} (D^1 \times S^2) \\ 
= \overline{(B_2^{3} \setminus D^3)} \cup_{h_s|_{S^0\times S^{2}}} (D^1 \times S^2).
\end{aligned} \nonumber
\end{equation}

As we will see in next section, there is a simpler homeomorphic representation of the resulting manifold.

\subsubsection{$3$-dimensional $0$-surgery in $M^3$}\label{ConnSum3d0}
Let us start by defining the connected sum:
\begin{defn} \label{ConnSum} \rm  The \textit{connected sum} of two $m$-dimensional manifolds $M$, $M'$ is the $m$-dimensional manifold $M \# M'$: 

$$ M \# M'=\overline{(M\setminus D^{m})} \cup (D^1\times {S^{m-1}}) \cup \overline{(M'\setminus D^{m})}  $$

\noindent obtained by excising the interiors of two embedded $m$-discs, $D^{m}\hookrightarrow  M$ and $ D^{m}\hookrightarrow  M' $,  and joining the resulting boundary components $S^{m-1}\hookrightarrow  \overline{(M\setminus D^{m})}$ and  $S^{m-1}\hookrightarrow  \overline{(M'\setminus D^{m})}$ by an $m$-dimensional tube (or a thickened sphere) $D^1\times {S^{m-1}}$.
\end{defn}

Equivalently, the connected sum $M \# M'$ can be viewed as the effect of an $m$-dimensional $0$-surgery on the disjoint union $M \sqcup M'$ which removes the embeddings ${S^{0}\times D^{m}}\hookrightarrow  M \sqcup M'$ defined by the disjoint union of embeddings $D^{m}\hookrightarrow  M$ and $ D^{m}\hookrightarrow  M' $ and connects $M$ and $M'$ by an $m$-dimensional tube $D^1\times {S^{m-1}}$. Conversely, an $m$-dimensional $0$-surgery can be viewed as a connected sum. More precisely, in the following proposition we show that the result of $m$-dimensional $0$-surgery on an $m$-manifold $M$ is homeomorphic to connecting $M$ and $S^1 \times S^{m-1}$ by a higher dimensional tube $D^1 \times S^{m-1}$, see Fig.~\ref{3D0_TopoChange}. Note that, in the figure all manifolds are shown for $m=2$. For example $D^m, D^{1}\times S^{m-1}$ and $S^{1}\times S^{m-1}$   are shown as $D^2, D^{1}\times S^{1}$ and $S^{1}\times S^1$ respectively. 

\begin{prop} \label{3d0S1S2}
The result $\chi(M)$ of $m$-dimensional $0$-surgery on a $m$-manifold $M$ is homeomorphic to the connected sum $M \# ({S^{1}\times S^{m-1}})$.
\end{prop}

\begin{proof}
We will first show that  $$ \chi{(D_0^{m})} \cong  \overline{({S^{1}\times S^{m-1}})\setminus D^{m}} \quad (\star)$$
In other words, the result of $m$-dimensional $0$-surgery on the disc $D_0^{m}$ is homeomorphic to the punctured  ${S^{1}\times S^{m-1}}$. For seeing this, we first consider $S^1$ as made up by two segments $D^1$: ${S^{1}\times S^{m-1}} = (D^{1} \cup D^1) \times S^{m-1}$. With this decomposition, we can remove a $D^{1}\times S^{m-1}$  from both sides of equation $(\star)$: $\overline{D_0^{m}\setminus (S^0\times D^{m})} \cup (D^{1}\times S^{m-1}) \setminus (D^{1}\times S^{m-1}) \cong  \overline{({D^{1}\times S^{m-1}}) \cup ({D^{1}\times S^{m-1}}) \setminus D^{m}} \setminus (D^{1}\times S^{m-1}) \iff \overline{D_0^{m}\setminus (S^0\times D^{m})} \cong  \overline{({D^{1}\times S^{m-1}})\setminus D^{m}}$.  
 
So, with the handles $D^{1}\times S^{m-1}$ removed, we only need to show that the remaining manifolds are homeomorphic. View Fig.~\ref{3D0_MorseConcat}~(1) where both $D^{1}\times S^{m-1}$ are shown with increased transparency. This is made clear in Fig.~\ref{3D0_MorseConcat}~(2) where both $\overline{D_0^{m}\setminus (S^0\times D^{m})}$ and $\overline{(D^{1}\times S^{m-1})\setminus D^{m}}$ are decomposed into Morse levels. For $m=2$ the Morse levels start as one circle (see levels $-4$ to $-1$ in  Fig.~\ref{3D0_MorseConcat}~(2)), which passes through a critical point (see level $0$ in  Fig.~\ref{3D0_MorseConcat}~(2)) and is divided into two circles (see levels $1$ to $4$  in  Fig.~\ref{3D0_MorseConcat}~(2)). Since the Morse levels of both $\overline{D_0^{2}\setminus (S^0\times D^{2})}$ and $\overline{(D^{1}\times S^{1})\setminus D^{2}}$ have been corresponded, these two manifolds are homeomorphic. The same decomposition can be generalized for $\overline{D_0^{m}\setminus (S^0\times D^{m})}$ and $\overline{(D^{1}\times S^{m-1})\setminus D^{m}}$ by considering level spheres $S^{m-1}$ instead of circles $S^1$.

\smallbreak
\begin{figure}[ht!]
\begin{center}
\includegraphics[width=13cm]{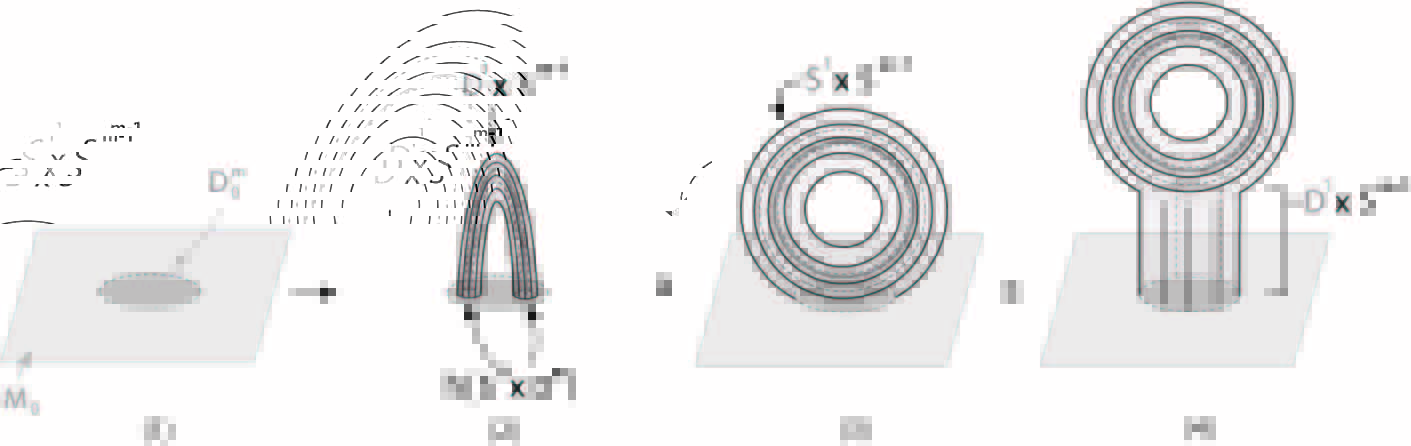}
\caption{ \textbf{(1)} $M=M_0 \cup D_0^{m}$ \textbf{(2)} $\chi(M)$ \textbf{(3)} $M \# ({S^{1}\times S^{m-1}})$ \textbf{(4)} Homeomorphic representation of $M \# ({S^{1}\times S^{m-1}})$ }
\label{3D0_TopoChange}
\end{center}
\end{figure}

\smallbreak
\begin{figure}[ht!]
\begin{center}
\includegraphics[width=15cm]{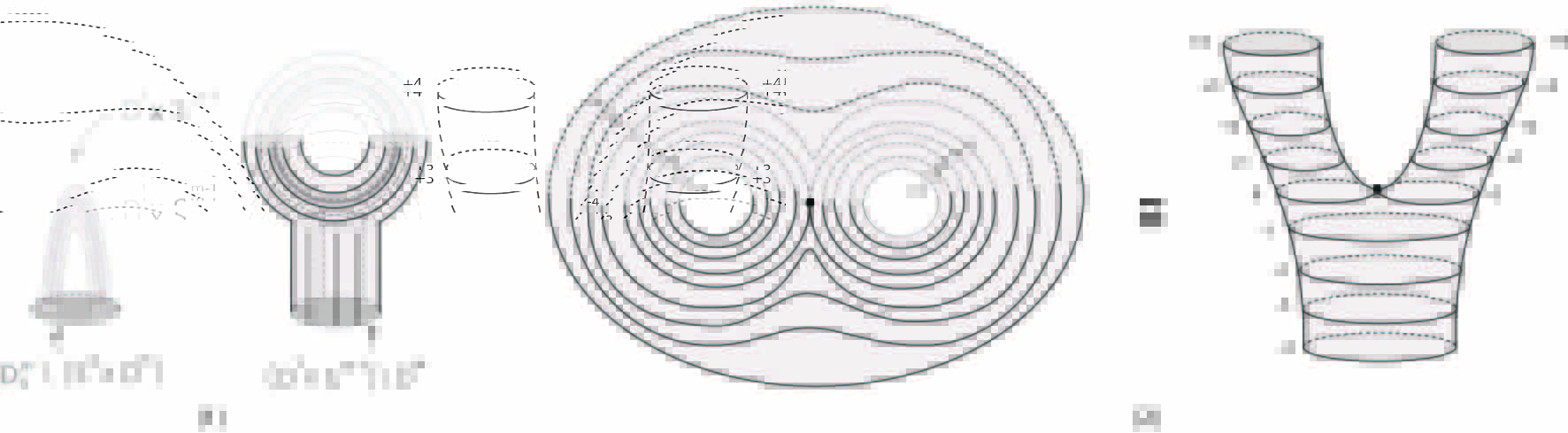}
 \caption{ \textbf{(1)} Removing $(D^{1}\times S^{m-1})$  \textbf{(2)} Homeomorphic representations of $D_0^{m}\setminus (S^0\times D^{m})$ }
\label{3D0_MorseConcat}
\end{center}
\end{figure}

Now, if we consider $M_0$ and $D_0^{m}$ such that $M=M_0 \cup D_0^{m}$, see Fig.~\ref{3D0_TopoChange}~(1), with $D_0^{m}$ containing $h(S^0\times D^{m})$, see Fig.~\ref{3D0_TopoChange}~(2), then $\chi(M) = \overline{M\setminus h(S^0\times D^{m})} \cup_{h} (D^{1}\times S^{m-1})= \overline{M_0 \cup (D_0^{m} \setminus h(S^0\times D^{m}))} \cup_{h} (D^{1}\times S^{m-1}) = \overline{M_0} \cup \overline{(D_0^{m} \setminus h(S^0\times D^{m}))} \cup_{h} (D^{1}\times S^{m-1})$, which, using $(\star)$, gives  $\chi(M) \cong  \overline{M_0} \cup \overline{({S^{1}\times S^{m-1}}\setminus D_0^{m})} = \overline{M \setminus{D_0^{m}}} \cup \overline{({S^{1}\times S^{m-1}} \setminus D_0^{m})}$. Notice now that elongating  the punctured ${S^{1}\times S^{m-1}} $ by a tube $D^{1}\times S^{m-1}$ results in a homeomorphic manifold:  ${S^{1}\times S^{m-1}} \setminus D_0^{m}  \cong (D^{1}\times S^{m-1})   \cup ({S^{1}\times S^{m-1}}\setminus D_0^{m})$, see the passage of  Fig.~\ref{3D0_TopoChange}~(3) to Fig.~\ref{3D0_TopoChange}~(4). Hence, $\chi(M) \cong  \overline{M \setminus{D_0^{m}}} \cup  (D^{1}\times S^{m-1})   \cup \overline{({S^{1}\times S^{m-1}}\setminus D_0^{m})}= M \# ({S^{1}\times S^{m-1}})$.
\end{proof}

Let now $M$ be an arbitrary $3$-manifold. The process of $3$-dimensional $0$-surgery on $M$ is analogous to the process described in Section~\ref{3D0S3} for $S^3$. By Proposition~\ref{3d0S1S2}, the effect of $3$-dimensional $0$-surgery on $M$ is homeomorphic to connecting $M$ and the lens space ${S^{1}\times S^2}$ by a higher dimensional tube ${D^1 \times S^{2}}$. Recall Fig.~\ref{3D0_TopoChange} where  all manifolds are shown one dimension lower.

\subsubsection{Fundamental group}\label{Fundamen3d0}
Another way of characterizing the effect of $m$-dimensional $0$-surgery on an $m$-manifold $M$ is by determining the fundamental group of the resulting manifold. The fundamental group records basic information about a manifold and is a topological invariant: homeomorphic manifolds have the same fundamental group. For details on the fundamental group see Appendix~\ref{Fundappendix}. The fundamental group of $\chi(M)$ can be characterized using the following lemma which is a consequence of the  Seifert–van Kampen theorem (see for example~\cite{Munkres}):

\begin{lem} \label{fundalem3d0} \rm  Let $m \geq 3$. Then the fundamental group of a connected sum is the free product of the fundamental groups of the components: 
$$ \pi_1(M \# M')\cong \pi_1(M) * \pi_1(M')$$
\end{lem}

Based on the above, a $3$-dimensional $0$-surgery on $M$ alters its fundamental group as follows: $\pi_1(\chi(M)) \cong \pi_1(M \# ({S^{1}\times S^2})) \cong \pi_1(M) * \pi_1({S^{1}\times S^2})\cong \pi_1(M) * ( \pi_1{(S^{1})} \times \pi_1{(S^{2})}) \cong \pi_1(M) * \mathbb{Z}$.

\subsection{$3$-dimensional $1$-surgery}\label{3D1Topo}
In Section~\ref{3D1S3}, we present the process of  $3$-dimensional $1$-surgery on $M=S^3$ when a trivial embedding is used. Then, in Section~\ref{3D1Knot}, we introduce the notion of `knot surgery', which also includes non-trivial embeddings, and we present the Lickorish-Wallace theorem stating that knot surgery can create all closed, connected, orientable $3$-manifolds. Finally, in Section~\ref{Fundamen3d1}, we characterize the effect of knot surgery on $M=S^3$ by determining the fundamental group of the resulting manifold.

\subsubsection{$3$-dimensional $1$-surgery in \texorpdfstring{$S^3$}{S3}}\label{3D1S3}
In Section~\ref{Types3D}, we described briefly the mechanism of $3$-dimensional $1$-surgery. Let us now recall that the $3$-sphere $S^3$ can be obtained as the union of two solid tori,  $S^3 = V_1\cup_\theta  V_1\textsubscript{c}$, where $V_1\textsubscript{c}$ stands for the complement of $V_1$ and $\theta$ is the standard torus homeomorphism  along the common boundary mapping each longitude (respectively meridian) of $V_1$ to a meridian  (respectively longitude) of $V_1\textsubscript{c}$. A visualization of both solid tori in $\mathbb{R}^3$ using the stereographic projection can be found in Appendix \ref{2D3DStereo}, see Fig.~\ref{3D_Decomp_F}~({3\textsubscript{a}}) or Fig.~\ref{3D_Decomp_F}~({3\textsubscript{b}}).
This decomposition is clearly very helpful for examining the effect of $3$-dimensional $1$-surgery on $S^3$ for the case of a trivial embedding $h_s$ of $V_1$. Namely, the complement solid torus $V_1\textsubscript{c}$ remains identically fixed throughout the process while $V_1$ is replaced by a solid torus $V_2$ with the factors reversed via a homeomorphism $h_s$ from the boundary of $V_1$ to the boundary of $V_2$.

To avoid confusion and keep the color coding consistent with previous sections, the solid tori $V_1$ and $V_2$ will be considered as the initial and final instances of the local process of surgery (keeping the respective red-green color coding of their core curves) while the complement torus of $V_1$ in $S^3$ will be $V_1\textsubscript{c}$ and its core curve will be shown in blue. The initial manifold $M=S^3$ can be seen in Fig.~\ref{3D_S3_S2S1}~(1) where the curved vectors in grey represent `gluing along common boundary'. We will consider Fig.~\ref{3D_S3_S2S1}~(1) as the initial setup for $3$-dimensional $1$-surgery on $S^3$. 

\smallbreak
\begin{figure}[ht!]
\begin{center}
\includegraphics[width=16cm]{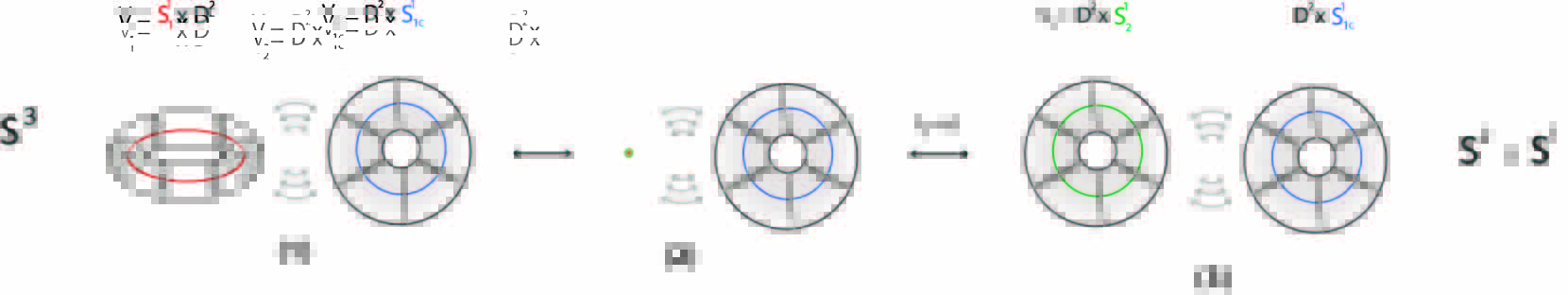}
\caption{ \textbf{(1)} $S^3 = V_1\cup V_1\textsubscript{c}$ \textbf{(2)} $V_1$ collapses  \textbf{(3)} $V_2$ emerges}
\label{3D_S3_S2S1}
\end{center}
\end{figure}

One key difference compared to $3$-dimensional $0$-surgery where the embedding of the core $S^0$ of $S^0 \times D^3$ didn't influence the resulting manifold is that, here, the higher dimension of the core $S^1$ allows for knotted embeddings of the solid torus $S^1 \times D^2$. As we will see, this knotting plays a crucial role in the result of surgery. We will start by discussing the trivial embedding in this section and then introduce the notion of `knot surgery' in Section~\ref{3D1Knot}.

When a trivial embedding $h_s$ is used, the embedding $h_s(V_1)=h_s(S_1^1\times D^2)$ corresponds to taking the tubular neighborhood of an unknotted core $S_1^1$ where longitudes $\ell_1$ are parallel to the core. As $S^3 = V_1\cup_{h_s}  V_2$, the induced `gluing' homeomorphism along the common boundary $S_1^1\times S_2^1$ maps each  longitude (respectively meridian) of solid torus $V_1$ to the meridian (respectively longitude) of solid torus $V_2$. Hence, $h_s(\ell_1)=m_2$ and  $h_s(m_1)=\ell_2$. The process of surgery collapses $S_1^1\times D^2$, see Fig.~\ref{3D_S3_S2S1}~(2) and then uncollapses $D^2 \times S_2^1$, see Fig.~\ref{3D_S3_S2S1}~(3). Given
that the solid torus $V_2=D^2 \times S_2^1$ is homeomorphic to  $V_1\textsubscript{c}=D^2 \times S_{1c}^1$ as they are both complements of $V_1$ in $S^3$, the resulting manifold is:

\begin{equation}
\begin{aligned}
\chi(S^3) = \overline{S^3\setminus h_s(S_1^1\times D^2)} \cup_{h_s|_{S_1^1\times S_2^{1}}} (D^2 \times S_2^1) = (D^2 \times S_{1c}^1) \cup_{h_s|_{S^1\times S_2^{1}}} (D^2 \times S_2^1)  \\
\cong (D^2 \times S_2^1) \cup_{h_s|_{S^1\times S_2^{1}}} (D^2 \times S_2^1) = (D^2 \cup_{h_s|_{S^1}} D^2) \times S_2^1 = S^2 \times S_2^1 .
\end{aligned} \nonumber
\end{equation}

\subsubsection{Knot surgery}\label{3D1Knot}
We start with the following fundamental theorem by A.H. Wallace~\cite{Wal} and W. Lickorish~\cite{LickTh}:

\begin{thm} [\cite{Wal} Thm 6, \cite{LickTh} Thm 2] \label{LickorishWal} \rm Every closed, connected, orientable $3$-manifold can be obtained by surgery on a knot or a link in $S^3$.
\end{thm}
{\noindent} Let us mention that a knot $K$ is an embedding of $S^1$ in  $\mathbb{R}^3$ or $S^3$ while a link is a collection of knots which do not intersect, but which may be linked (or knotted) together. It can be shown that this theorem is equivalent to saying that, \textit{starting with $M=S^3$, we can create every closed, connected, orientable $3$-manifold by performing a finite sequence of $3$-dimensional $1$-surgeries}, see~\cite{LickTh} or~\cite{PS} for details.

In this type of surgery, the role of the embedding is crucial. When using the standard embedding $h_{s}$, the core $S_1^1$ and the longitude $\ell_1$ of the removed solid torus $V_1$ are both trivial loops (or unknotted circles) and $3$-dimensional $1$-surgery generates a restricted family of $3$-manifolds. Indeed, starting from $S^3$, standard  embeddings $h_{s}$ can only produce $S^2 \times S^1$ or connected sums of $S^2 \times S^1$ while more complicated $3$-manifolds require using a non-trivial embedding $h$, where the core curve and the longitude of the removed solid torus $h(S_1^1\times D^2)$ can be knotted. One such $3$-manifold is the   Poincar\'{e}   homology sphere which is obtained by doing surgery on the trefoil knot with the right framing, see Fig.~\ref{TrefoilP}. For the definition the blackboard framing of a knot, see Appendix~\ref{Blackboard}. For details on the Poincar\'{e}   homology sphere, see Appendix~\ref{surgeryonK}.


Hence the process of $3$-dimensional $1$-surgery can be also described in terms of knots. We will call this process `knot surgery' in order to differentiate it from the process of $3$-dimensional $1$-surgery where  $h_{s}$ is used. Here, we can view the embedding $h(V_1)=h(S_1^1\times D^2)$ as a tubular neighbourhood $N(K)$ of knot $K$: $N(K)=K\times D^2=h(S_1^1\times D^2)$. The knot $K=h(S_1^1\times \{ 0 \})$ is the surgery curve at the core of solid torus $N(K)=h(S_1^1 \times D^2)$. On the boundary of $N(K)$, we further define the \textit{framing longitude} $\lambda \subset \partial N(K)$ with $ \lambda=h(S_1^1 \times \{ 1 \} )$, which is a parallel curve of $K$ on $\partial N(K)$, and the meridian $m_1 \subset \partial N(K)$ which bounds a disk of solid torus $N(K)$ and intersects the core $K$ transversely in a single point. 
 
A \textit{`knot surgery'} (or `framed surgery') along $K$ with framing $\lambda$ on a manifold $M$ is the process whereby $N(K)=h(V_1)$ is removed from $M$ and $V_2=D^2 \times S_2^1$ is glued along the common boundary. The interchange of factors of the `gluing' homeomorphism $h$ along $S_1^1 \times S_2^1$ can now be written as  $h(\lambda)=m_2$ and  $h(m_1)=l_2$.

Unlike the case of the standard embedding $h_{s}$ discussed in Section~\ref{3D1S3},  
the possible knottedness of $h$ makes this process harder to visualize. However, the manifold resulting from knot surgery can be understood by determining its fundamental group. This is done in next section.

\subsubsection{Fundamental group}\label{Fundamen3d1}
In this section, we present the theorem which characterizes the effect of knot surgery on $M=S^3$ by determining the fundamental group of the resulting manifold. We then apply it on the simple case of framed surgery along an unknotted surgery curve.

The fundamental group of the $3$-sphere $S^3$ is trivial, as any loop on it can be continuously shrunk to a point without leaving $S^3$. To examine how knot surgery alters the trivial fundamental group of $S^3$, let us consider the tubular neighborhood $N(K)$ of knot $K$. The generators of the group of $\partial N(K)$ are the longitudinal curve $\lambda$ and the meridional curve $m_1$. Note now that in $V_1=N(K)$ meridional curves bound discs while it is the specified framing longitudinal curve $\lambda$ that bounds a disc in $V_2=D^2 \times S_2^1$, since,  after surgery, the disc bounded by $m_2$ is now filling the longitude $\lambda$. Thus, $\lambda$ is made trivial in the fundamental group of $\chi_K(S^3)$. In this sense, surgery collapses $\lambda$. This statement is made precise by the following theorem which is a consequence of the Seifert-{}-van Kampen theorem (see for example~\cite{Munkres}): 
 
\begin{thm} \label{3d1long} \rm  Let $K$ be a blackboard framed knot with longitude $\lambda \in \pi_1(S^3 \setminus N(K))$. Let  $\chi_{\mbox{\tiny K}}(S^3)$ denote the $3$-manifold obtained by surgery on $K$ with framing longitude $\lambda$.  Then we have the isomorphism:
 $$  \pi_1(\chi_{\mbox{\tiny K}}(S^3)) \cong \frac{\pi_1(S^3 \setminus N(K))}{<\lambda>} $$
where $<\lambda>$ denotes the normal subgroup generated by $\lambda \in \pi_1(S^3 \setminus N(K))$.
\end{thm}

For a proof, the reader is referred to~\cite{Kirby,DNA}. The theorem tells us that in order to obtain the fundamental group of the resulting manifold, we have to factor out $<\lambda>$ from $\pi_1(S^3 \setminus N(K))$. 

\begin{exmp}  \label{ExFramedUnknot}  \rm
When the trivial embedding $h_{s}$ is used, then the `gluing' homeomorphism is $h_{s}(l_1)=m_2$, $K=the \ unknot$, $\lambda=l_1$ and $l_1$ is a trivial element in $\pi_1(S^3 \setminus N(K))$, so $<\lambda>=<0>$. In this case, we obtain the lens space $L(0,1)=S^2 \times S^1$ and the above formula gives us:

\begin{equation}
\begin{aligned}
\pi_1(\chi(S^3))\cong \frac{\pi_1(S^3\setminus h_s(S_1^1\times D^2))}{<\lambda>} = \frac{\pi_1(\mathring{D^2} \times S^1_{1c})}{<0>} = \frac{ \mathbb{Z}}{\{0\}} \cong\mathbb{Z}
\end{aligned} \nonumber
\end{equation}
\end{exmp}

\smallbreak
\begin{exmp}  \label{ExFramedUnknot2}  \rm
When we use a non-trivial embedding $N(K)=h(S_1^1 \times D^2)$ where the specified framing longitude $\lambda$ performs $p$ curls, the `gluing' homeomorphism is $h(\lambda)=m_2$ and we can consider that  $K=the \ unknot$, see Appendix~\ref{Blackboard} for details. In order to use Theorem~\ref{3d1long}, we have to find the subgroup generated by $\lambda=l_1+p \cdot m_1$ in $\pi_1(S^3 \setminus N(K))$. This subgroup is $<\lambda> = <l_1+p \cdot m_1> \cong <p \cdot m_1> \cong p \cdot <m_1>\cong p\mathbb{Z}$. In this case we obtain the lens space $L(p,1)$ and its fundamental group   is the cyclic group of order $p$:

\begin{equation}
\begin{aligned}
\pi_1(\chi(S^3))\cong \frac{\pi_1(S^3\setminus h(S_1^1\times D^2))}{<\lambda>} =\frac{\pi_1(\mathring{D^2} \times S^1_{1c})}{p\mathbb{Z}} =  { \mathbb{Z}}/{p\mathbb{Z}}  
\end{aligned} \nonumber
\end{equation}
\end{exmp}

As we saw in Example~\ref{ExFramedUnknot2}, if $\lambda$ is not a bounding curve in the knot complement, then we need to work out just what element $\lambda$ is in the fundamental group of the knot complement. This can be done by using one of the known presentations of the fundamental group, such as the Wirtinger presentation. A detailed presentation on the fundamental group of a knot $K$ and how we can use this presentation to determine the resulting manifold for knot surgery on $M=S^3$ along $K$ is done in Appendix~\ref{Fundappendix}.

\section{Topological processes of cosmic phenomena}\label{3d0WH}
In this section, we describe cosmic phenomena using topological surgery by  exploring the mathematical setting and developing the ideas presented in essay~\cite{BHsurg}. More precisely, we use $3$-dimensional surgery to analyze the temporal evolution and the topology change occurring during the formation of wormholes and cosmic string  black holes and we connect both of these cosmic phenomena with the $ER=EPR$ hypothesis of L. Susskind and J. Maldacena, see~\cite{ER_EPR_,ER_EPR}. Wormholes and cosmic string black holes are analyzed in Sections~\ref{3d0WH_GC} and~\ref{3d1BHWH} respectively. Significant outcomes of our study include the presentation of a possible entangled quantum state for wormholes, in Section~\ref{WQS}, and the avoidance of the singularity by conjecturing that a new $3$-manifold is created behind the event horizon, in Section~\ref{3d1BHWH_singA}.

In all subsequent sections we consider `space' as being the $3$-dimensional spatial section of the $4$-dimensional spacetime manifold of the universe. More precisely, given some natural definition of time, one can use this time function to slice up the spacetime (at least locally) into a set of hypersurfaces, which might each be thought of as `space'. Let us consider the initial space $M$ as being the $3$-sphere $S^3$ or the $3$-space $\mathbb{R}^3$ or a $3$-manifold corresponding to the aforementioned $3$-dimensional spatial section, and suppose that a cosmic phenomenon induces a topological change transforming $M$ into $M'$. Then, the $4$-dimensional spacetime manifold with past boundary the spacelike component $M$ and future spacelike boundary  $M'$ coincides with the $4$-dimensional cobordism $W$ bounded by $M$ on one end and $M'$ on the other. If this topological change is surgery then $M'=\chi(M)$ and the cobodism $W$ describes the global process of surgery as detailed in Section~\ref{global}. Moreover, as also explained in Section~\ref{global} and illustrated in Fig.~\ref{Cobordism2}, the temporal `slices' of the global process of surgery are  perpendicular crossections of the cobordism $W$.  

\subsection{Wormholes}\label{3d0WH_GC}
Einstein and Rosen~\cite{ERB} introduced in 1935 what would be called the `Einstein-Rosen bridge' as a possible geometric model that avoided singularities via a coordinate change of the Einstein field equations. This `bridge' evolved to the modern term `wormhole' introduced by Wheeler in~\cite{MisWhee}. Since then, a great variety of wormholes have been considered by the physics community.

From a geometrical point of view, Wheeler's diagram of a wormhole in~\cite{MisWhee} is a tunnel connecting two mouths. As he mentions, the number of space dimensions have been reduced from three to two, hence his diagram depicts the $3$-dimensional tunnel $D^1 \times S^2$ by a $2$-dimensional cylinder $D^1 \times S^1$ joining the two mouths. Similar representations are found in subsequent works~\cite{EmWorm,Inter}, where circular crossections along a cylinder represent $2$-spheres. For the purpose of our work, we consider that a wormhole is a higher dimensional tunnel $D^1 \times S^2$ joining two spherical regions of space. With this consideration at hand, we provide a topological description of wormhole formation and present a novel perspective on  its association with entanglement. See Sections~\ref{Ex_ER_EPR} and~\ref{WQS}.

More precisely, in Section~\ref{3d0WH_GC_Form} we describe wormhole formation via $3$-dimensional $0$-surgery and we pin down the core topology of this process, which can be seen independently of the physical theories of its formation. In Section~\ref{Ex_ER_EPR} we use our description in the context of the $ER=EPR$ hypothesis to view wormholes as a continuous process resulting from two entangled black holes. In Section~\ref{WQS} we present a way to associate a possibly entangled state with a wormhole.

\subsubsection{The topological process of wormhole formation}\label{3d0WH_GC_Form}
If one considers an initial $3$-manifold $M$ corresponding to space (as previously defined), then a wormhole joins the $3$-dimensional neighborhoods $S^0 \times D^3$  of two points in space via a tunnel $D^1 \times S^2$, as sketched in instance~(e) of Fig.~\ref{3D_ThickenedCores}~(2). This is, by definition, the effect of a $3$-dimensional $0$-surgery on $M$. Recall from Section~\ref{4Doutline} that the higher dimensional merging and recoupling which  produces the wormhole is not visible from the $3$-space $M$. For instance, let us consider a `mathematical' observer living in $M$, who is not subject to the restrictions of physical laws. The only difference for him is that, after surgery, he can exit from any point on the boundary of one $3$-ball and re-emerge from any point on the boundary of the other $3$-ball.

As also mentioned in Section~\ref{4Doutline}, this tunnel is a higher dimensional analogue of the cylinder seen during the formation of Falaco solitons where the $2$-dimensional neighborhoods $S^0 \times D^2$ of the two indentations are joined by the cylindrical vortex $D^1 \times S^1$. In fact, a possible connection between Falaco solitons and wormholes has already been mentioned by R.M. Kiehn. Namely, in~\cite{KiehnSmall} he conjectures that `the universal coherent topological features of the Falaco solitons can appear as cosmological realizations of Wheeler’s wormholes'. Our surgery description reinforces this connection.

\smallbreak
\begin{figure}[ht!]
\begin{center}
\includegraphics[width=15cm]{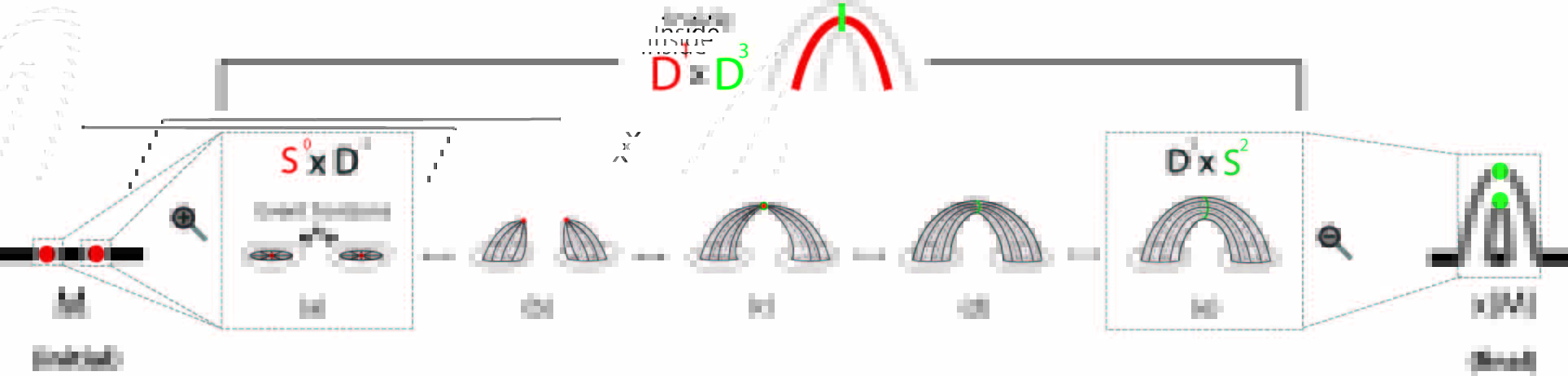}
\caption{Entangled black holes connected by a wormhole}
\label{WH_BH}
\end{center}
\end{figure}

Further, let us point out that the formation of certain wormholes are followed by their annihilation. For example, the dynamical evolution of the Schwarzschild wormhole starts with two singularities annihilating each other, thus creating the wormhole. The wormhole then grows in circumferences until its maximum size is reached, from which the wormhole starts contracting until it pinches off by creating two other singularities, see~\cite{Thorne}. The numerical calculations done in~\cite{EmWorm} show that this process is the same as the creation of a pair of Falaco solitons followed by their annihilation. Indeed the simulations of~\cite{EmWorm} amalgamate the instances of Fig.~\ref{3D_ThickenedCores}~(1) from left to right followed by the reverse process, which is made up of the instances of Fig.~\ref{3D_ThickenedCores}~(1) from right to left. Although we focus on wormhole formation, the topological process of their annihilation can be seen as the reverse process of their formation.  
\smallbreak
\textit{Viewing wormholes as the result of a $3$-dimensional $0$-surgery on $M$, allows us to apply the topological tools developed in previous sections, thus providing a simpler dynamical description in terms of hypersurfaces, which is coupled with a topological characterization of the resulting manifold.}
\smallbreak
Namely, as analyzed in Section~\ref{global}, the instances of the global process of this topological change from the initial manifold $M$ to the resulting manifold $\chi(M)$ make the spatial $4$-dimensional cobordism $W$ obtained by attaching a handle $D^1\times D^3$ to $M \times I$. The effect of this topological change on space $M$ can be characterized by determining the fundamental group of the resulting manifold $\chi(M)$ as shown in Section~\ref{Fundamen3d0}. Moreover, the global change of topology occuring during wormhole formation can now be also considered as a result of a continuous topological change of $3$-space. Namely, as seen in Section~\ref{3dMorse}, the local changes of surgery can be algebraically described by the hypersurfaces resulting from the local form of the corresponding Morse function. Further, the Morse function can be seen as a potential energy function whose gradient field controls the topological evolution of the surgery, recall Section~\ref{3DTemp}.

Following the core description of Section~\ref{3DTemp}, we can now think of a wormhole as a topological change starting with two sites in space (an $S^{0}$) which collapse to one site (the singular point) and re-emerge as a sphere $S^{2}$ (the core of the wormhole), see Fig.~\ref{3D_Cores}~(1). Inversely, if the core $S^2$ of a wormhole collapses then the handle (the wormhole itself) is removed and  we receive a new manifold with two special sites $S^0$. 

Note finally that wormhole formation can be viewed in the context of different physical theories. For instance, according to J.A. Wheeler, wormholes can be seen as resulting from quantum fluctuations at the Planck scale~\cite{Whee_QF}. Further, they can be seen as a result of entanglement~\cite{ER_EPR_,ER_EPR}. Our perspective describes the core topology of wormholes, independently of the physical theory in which it is viewed. In the next section we will see how the core description applied to the $ER=EPR$ hypothesis~\cite{ER_EPR_,ER_EPR} can provide a `classical path' to this quantum perspective.

\subsubsection{Wormholes as entangled black holes}\label{Ex_ER_EPR}
Our topological perspective may shed light on certain suggestions about quantum gravity and black holes. Specifically, we consider the $ER=EPR$ hypothesis, see~\cite{ER_EPR_,ER_EPR}, which states that an Einstein-Rosen bridge (that is, a wormhole) is equivalent to the quantum entanglement of two concentrated masses that each forms a respective black hole. This entanglement implies, by $ER=EPR$,  that the two black holes will not collapse individually, but rather form a single wormhole. The connectivity of the wormhole is, according to $ER=EPR$, a consequence of the quantum entanglement of the masses prior to the wormhole formation. See Section~\ref{WQS} for a specific discussion of this point.

Applying our description to the $ER=EPR$ hypothesis leads to conjecturing a classical counterpart to the formation of such a wormhole. In `classical' surgery  description the two sites in space (the $S^{0}$) are the centers of formation for the black holes that then become the core $S^2$ of the wormhole.

More precisely, the process starts with Fig.~\ref{WH_BH}~(initial) and ends with Fig.~\ref{WH_BH}~(final), where we show $1$-dimensional analogues of the $3$-dimensional instances ($\chi(M)$ is shown as the result of a $1$-dimensional $0$-surgery on line $M$ while $S^0 \times D^3$, $D^1 \times S^2$ are shown as  $S^0\times D^1$, $D^1 \times S^0$). In instances~(a) to~(e) of Fig.~\ref{WH_BH}, we zoom in the region where the local process of surgery happens and present $2$-dimensional analogues of the $3$-dimensional instances ($S^0 \times D^3$ is shown as  $S^0\times D^2$ and $D^1 \times S^2$ is shown as $D^1 \times S^1$).

In this scenario, as surgery happens within the event horizons of the black holes, the thickenings $D^3$ are inside the Schwarzschild radii of the black holes, see instance~(a) of Fig.~\ref{WH_BH}. In fact, the whole handle  $D^1\times D^3$ (see the upper part of Fig.~\ref{WH_BH} for its $1$-dimensional analogue), that contains all instances of the process, is within the event horizons of the black holes. The process brings the two black holes together to form a wormhole where their singularities $S^{0}$ have been transformed to the core $S^{2}$ of the wormhole, see instances~(b) to~(e) of Fig.~\ref{WH_BH}.  

Note that a quantum description of the formation of such a wormhole would directly pass from the initial instance at the beginning of the black hole formation to the final instance of the wormhole. Here however, $3$-dimensional $0$-surgery gives  a continuous description of the creation of this entangled pair of black holes forming the wormhole. This could be regarded as a possible classical path from the initial event to the wormhole. Inversely, the collapsing of the core of a wormhole can be seen as the disentanglement of the black hole pair.

\subsubsection{A possible entangled quantum state for wormholes}\label{WQS}
In this section, we present a way to associate a possibly entangled state with a wormhole that is coherently related to the $ER=EPR$ hypothesis. Recall that a \textit{cobordism} between two manifolds $M$ and $M'$ is a manifold $W$ of one higher dimension such that the boundary of $W$ is the union of $M$ and $M'.$ If $M'$ is empty, then we say that $W$ is a cobordism of $M$ to the empty manifold and, of course, this simply means that the boundary of $W$ is $M.$
View Figure~\ref{worm}. We illustrate a wormhole as a cobordism between an empty manifold and two spheres, drawn as circles in the figure.
For a spacetime wormhole, the spheres would each be two-dimensional (forming the event horizons of two black holes). This view of a wormhole fits precisely with the surgery description for the wormhole that we have given in this paper. 

\smallbreak
\begin{figure}[ht!]
\begin{center}
\includegraphics[width=3.5cm]{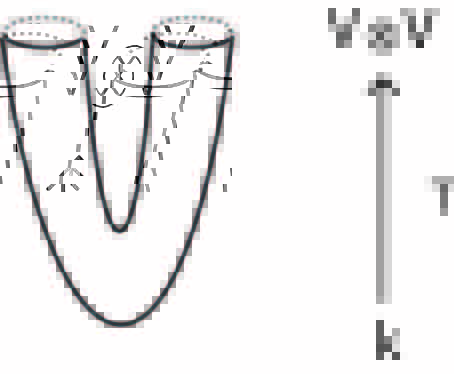}
\caption{The quantum state of a wormhole}
\label{worm}
\end{center}
\end{figure}

In Topological Quantum Field Theory one considers functors from the category of manifolds (as objects) and cobordisms (as morphisms) to the category of vector spaces and linear transformations. In this point of view a wormhole as in Figure~\ref{worm} would be sent by the functor to a linear mapping $$T: k \longrightarrow V \otimes V$$ where the two-sphere $S^2$ (depicted as a circle in the figure) maps to $V$, the disjoint union of the two two-spheres maps to $V \otimes V,$ and the empty object maps to the ground field $k.$
\smallbreak
\textit{Here $T$ is the map corresponding to the wormhole itself. With this point of view, we can see how an entangled quantum state can be associated with a wormhole.} 
\smallbreak
The possible state would occur with $k = \mathbb{C}$, the complex numbers, and $V$ a finite dimensional complex vector space associated with the two-sphere. Then $T(1)$ is a vector in the tensor product $V \otimes V,$ and is a possibly entangled quantum state to be associated with the wormhole. It remains to be seen whether properties of the wormhole resulting from the formation and amalgamation of two black holes imply the existence of such an entangled state. Nevertheless, the surgery picture of the wormhole as a cobordism is fundamental for this investigation. The possibly entangled state $T(1)$ can be interpreted as an element of the tensor product of Hilbert spaces associated with each black hole (represented by their respective event horizons). Thus this viewpoint also provides a framework in which to discuss the L. Susskind and J. Maldacena principle that quantum entanglement of two black holes should correspond to a wormhole that they form together. Here $T(1)$ would represent the quantum entanglement of the black holes. 
\smallbreak
At this writing we do not know a general condition in the spacetime manifold that would imply the entanglement of the state $T(1).$ For more information about topological quantum field theory, quantum entanglement and its relation to wormholes, see~\cite{WHE,EMLK,ER_EPR_,WITT}.

\subsection{Cosmic string black holes}\label{3d1BHWH}
Cosmic strings are hypothetical topological defects which may have formed in the early universe and are predicted by both quantum field theory and string theory models. Their existence was first contemplated by Tom Kibble~\cite{Kibble} in the 1970s. Then, in \cite{Hawk} S.W. Hawking estimated that a fraction of cosmic string loops can collapse to a small size inside their Schwarzschild radius, thus forming a black hole. As he mentions, under certain conditions, `one would expect an event horizon to form, and the loop to disappear into a black hole'.  We will call such black holes  `cosmic string black holes'. 

In Section~\ref{3d1BHWH_singA} we describe the formation of cosmic string black holes via $3$-dimensional $1$-surgery and present how this description proposes a conjecture resulting in the creation of a non-singular $3$-space. In Section~\ref{Ex_Poincare}, we examine the possible $3$-manifolds than can occur if such processes are followed, we focus on the example of the Poincaré dodecahedral space and discuss possible implications of observing such $3$-manifolds in our universe. In Section~\ref{StringBH} we use our description in the context of the $ER=EPR$ hypothesis, to present how the example of Section~\ref{Ex_ER_EPR} can be generalized to a cosmic string of entangled black holes forming a wormhole.

\subsubsection{Using topological surgery to avoid singularities}\label{3d1BHWH_singA}

Except from S.W. Hawking's original estimation in~\cite{Hawk}, other estimations of the fraction of cosmic string loops which collapse to form black holes have been made in subsequent works, see \cite{Cald} and \cite{McGibb}. While the details of the different estimations have no direct implications on this analysis, it is worth mentioning the following two statements. In \cite{Cald}, R.R. Caldwell and P. Casper point out that the loop `collapses in all three directions' and in  \cite{McGibb}, J.H. MacGibbon, R.H. Brandenberger and U.F. Wichosk give the following example for a collapsing symmetric string loop:  `For example, a planar circular string loop after a quarter period will collapse to a point and hence form a black hole.' 

Topologically, the aforementioned loop can be considered to be a solid torus $S^1 \times D^2$ embedded in the $3$-space $M$. The thickening $D^2$ can be considered to be very small, as the diameter of a cosmic strings is of the same order of magnitude as that of a proton, that is, $\approx$ 1 fm or smaller. The loop $S^1 \times D^2$ collapses to a small size inside its Schwarzschild radius, thus creating a black hole the center of which contains the singularity. In this scenario, $M$ becomes a singular manifold at that point. Physicists are undecided whether the prediction of this singularity means that it actually exists or that current knowledge is insufficient to describe what happens at such extreme density. This singularity can be avoided by considering that the collapsing of a cosmic string loop is followed by the uncollapsing of another cosmic string loop. In other words, we propose that the creation of a cosmic string black hole is a  $3$-dimensional $1$-surgery which changes the initial $3$-manifold $M$ to another $3$-manifold $\chi(M)$ by passing through a singular point.

\smallbreak
\begin{figure}[ht!]
\begin{center}
\includegraphics[width=14.5cm]{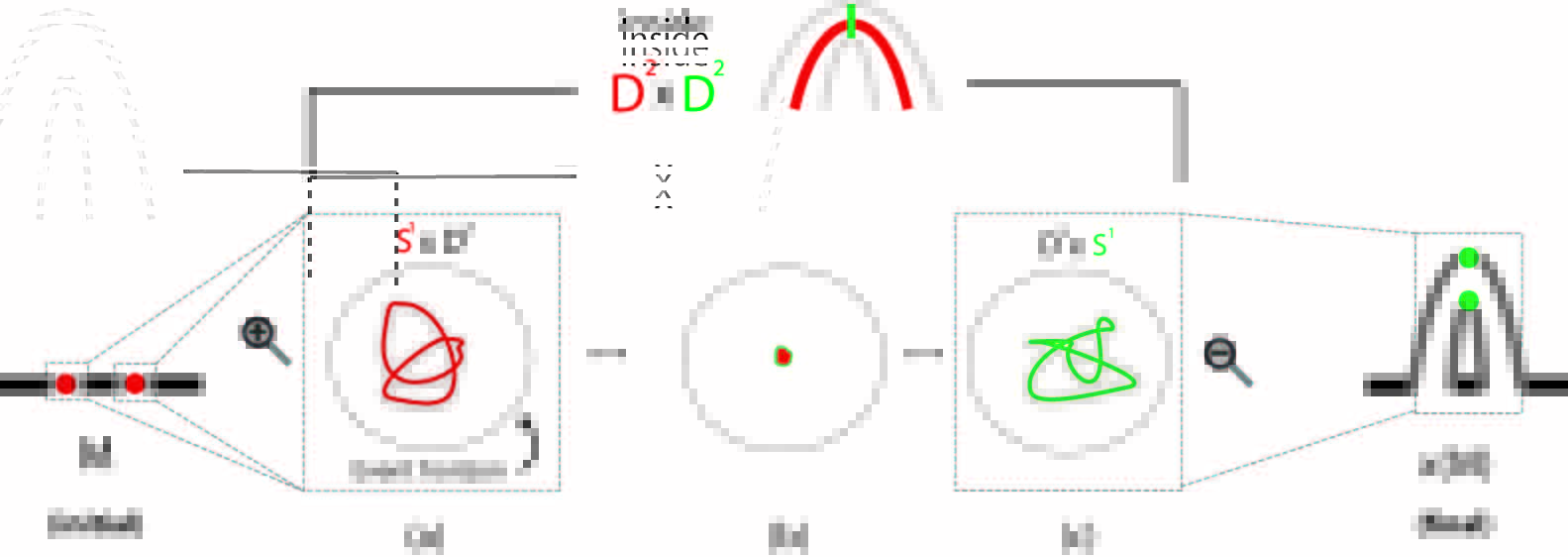}
\caption{3-dimensional 1-surgery inside the event horizon}
\label{Fig35}
\end{center}
\end{figure} 

The process starts with Fig.~\ref{Fig35}~(initial) and ends with Fig.~\ref{Fig35}~(final), where we show $1$-dimensional analogues of the $3$-dimensional instances ($\chi(M)$ is shown as the result of a $1$-dimensional $0$-surgery on line $M$, while $S^1 \times D^2$, $D^2 \times S^1$ are shown as  $S^0\times D^1$, $D^1 \times S^0$). In instances~(a) to~(c) of Fig.~\ref{Fig35}, we zoom in the region where the local process of surgery happens and we present a sketch of the $4$-dimensional process. More precisely, in instance~(a) of Fig.~\ref{Fig35}, we show a knotted embedding of the loop $S^1 \times D^2$. As we consider that the cosmic string has already shrunk to a radius smaller than its Schwarzschild radius, the event horizon is also shown in  instance~(a) of Fig.~\ref{Fig35}. Fig.~\ref{Fig35}~(b) shows the loop shrinking to the critical point which coincides with the physical singularity. After the collapsing, the process does not stop, but another  manifold, $D^{2} \times S^1$, which corresponds to another cosmic string loop, grows from the singular point of Fig.~\ref{Fig35}~(b). In instance~(c) of Fig.~\ref{Fig35}  we show the uncollapsing of the cosmic string $D^{2} \times S^1$ which transforms the initial manifold $M$ to $\chi(M)$, see Fig.~\ref{Fig35}~(final). As in previous section, the whole handle $D^2\times D^2$ (see the upper part of Fig.~\ref{Fig35} for its $1$-dimensional analogue), which contains all instances of the process, is within the event horizon of the black hole. 
\smallbreak
\textit{Thus, considering black hole formation as a knot surgery (or $3$-dimensional $1$-surgery) on a cosmic string loop allows us to go through the singular point of the black hole without having a singular manifold  in the end. Instead, we end up with a topologically new universe with a local topology change from the $3$-space $M$ to the $3$-space $\chi(M)$ and, as suggested in Fig.~\ref{Fig35}, this topology change happens within the event horizon.}
\smallbreak

In analogy with the previous section, the instances of this global process also make a spatial $4$-dimensional cobordism $W$ which, in this case, is obtained by attaching a handle $D^2\times D^2$ to $M \times I$, recall Section~\ref{global} for details. The effect of this topological change on space $M$ can be characterized by determining the fundamental group of the resulting manifold $\chi(M)$ as shown in Section~\ref{Fundamen3d1}. Further, as seen in Section~\ref{3dMorse}, the local changes of surgery can be algebraically described by the local form of the corresponding Morse function. As pointed out in Section~\ref{3DTemp}, the gradient of this function can be seen as a force which, in this case, corresponds to the string tension, which collapses the cosmic string, see \cite{Hawk} for details. 

Following the core description of Section~\ref{3DTemp}, we can now think of a cosmic string black hole as a knot surgery starting with a cosmic string in space (a possibly knotted $S^{1}$) which collapses to one site (the singular point) and re-emerges as another cosmic string (or possibly knotted $S^{1}$). See Fig.~\ref{3D_Cores}~(2) for a core view of the unknot and Fig.~\ref{Fig35} for the case of a non-trivial knot.

\subsubsection{New $3$-manifolds behind the event horizon and the Poincaré dodecahedral space}\label{Ex_Poincare}
As mentioned in Section~\ref{3D1Topo}, starting with $M=S^3$, knot surgery can produce every closed, connected, orientable $3$-manifold. This means that, if we consider the initial $3$-space to be $M=S^3$, our approach, apart from avoiding a singular $3$-space, also gives rise to a very large family of $3$-manifolds. One such $3$-manifold, which is of great interest to physicists, is the Poincaré dodecahedral space. This space can be described by taking a dodecahedron and identifying the opposite faces, as shown in Fig.~\ref{Poincare}, and has been proposed as a possible shape for the geometric universe, see~\cite{Weeks,Luminet,Levin}. As J-P. Luminet states in~\cite{JPrecent}, the 2015 release of Planck data remains consistent with more complex shapes, such as the spherical Poincaré dodecahedral space. 

From our viewpoint, this manifold is obtained by doing knot surgery on the trefoil knot with the right framing. See for example~\cite{PS}. Further details on the Poincaré dodecahedral space and its fundamental group are given in Appendix~\ref{surgeryonK}. Hence, in such a scenario, our approach suggests that:

\smallbreak
\textit{The shape of the universe came about via a knot surgery following the process showed in  Fig.~\ref{Fig35}, where the collapsed knot is a trefoil cosmic string.} 
\smallbreak

\begin{figure}[ht!]
\begin{center}
\includegraphics[width=4cm]{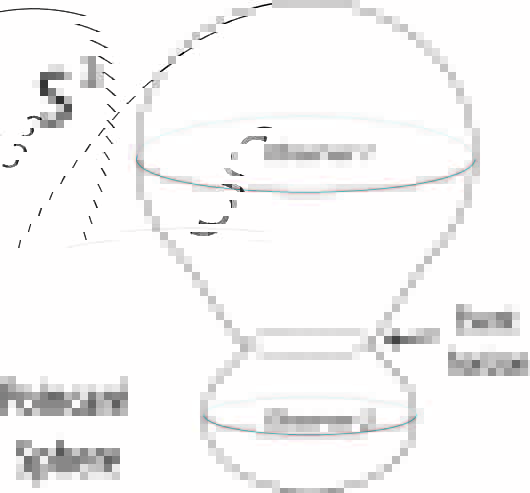}
\caption{Observer 1 and 2}
\label{BH_PoincaFlip}
\end{center}
\end{figure}

Let us now take this scenario further and suppose we have observers in an initial spherical universe $M=S^3$. After surgery, a `mathematical' observer would be able to see the Poincaré dodecahedral space and detect the topology change. From his point of view, he could exit from any point on the boundary of the thickened trefoil knot and re-emerge from any other point of its boundary. However, a physical observer, who is subject to the restrictions of physical laws, would only see the event horizon in which the trefoil cosmic string has collapsed. Let us call this observer, \textit{Observer~1}. After surgery, \textit{Observer~1} would see the same universe $S^3$, the only change being the formation of the spherical event horizon, shown as an $S^1$ in the lower dimensional analogue of Fig.~\ref{BH_PoincaFlip}. On the other side of the event horizon we can conjecture that a new universe has emerged in which new observers might evolve. Such an observer, say \textit{Observer~2}, will see a Poincaré dodecahedral space and the event horizon from the other side, unaware that the original $S^3$ universe is behind it, see Fig.~\ref{BH_PoincaFlip}. 
\smallbreak
\textit{Hence, finding the Poincaré dodecahedral space (or some other non-trivial $3$-manifold) in our universe may indicate that we are observers that evolved inside the event horizon of a collapsed trefoil cosmic string (or some other cosmic string).}
\smallbreak

\subsubsection{String of entangled black holes as a generalized wormhole}\label{StringBH}
Continuing the example of Section~\ref{Ex_ER_EPR}, we will discuss the relation of cosmic string black holes with the $ER=EPR$ hypothesis, see~\cite{ER_EPR_,ER_EPR}. As we will see, our topological perspective makes cosmic string black holes equivalent to wormholes made from a string of entangled black holes. 

To see this, we will first present a visualization, which will allow us to connect both types of surgery. Recall, from Section~\ref{3d0WH_GC}, how the core description of the process of $3$-dimensional $0$-surgery of Fig.~\ref{3D_Cores}~(1) fits the formation of a wormhole from an entangled pair of black holes. In the figure, the two centers of the black holes $S^0$ (in red) represent the boundary component $S^0 \times D^3$ of the handle $D^1 \times D^3$, while the wormhole core $S^2$ (in green) represents the other boundary component $D^1 \times S^2$ of $D^1 \times D^3$.  

In Fig.~\ref{EnStrings}~(1) we show both the initial and the final stage of the process in one instance. We further simplify Fig.~\ref{3D_Cores}~(1) by representing the boundary component $D^1 \times S^2$ of $D^1 \times D^3$ with $D^1$ instead of $S^2$. Hence, in Fig.~\ref{EnStrings}~(1) the two black holes $S^0$ (in red) come together to form the core $D^1$ of the wormhole (in green), which is also the core $D^1$ of the handle $D^1 \times D^3$ containing the temporal `slices' of the process of $3$-dimensional $0$-surgery.
 
\smallbreak
\begin{figure}[ht!]
\begin{center}
\includegraphics[width=7cm]{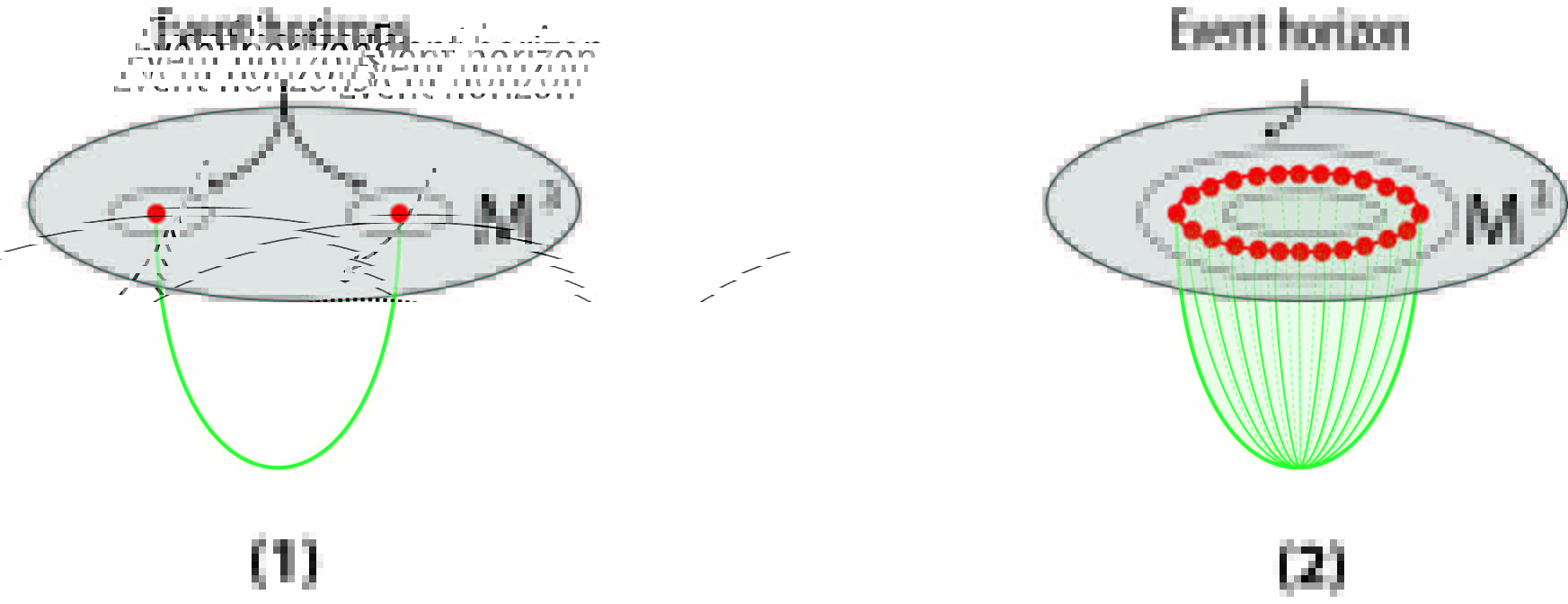}
\caption{\textbf{(1)} Pair of entangled black holes \textbf{(2)} String of entangled black holes}
\label{EnStrings}
\end{center}
\end{figure} 

Let us now consider a cosmic string loop made of several pairs of entangled concentrated masses. When each pair of masses collapses, they become connected by a wormhole, as shown in Fig.~\ref{EnStrings}~(1). Given that all these pairs of masses have started on the same cosmic string, the distinct wormholes merge and the entire collection of wormhole cores (the green arcs $D^1$ in Fig.~\ref{EnStrings}~(1)) forms a $2$-disc $D^2$, see Fig.~\ref{EnStrings}~(2), which is the core of the higher dimensional handle $D^2 \times D^2$ containing the temporal `slices' of the process of $3$-dimensional $1$-surgery. Note that, as $D^2$ cups off a circle while $D^1$ joins two points, one can rotate Fig.~\ref{EnStrings}~(1) to receive Fig.~\ref{EnStrings}~(2).
\smallbreak
\textit{Hence, a cosmic string black hole can be seen as a collection of Einstein-Rosen bridges, which generalizes having a separate bridge for each pair of entangled black holes.}
\smallbreak
The process of surgery amalgamates these bridges to form a new $3$-manifold resulting from surgery on the cosmic string. The effect of knot surgery is that, from any black hole location on the cosmic string to any other, there is a `bridge' through the new $3$-manifold.

\section{Conclusions}
In this paper, we use tools from Morse theory and algebraic topology to describe the process of topological surgery both locally and globally. This approach  provides continuous paths to wormhole and cosmic string black hole formations. Adding the $ER=EPR$ hypothesis, we also describe the entanglement of a pair (or a string) of black holes, thus binding the quantum connectivity of space with the rich structure of $3$- and $4$-dimensional manifolds. 

Our knot surgery hypothesis for cosmic strings suggests that there should be a generalization of the $ER=EPR$ hypothesis to relate quantum entanglement with more general cobordisms and in particular with the new $3$-manifold structure that results from cosmic string collapse. This will be the subject of a sequel to the present paper. 

We also describe how we can receive the Poincaré dodecahedral space and a plethora of non-trivial $3$-manifolds from the formation of cosmic string black holes. In our description, the formation of such a black hole does not result in a singular $3$-manifold but rather a topologically new universe with a local topology change of $3$-space. As the proposed process avoids the singularity problem, we are currently working on the physical implications and the potential observational evidence of this novel topological perspective.

\section*{Acknowledgments}
Antoniou's work was partially supported by the Papakyriakopoulos scholarship  which was awarded by the Department of Mathematics of the National Technical University of Athens. Kauffman's work was supported by the Laboratory of Topology and Dynamics, Novosibirsk State University (contract no. 14.Y26.31.0025 with the Ministry of Education and Science of the Russian Federation).

\newpage
\appendix
\section{Appendices}\label{appendix}
\subsection{Visualizing surgery using stereographic projection}\label{2D3DStereo}
We present here a way to visualize the initial and the final instances of $m$-dimensional surgery in $\mathbb{R}^{m}$ and discuss the cases of $m=2$ and $m=3$.

Let us first be reminded from Section~\ref{local} that, if we glue together the two $m$-manifolds with boundary involved in the process of $m$-dimensional $n$-surgery along their common boundary using the standard embedding we obtain the $m$-sphere $S^m$. The idea of our proposed  visualization of surgery is that while $S^m$ is embedded in $\mathbb{R}^{m+1}$, it can be stereographically projected to $\mathbb{R}^{m}$. Hence, for every $m$, one can visualize the initial and the final instances of the local process of $m$-surgery one dimension lower. In the following examples we deliberately did not  project the intermediate instances, as this cannot be done without self-intersections.

\subsubsection{Visualizing 2-dimensional 0-surgery in \texorpdfstring{$\mathbb{R}^2$}{R2} }\label{2DStereo}
For $m=2$ and $n=0$, the initial and final instances of $2$-dimensional $0$-surgery that make up $S^2$ are shown in Fig.~\ref{2D_Proj}~(1). If we remove the point at infinity, we can project the points of $S^2 \setminus \{\infty\}$ on $\mathbb{R}^2$ bijectively. We will use two different projections for two different choices for the point at infinity. The first one is shown in Fig.~\ref{2D_Proj}~({2\textsubscript{a}}) where the point at infinity is a point of the core $S_2^1$ of $D^{1}\times S_2^{1}$. In this case, the two great circles $S_2^1 = \ell \cup \{\infty\}$ and $\ell' \cup \{\infty\}$ of $S^2$ are projected on the two perpendicular infinite lines $\ell$ and  $\ell'$ in ${\mathbb R}^2$. In the second one, shown in Fig.~\ref{2D_Proj}~({2\textsubscript{b}}), the point at infinity is the center of one of the two discs $S_1^0\times D^{2}$. In this case the great circle $\ell' \cup \{\infty\}$ in $S^2$ is, again, projected to the  infinite line $\ell'$ in ${\mathbb R}^2$ but the great circle $S_2^1 = \ell$ is now projected to the  circle $\ell$ in ${\mathbb R}^2$. 

\smallbreak
\begin{figure}[ht!]
\begin{center}
\includegraphics[width=14cm]{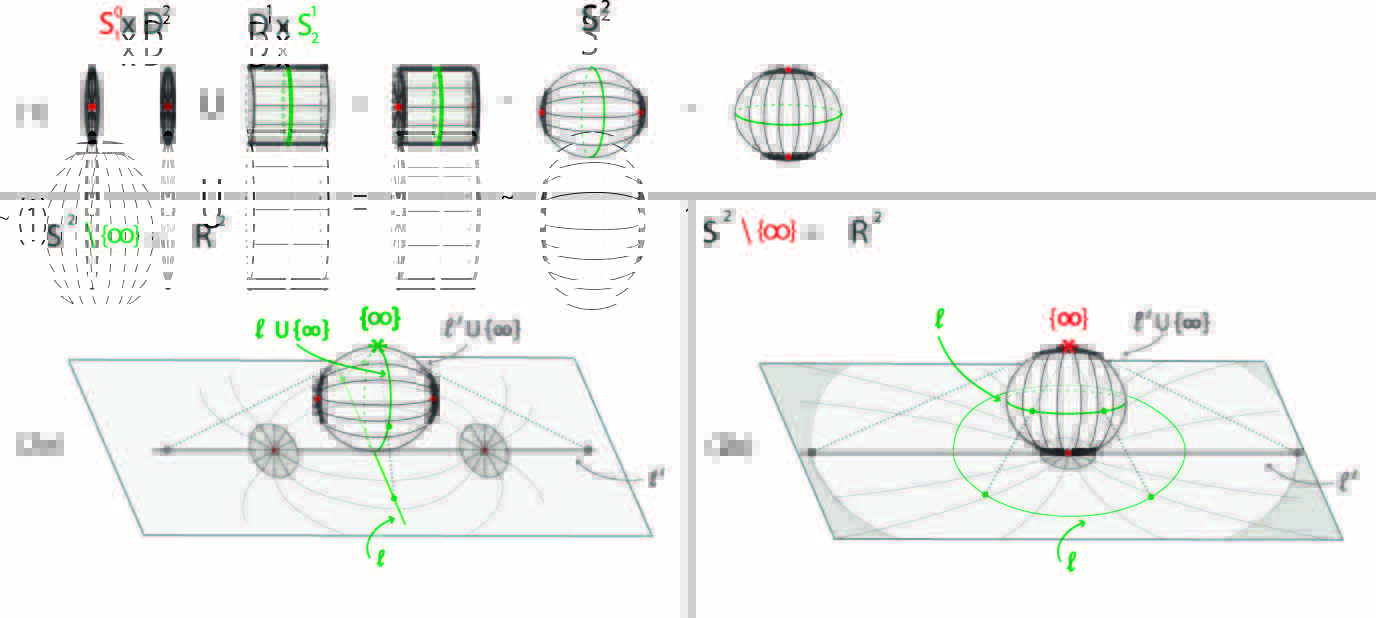}
\caption{ \textbf{(1)}  $(S_1^0\times D^{2}) \cup (D^{1}\times S_2^{1}) =S^{2}$ (\textbf{{2\textsubscript{a}})} First projection of $S^2 \setminus \{\infty\}$ to $\mathbb{R}^2$ \textbf{({2\textsubscript{b}})} Second projection of $S^2 \setminus \{\infty\}$ to $\mathbb{R}^2$}
\label{2D_Proj}
\end{center}
\end{figure}

As mentioned in Section~\ref{local}, the one dimension higher of the disc $D^{m+1}$  leaves room for the process of $m$-dimensional surgery to take place continuously. For $2$-dimensional surgery, the third dimension allows the two points of the core $S_1^0$ to touch at the critical point, recall Fig.\ref{3D_ThickenedCores}~(1). Using the two stereographic projections discussed above and shown again in Fig.~\ref{2D_Decomp}~(1\textsubscript{a}) and~(1\textsubscript{b}), we present in Fig.~\ref{2D_Decomp}~(2\textsubscript{a}) and~(2\textsubscript{b}) two different local visualizations of $2$-dimensional surgery in $\mathbb{R}^2$. Note that in Fig.~\ref{2D_Decomp}~({1\textsubscript{b}}) and~({2\textsubscript{b}}), the red dashes show that all lines converge to the point at infinity which is the center of the decompactified disc and one of the points of $S_1^0$. The process of $2$-dimensional $0$-surgery starts with either one of the first instances of Fig.~\ref{2D_Decomp}~(2\textsubscript{a})  and~(2\textsubscript{b}). Then the centers of the two discs $S_1^0\times D^{2}$ collapse to the critical point which is shown with increased transparency to remind us that this happens in one dimension higher, see the second instances of either Fig.~\ref{2D_Decomp}~(2\textsubscript{a}) or~(2\textsubscript{b}). Finally the cylinder $D^{1}\times S_2^{1}$ uncollapses, as illustrated in the last instances of Fig.~\ref{2D_Decomp}~(2\textsubscript{a}) and~(2\textsubscript{b}). Clearly, the reverse processes provide visualizations of $2$-dimensional $1$-surgery.

\smallbreak
\begin{figure}[ht!]
\begin{center}
\includegraphics[width=14cm]{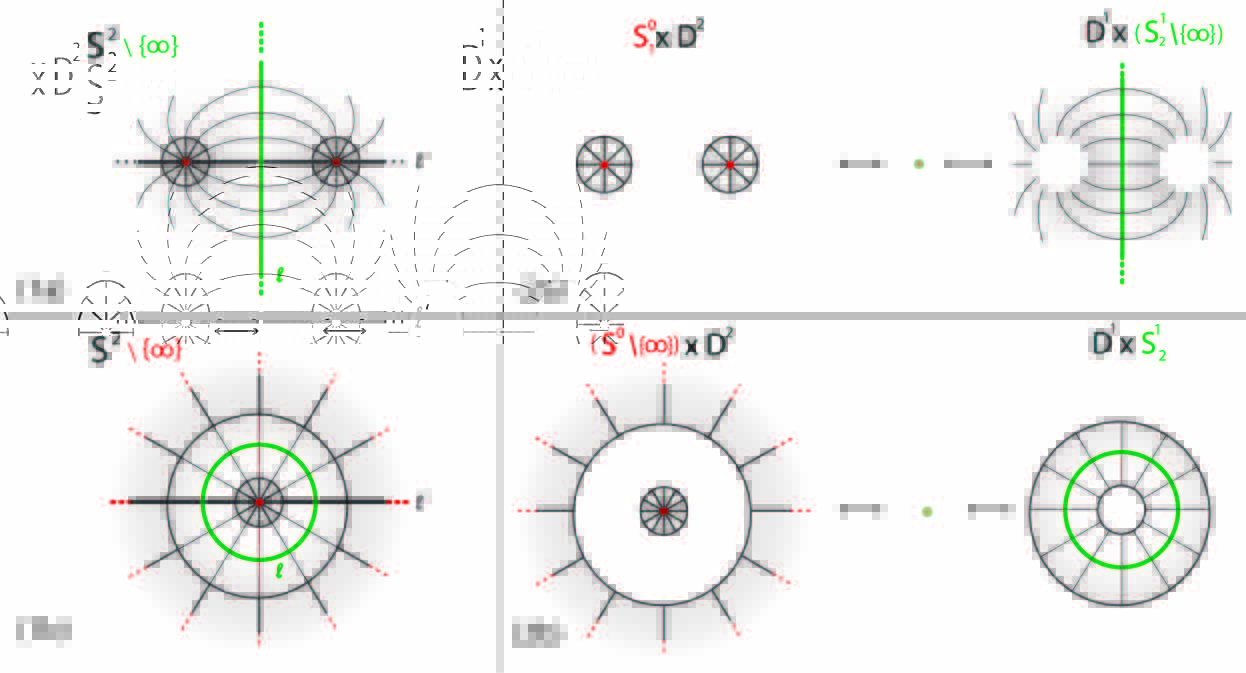}
\caption{\textbf{({1\textsubscript{a}})} First projection \textbf{({1\textsubscript{b}})} Second projection. 
\textbf{({2\textsubscript{a}})}, \textbf{({2\textsubscript{b}})} Corresponding initial and final instances of $2$-dimensional $0$-surgery in $\mathbb{R}^2$ }
\label{2D_Decomp}
\end{center}
\end{figure}

\subsubsection{Visualizing 3-dimensional surgery in  \texorpdfstring{$\mathbb{R}^3$}{R3}}\label{3DStereo}
Moving up one dimension, the initial and final instances of $3$-dimensional surgery form $S^{3}=\partial D^{4}$. Since, now, $S^3 \setminus \{\infty\}$ can be projected on ${\mathbb R}^3$ bijectively, we will present a new  way of visualizing $3$-dimensional surgery in ${\mathbb R}^3$ by rotating appropriately the projections of the initial and  final instances of $2$-dimensional $0$-surgery in ${\mathbb R}^2$.

The underlying idea is that, in general, $S^n$ which is embedded in ${\mathbb R}^{n+1}$ can be obtained by a 180° rotation of $S^{n-1}$, which is embedded in ${\mathbb R}^n$. So, a 180° rotation of $S^{0}$ around an axis bisecting the interval of the two points (e.g. line $\ell$ in  Fig.~\ref{2D_Proj}~({2\textsubscript{a}})) gives rise to $S^1$ (which is $\ell' \cup \{\infty\}$ in  Fig.~\ref{2D_Proj}~({2\textsubscript{a}})), while a 180° rotation of $S^{1}$ around any diameter gives rise to $S^2$. For example, in Fig.~\ref{2D_Proj}~({2\textsubscript{b}})), a 180° rotation of $\ell' \cup \{\infty\}$ around the north-south pole axis results in the $2$-sphere shown in the figure. Now, the creation of $S^3$ (which is embedded in ${\mathbb R}^4$) as a rotation of $S^2$ requires a fourth dimension in order to be visualized. Instead we can obtain its stereographic projection in ${\mathbb R}^3=S^3 \setminus \{\infty\}$ by rotating the stereographic projection of $S^2 \setminus \{\infty\}={\mathbb R}^2$. Indeed, a 180° rotation of the plane around any line in the plane gives rise to the $3$-space.

As we will see, each type of $3$-dimensional surgery  corresponds to a different rotation, which, in turn, corresponds to a different decomposition of $S^{3}$. As we consider here two kinds of projections of $S^2 \setminus \{\infty\}$ in ${\mathbb R}^2$, see Fig.~\ref{2D_Decomp}~({1\textsubscript{a}}) and~({1\textsubscript{b}}), these give rise to two kinds of decompositions of $S^{3}$ via rotation, see Fig.~\ref{3D_Decomp_F}~({1\textsubscript{a}}) and~({1\textsubscript{b}}). Each decomposition, now, leads to the visualizations of both types of $3$-dimensional surgery.

\smallbreak
\begin{figure}[ht!]
\begin{center}
\includegraphics[width=14cm]{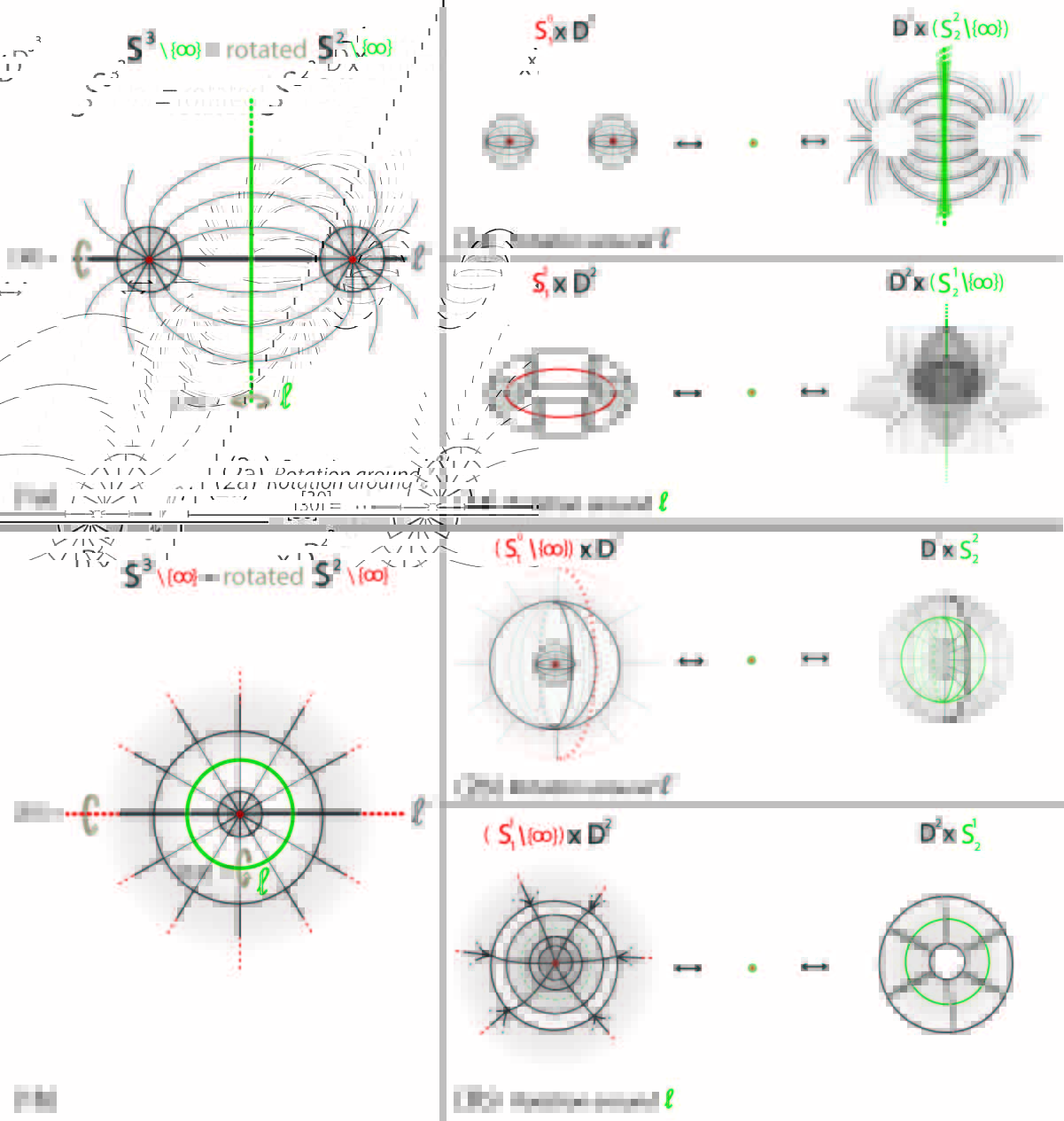}
\caption{ \textbf{({1\textsubscript{a}})},\textbf{({1\textsubscript{b}})} Representations of $S^3 \setminus \{\infty\}$ as rotations of $S^2 \setminus \{\infty\}$  using first and second projection. \textbf{({2\textsubscript{a}})},\textbf{({2\textsubscript{b}})} $3$-dimensional $0$-surgery in $\mathbb{R}^3$  using  first and second  projection. \textbf{({3\textsubscript{a}})}, \textbf{({3\textsubscript{b}})} $3$-dimensional $1$-surgery in $\mathbb{R}^3$    using first and second projection.}
\label{3D_Decomp_F}
\end{center}
\end{figure}

Let us start with the first projection. In Fig.~\ref{3D_Decomp_F}~({1\textsubscript{a}}), we show this decompactified view in ${\mathbb R}^2$ and the two axes of rotation $\ell'$ and $\ell$. As we will see, a rotation around axis $\ell'$ induces $3$-dimensional $0$-surgery in ${\mathbb R}^3$ while a rotation around axis $\ell$ induces $3$-dimensional $1$-surgery in ${\mathbb R}^3$.

Namely, in the case of $3$-dimensional $0$-surgery, a \textit{horizontal} rotation of 180° around axis $\ell'$ transforms the two discs $S_1^0\times D^2$ of Fig.~\ref{2D_Decomp}~({2\textsubscript{a}}) (the first instance  of $2$-dimensional $0$-surgery) to the two $3$-balls $S_1^0\times D^3$ of Fig.~\ref{3D_Decomp_F}~(2\textsubscript{a}) (the first instance of $3$-dimensional $0$-surgery). After the collapsing of the centers of the two $3$-balls $S_1^0\times D^3$, the rotation transforms the decompactified cylinder $D^{1}\times (S_1^{1}\setminus \{\infty\})$ of Fig.~\ref{2D_Decomp}~({2\textsubscript{a}}) (the last instance  of $2$-dimensional $0$-surgery) to the decompactified thickened sphere $D^{1}\times (S_2^{1}\setminus \{\infty\})$ of Fig.~\ref{3D_Decomp_F}~(2\textsubscript{a}) (the last instance of $3$-dimensional $0$-surgery). Indeed, the rotation of line $\ell$ along $\ell'$ creates the green plane that cuts through ${\mathbb R}^3$ and separates the two resulting 3-balls $S_1^0\times D^3$. This plane is shown in green in the last instance of Fig.~\ref{3D_Decomp_F}~(2\textsubscript{a}) and it is the decompactified view of the sphere $S_2^2$ in  ${\mathbb R}^3$. Note that it is thickened by the arcs connecting the two discs $S_1^0\times D^2$ which have also been rotated.

Similarly, in the case of $3$-dimensional $1$-surgery, a \textit{vertical} rotation of 180° around axis $\ell$ transforms the two discs $S_1^0\times D^2$ (the first instance  of $2$-dimensional $0$-surgery shown in Fig.~\ref{2D_Decomp}~({2\textsubscript{a}})) to the solid torus $S_1^1\times D^2$ (the first instance of $3$-dimensional $1$-surgery), see Fig.~\ref{3D_Decomp_F}~(3\textsubscript{a}). After the collapsing of the (red) core $S_1$ of $S_1^1\times D^2$, the rotation transforms the decompactified cylinder $D^{1}\times (S_1^{1}\setminus \{\infty\})$ of Fig.~\ref{2D_Decomp}~({2\textsubscript{a}}) (the last instance  of $2$-dimensional $0$-surgery) to the decompactified solid torus $D^{2}\times (S_1^{1}\setminus \{\infty\})$ of Fig.~\ref{3D_Decomp_F}~(3\textsubscript{a}) (the last instance of $3$-dimensional $1$-surgery). Indeed, each of the arcs $D^1$ connecting the two discs $S_1^0\times D^2$ generates through the rotation a $2$-dimensional disc $D^2$, and the set of all such discs are parametrized by the points of the line $\ell$ in ${\mathbb R}^3$.

In both cases, in Fig.~\ref{3D_Decomp_F}~({1\textsubscript{a}}), $S^3$ is presented as the result of rotating the $2$-sphere $S^2 = {\mathbb R}^2 \cup \{\infty\}$. For $3$-dimensional $0$-surgery, $S^2$ is rotated about the circle $\ell' \cup \{\infty\}$  where $\ell'$ is a straight horizontal line in ${\mathbb R}^2$. The resulting decomposition of $S^3$ is $S^3 =(S_1^0\times D^3) \cup (D^1 \times S_2^2)$, a thickened sphere with two $3$-balls glued along the boundaries, which is visualized as $S^3\setminus \{\infty\} =(S_1^0\times D^3) \cup (D^1 \times (S_2^2\setminus \{\infty\}))$. For $3$-dimensional $1$-surgery, $S^2$ is rotated about the circle $\ell \cup \{\infty\}$  where $\ell$ is a straight vertical line in ${\mathbb R}^2$. The resulting decomposition of $S^3$ is $S^3 =(S_1^1\times D^2) \cup (D^2 \times S_2^1)$, two solid tori glued along their common boundary,  which is visualized as $S^3\setminus \{\infty\} =(S_1^1\times D^2) \cup (D^2 \times (S_2^1\setminus \{\infty\}))$. 
 
Analogously, starting with the second projection of Fig.~\ref{2D_Decomp}~({1\textsubscript{b}}), the same rotations induce each type of $3$-dimensional surgery and their corresponding decompositions of $S^3$, see Fig.~\ref{3D_Decomp_F}~({1\textsubscript{b}}). More precisely, a \textit{horizontal} rotation of the instances of Fig.~\ref{2D_Decomp}~({2\textsubscript{b}}) by 180° around axis $\ell'$ induces the initial and final instances of $3$-dimensional $0$-surgery visualized in ${\mathbb R}^3$, see Fig.~\ref{3D_Decomp_F}~({2\textsubscript{b}}). The $3$-sphere $S^3$ is now visualized as $S^3\setminus \{\infty\} =((S_1^0\setminus \{\infty\})\times D^3) \cup (D^1 \times S_2^2)$, a thickened sphere union two $3$-balls with the center of one of them removed (being the point at infinity).

Similarly, a  rotation of the instances of Fig.~\ref{2D_Decomp}~({2\textsubscript{b}}) by 180° around the (green) \textit{circle} $\ell$ induces the initial and final instances of $3$-dimensional $1$-surgery visualized in ${\mathbb R}^3$, see Fig.~\ref{3D_Decomp_F}~({3\textsubscript{b}}). Note that $\ell$ is now a circle and not a (vertical) line. The easiest part for visualizing this rotation is the rotation of the middle annulus of Fig.~\ref{3D_Decomp_F}~({1\textsubscript{b}}) which gives rise to the solid torus $D^2 \times S_2^1$ in  Fig.~\ref{3D_Decomp_F}~({3\textsubscript{b}}). The same rotation of the two remaining discs around $\ell$ can be visualized as follows: each radius of the inner disc lands from above the plane on the corresponding radius of the outer disc. At the same time, that radius of the outer disc lands on the corresponding radius of the inner disc from underneath the plane. So, the two corresponding radii together have created by rotation an annular ring around $\ell$. Note that the red center of the inner disc will land on all points at infinity, creating a half-circle from above and, at the same time, all points at infinity land on the center of the inner disc and create a half-circle from below. Glued together, the two half-circles create a (red) circle. Now, the set of all annular rings around $\ell$ and parametrized by $\ell$ make up the complement solid torus $S_1^1\setminus \{\infty\}\times D^2$ whose core is the aforementioned red circle. The $3$-sphere $S^3$ is visualized through this rotation as  $S^3\setminus \{\infty\} =((S_1^1\setminus \{\infty\})\times D^2) \cup (D^2 \times S_2^1)$, the decompactified union of two solid tori.

Finally, it is worth pinning down that the two types of visualizations presented above are related. Indeed, the $(D^1 \times (S_2^2\setminus \{\infty\}))$ shown in the rightmost instance of Fig.~\ref{3D_Decomp_F}~({2\textsubscript{a}}) is the decompactified view of the $(D^1 \times S_2^2)$ shown in the rightmost instance of Fig.~\ref{3D_Decomp_F}~({2\textsubscript{b}}). Likewise, the $(D^2 \times (S_2^1\setminus \{\infty\}))$ shown in the rightmost instance of Fig.~\ref{3D_Decomp_F}~({3\textsubscript{a}}) is the decompactified view of the the solid torus $(D^2 \times S_2^1)$ shown in the rightmost instance of Fig.~\ref{3D_Decomp_F}~({3\textsubscript{b}}). Further, the $(S_1^0\setminus \{\infty\})\times D^3$ and $(S_1^1\setminus \{\infty\})\times D^2$ shown in the leftmost instances of Fig.~\ref{3D_Decomp_F}~({2\textsubscript{b}}) and~({3\textsubscript{b}}) are the decompactified views of $S_1^0\times D^3$ and $S_1^1\times D^2$ shown in the leftmost instances of Fig.~\ref{3D_Decomp_F}~({2\textsubscript{a}}) and~({3\textsubscript{a}}) respectively.

\subsection{The fundamental group}\label{Fundappendix}
The fundamental group is one of the most significant algebraic constructions for obtaining topological information about a topological space.  It is a topological invariant: homeomorphic topological spaces have the same fundamental group.


\subsubsection{Topological spaces}\label{TSG}
The fundamental group of an arbitrary topological space $X$ with reference to a basepoint $p$ in that space is denoted as $\pi_1(X, p)$ and is generated by continuous paths in $X$ that start and end at $p$ (loops at $p$ in $X$). These loops in $X$ are taken up to an equivalence relation called \textit{homotopy} where two loops $\alpha$ and $\beta$ are said to be homotopic if one can be continuously deformed to the other. In other words, there is a continuous family of loops based at $p$ starting with $\alpha$ and ending with $\beta$, which is usually parametrized in the unit interval $[0, 1]$. The collection of all loops based at $p$, taken up to homotopy, forms a group where the inverse of a loop is the loop obtained by reversing the direction of its parametrization. 

For example, the Euclidean space $\mathbb{R}^n$ for $n \geq 1$ and the $n$-sphere $S^n$ for $n \geq 2$ have trivial fundamental groups, as all loops in these manifolds can be shrunk to a point. However, the fundamental group of $S^1$ is not trivial. In fact it is the infinite cyclic group generated by a single element. It contains all loops which wind around the circle a given number of times, which can be positive or negative, depending on the winding direction. The `product' of a loop which winds around the circle $m$ times with another that winds around $n$ times is a loop which winds around $m + n$ times. So the fundamental group of $S^1$ is isomorphic to $\mathbb{Z}$, the additive group of integers.

An interesting property is that the fundamental group of a product space $X \times Y$ is the direct product of the fundamental groups of $X$ and $Y$: $\pi_1(X \times Y)=\pi_1(X ) \times \pi_1(Y)$. Note that given groups $G$ and $H$ we define the direct product $G \times H$ as the set of ordered pairs $(g,h)$ with $g \in G$ and $h \in H$ and we define $(g,h)\cdot(g',h') = (gg',hh')$. With this structure, $G \times H$ is a new group with identity $e = (e_G, e_H)$. This property allows us to calculate the fundamental group of more complicated spaces. For instance, the fundamental group of the $n$-dimensional torus $T^n={S^1 \times ... \times S^1}$ which is the product of $n$ circles is $\pi_1(T^n)={\pi_1(S^1) \times ... \times \pi_1(S^1)}=\mathbb{Z}^n$. Similarly, the fundamental group of the $3$-manifold $S^1 \times S^2$ is $\pi_1(S^1 \times S^2)=\mathbb{Z} \times \{1\}=\mathbb{Z}$.

\subsubsection{The blackboard framing}\label{Blackboard}
In addition to the definition given in Section~\ref{3D1Knot}, a framing of a knot can be also viewed as a choice of non-tangent vector at each point of the knot. The \textit{blackboard framing} of a knot is the framing where each of the vectors points in the vertical direction, perpendicular to the plane, see Fig.~\ref{3D_31_Framing}~(2). The blackboard framing of a knot gives us a well-defined general rule for determining the framing of a knot diagram. Here the knot diagram is taken up to regular isotopy, namely up to Reidemeister II and III moves (see~\cite{Ad} for details on the Reidemeister moves). We use the curling in the diagram to determine the framing for an embedding corresponding to the knot, as will be explained below. Note that once we have chosen a longitude for the blackboard framing we can allow Reidemeister I moves (that might eliminate a curl) and just keep track of how the longitude now winds on the torus surface. 


\begin{wrapfigure}{L}{0.75\textwidth} 
\centering
\includegraphics[width=0.7\textwidth]{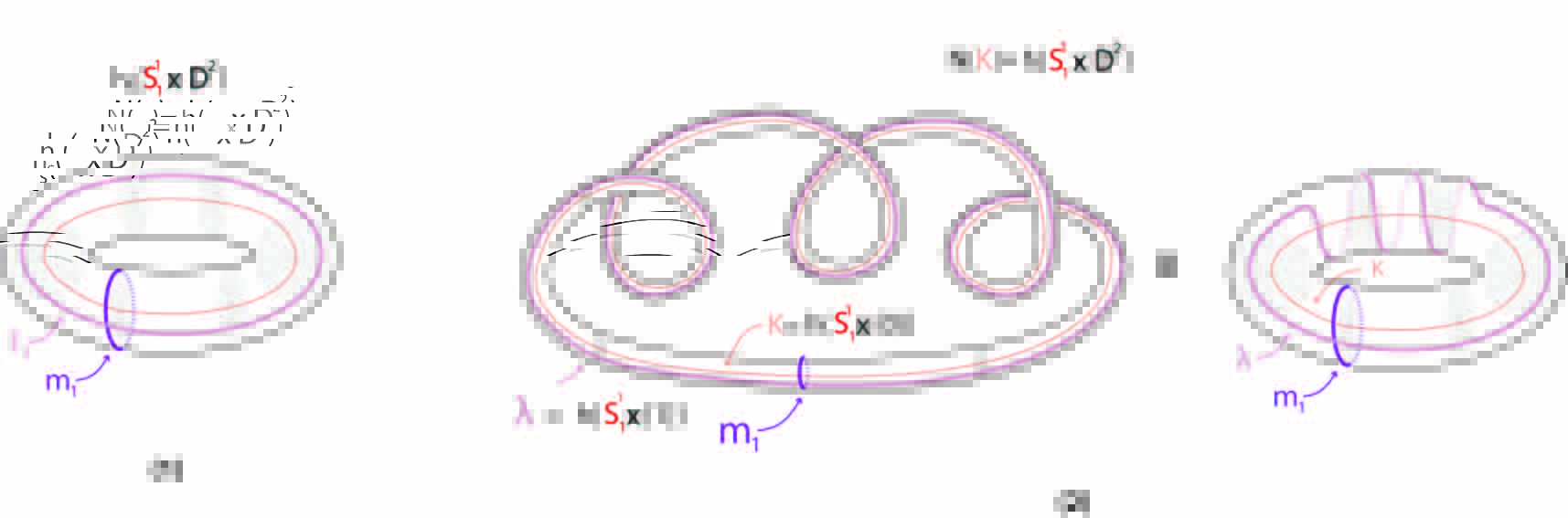}
\caption{\textbf{(1)} Longitude $l_1$ \textbf{(2)} Longitude $\lambda=l_1+3 \cdot m_1$ }
\label{3D_31_13}
\end{wrapfigure}

An example is shown in Fig.~\ref{3D_31_13}~(2). This case  corresponds to a non-trivial embedding $N(K)=h(S_1^1 \times D^2)$ where both the knot $K$ and the longitude $\lambda$ perform three curls. As also shown in Fig.~\ref{3D_31_13}~(2),  
there is an isotopic embedding of $N(K)$ where the surgery curve $K$ at the 
core of $N(K)$ is unknotted while the curls of $\lambda$ have become windings around $K$. This allows us to express $\lambda$ in terms of the unknotted longitude $l_1$ of the trivial embedding shown in Fig.~\ref{3D_31_13}~(1). Namely, as $\lambda$  performs $3$ revolutions around a meridian, it can be expressed as $\lambda=l_1+3 \cdot m_1$, see Fig.~\ref{3D_31_13}~(2).

More generally, if a longitude $\lambda$ performs $p$ revolutions around a meridian, it can be expressed as $\lambda=l_1+p \cdot m_1$. The induced `gluing' homeomorphism along the common boundary $S_1^1\times S_2^1$ maps each $\lambda$ of $V_1$ to a meridian of $V_2$, hence $h(l_1+p.m_1)=m_2$, while the meridians of $V_1$ are mapped to longitudes of $V_2$, hence $h(m_1)=h_{s}(m_1)=l_2$. The resulting manifolds obtained by doing a $3$-dimensional $1$-surgery on $M=S^3$ using such framings on the unknot are the lens spaces $L(p,1)$. For $p=0$ we have $h(l_1)=h_s(l_1)=m_2$ and $L(0,1)=S^2 \times S^1$, which was the case presented in Section~\ref{Types3D}. For more details on lens spaces see, for example, \cite{PS}. Note that the multiple of the meridian $p$ is also called the framing number.


\begin{wrapfigure}{L}{0.65\textwidth}
\centering
\includegraphics[width=0.6\textwidth]{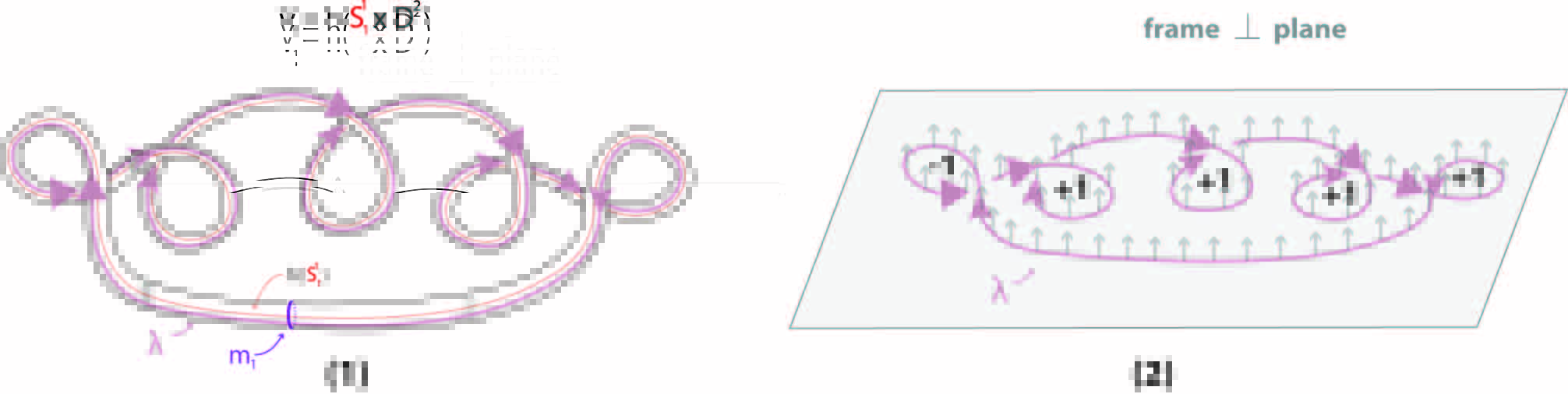}
\caption{\textbf{(1)} Isotopy of $\lambda$ \textbf{(2)} Blackboard framing of $\lambda$}
\label{3D_31_Framing}
\end{wrapfigure}

Recall that in Fig.~\ref{3D_31_13} the framing was $p=3$, as $\lambda$ performs $3$ revolutions. However, determining the framing of a knot diagram requires a well-defined general rule. For instance, that rule should give the same framing $p=3$ for the isotopic curve shown in Fig.~\ref{3D_31_Framing}~(1). This general rule is to  take the natural framing of a knot to be its \textit{writhe}, which is the total number of positive crossings minus the total number of negative crossings. The rule for the sign of a crossing is the following: as we travel along the knot, at each crossing we consider a counterclockwise rotation of the overcrossing arc. If we reach the undercrossing arc and are pointing the same way, then the crossing is positive, see Fig.~\ref{3D_31_Framing}~(2). Otherwise, the crossing is negative, see also Fig.~\ref{3D_31_Framing}~(2).

Using this convention we can calculate $\lambda$ and be sure that isotopic knots will have the same framing. For instance, in Fig.~\ref{3D_31_Framing}~(1), the framing number is the writhe of the knot diagram which is $p=Wr(\lambda)=4-1=3$.


\subsubsection{The knot group}\label{KG}
The \textit{fundamental group of a knot $K$} (or the \textit{knot group}) is defined as the fundamental group of the complement of the knot in $3$-dimensional space (considered to be either $\mathbb{R}^3$ or $S^3$) with a basepoint $p$ chosen arbitrarily in the complement. The group is denoted $\pi_1(K)$ or $\pi_1(S^3 \setminus N(K))$, where $N(K)$ is a tubular neighborhood of the knot $K$. To describe this group, it is useful to have the concept of the longitude and meridian elements of the fundamental group of a knot. The longitude and the meridian are loops in the knot complement that are on the surface of a torus, the boundary of $N(K)$. 


\begin{wrapfigure}{R}{0.4\textwidth} 
\centering
\includegraphics[width=0.35\textwidth]{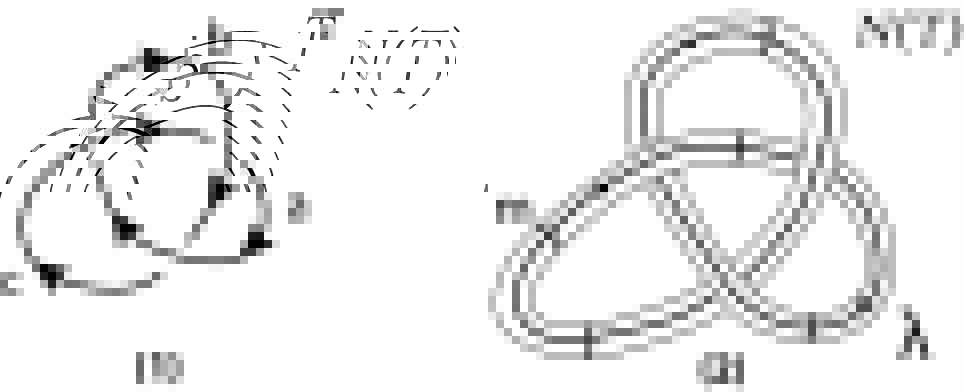}
\caption{\textbf{(1)} Trefoil knot $T$ \textbf{(2)} Tubular neighborhood $N(T)$}\label{App_Funda_solidtrefoil}
\end{wrapfigure}

For the case of the trefoil knot $T$ shown in Fig.~\ref{App_Funda_solidtrefoil}~(1), the meridian $m$ and the longitude $\lambda$ on the tubular neighborhood $N(T)$ are shown in Fig.~\ref{App_Funda_solidtrefoil}~(2). $N(T)$ is homeomorphic to a solid torus with the knot at the core of the torus. The meridian bounds a disk in the torus, that intersects $T$ transversely in a single point. The longitude runs along the surface of the torus in parallel to $T$, and so makes a second copy of the knot out along the surface of the torus. 
\begin{figure}[ht!]
\begin{center}
\includegraphics[width=14cm]{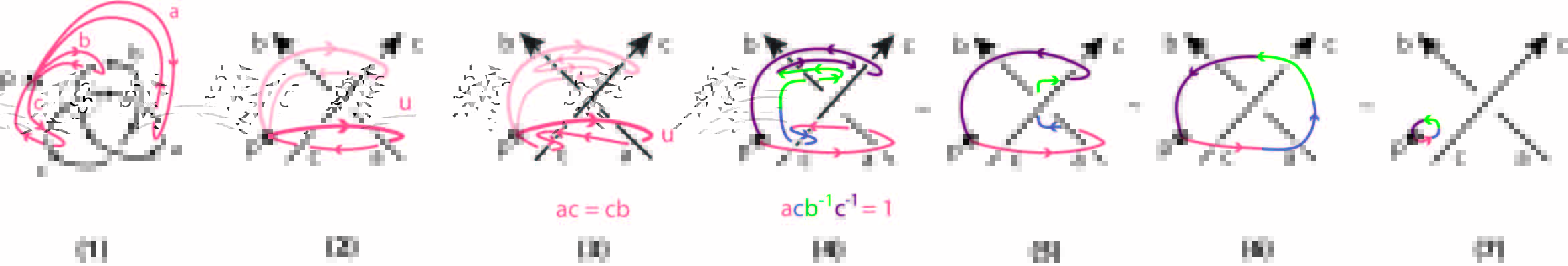}
\caption{\textbf{(1)} Generators represented as meridian loops  \textbf{(2)},\textbf{(3)} Homotopic loops \textbf{(4)},\textbf{(5)},\textbf{(6)},\textbf{(7)} Trivial curve }
\label{App_Funda_crossinglabels}
\end{center}
\end{figure}

The presentation of a knot group is generated by one meridian loop for each arc in a diagram of the knot. For the case of the trefoil, in Fig.~\ref{App_Funda_crossinglabels}~(1), we illustrate the three generators $a,b,c$ (in red) which are meridian elements associated with the corresponding arcs $a,b,c$ (in black). Each crossing gives rise to a relation among those elements. For example, let  us examine the crossing of the trefoil circled in Fig.~\ref{App_Funda_crossinglabels}~(1). By considering a loop $u$ in the close-up view of this crossing shown in Fig.~\ref{App_Funda_crossinglabels}~(2), it is shown that $u$ wraps around arcs $a$ and $c$ but can also slide upwards to wrap around arcs $c$ and $b$. In both cases, a homotopy of loop $u$ shows that we can write $u$ as a product of the generators of the fundamental group, see Fig.~\ref{App_Funda_crossinglabels}~(3). Since both homotopies describe the same loop $u$, we have $ac = cb$ which gives relation $b = c^{-1}ac$. Another way to obtain the same relation is by observing that curve $acb^{-1}c^{-1}$ contracts to a point and is therefore a trivial element of the fundamental group: $acb^{-1}c^{-1}=1$, see Fig.~\ref{App_Funda_crossinglabels}~(4),(5),(6),(7).


\begin{wrapfigure}{R}{0.3\textwidth} 
\centering
\includegraphics[width=0.25\textwidth]{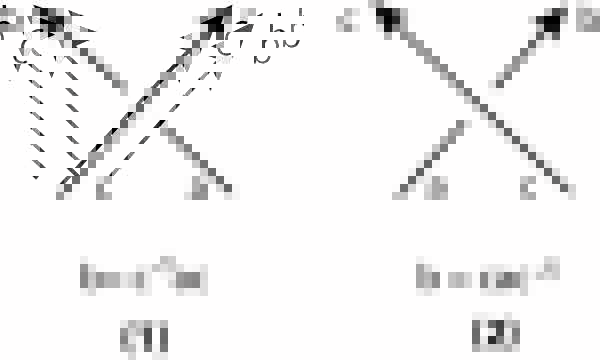}
\caption{\textbf{(1)} Positive crossing \textbf{(2)} Negative crossing}
\label{App_Funda_combi}
\end{wrapfigure}

Similarly, we can show that the relations obtained by the other two crossings are $a = b^{-1}cb$ and $c = a^{-1}ba$. More generally, given a diagram $D$ of an oriented knot $K$, if we label each arc of $D$, then the {\em fundamental group} of $K$ is the group whose generators are the labels of the arcs of $D$, and whose relations are the relations coming from the products of loops up to homotopy as we have just described them above. 
This presentation of the knot group is called the \textit{Wirtinger presentation} and its proof makes us of the Seifert-{}-van Kampen theorem, see for example~\cite{Rolfsen}. Hence for the trefoil knot, we have the presentation: $$ \pi_1(T)=\pi_1(S^3 \setminus N(T)) = (a,b,c \ | \ a = b^{-1}cb, b = c^{-1}ac, c = a^{-1}ba).$$ 

The fundamental group of a knot can be also defined in a combinatorial way as follows: consider a diagram of the knot and a crossing in diagram, as in Fig.~\ref{App_Funda_combi}~(1) or~(2), where the incoming undercrossing arc is labeled $a$, the overcrossing arc is labeled $c$ and the outgoing arc is labeled $b.$ Then write a relation in the form $b = c^{-1}ac$  for each positive crossing, as in Fig.~\ref{App_Funda_combi}~(1), and a relation $b = cac^{-1}$ for each negative crossing, as in Fig.~\ref{App_Funda_combi}~(2). The combinatorial approach defines the fundamental group as the group having one generator for each arc and one relation at each crossing in the diagram as we just specified them. One can  show that this group is invariant under the Reidemeister moves. This means that all diagrams of the same knot have the same fundamental group.

This combinatorial description is equivalent to the Wirtinger presentation. Indeed, see for example the relation coming from the positive crossing of Fig.~\ref{App_Funda_combi}~(1) and the relation coming from homotopic loops in Fig.~\ref{App_Funda_crossinglabels}~(3) or~(4). However, as we will see in Section~\ref{surgeryonK}, for the purpose of doing surgery we need the topological approach,  so that we can express the longitude in terms of the generators of the fundamental group of  $S^3 \setminus N(K)$. For more details on combinatorial group theory, the reader is referred to~\cite{Sti} or~\cite{MagnusKarassSolitar}.

\subsubsection{Differentiating knots}\label{NtK}
We can use the fact that homeomorphic knots have the same fundamental group to show that two knots are different. As an example, we will show that the trefoil knot is different from the trivial knot. 

\begin{wrapfigure}{R}{0.35\textwidth}
\centering
\includegraphics[width=0.3\textwidth]{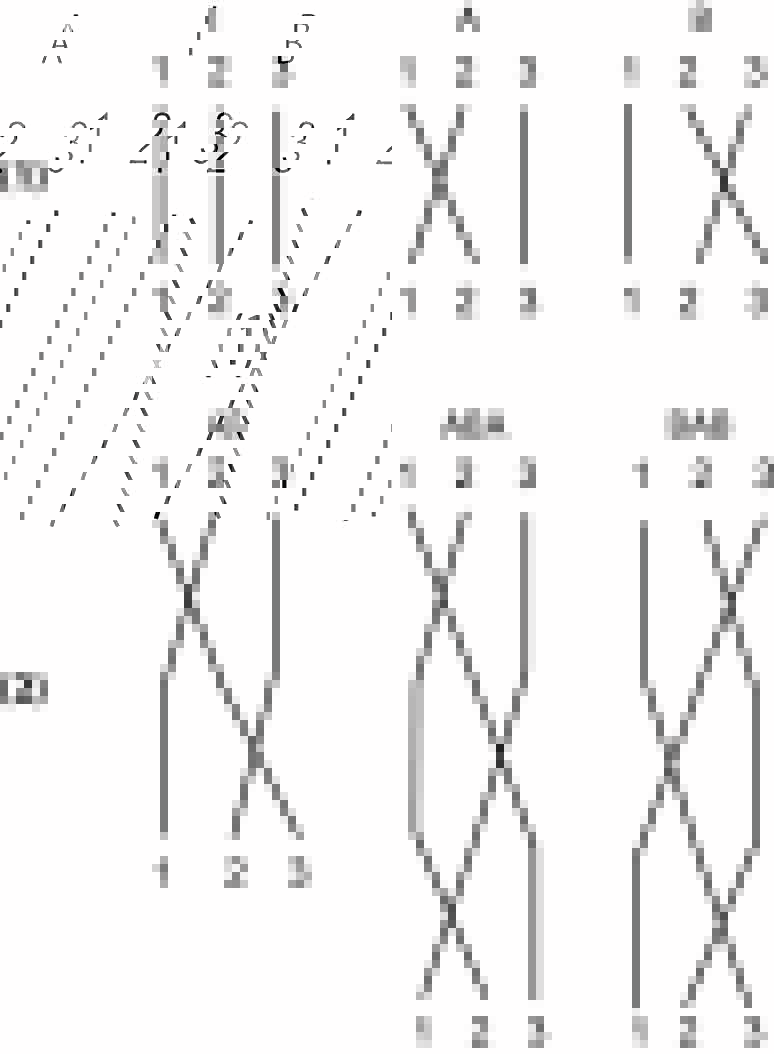}
\caption{\textbf{(1)} $R:\pi_1(T) \longrightarrow S_{3}$ \textbf{(2)} $ABA = BAB$ }
\label{App_Funda_permdiag}
\end{wrapfigure}

The trivial knot $U$, also called the unknot, is an embedding of $S^1$ as a geometrically round circle. Its diagram is a single arc $a$ with no crossings. Hence, its fundamental group $ \pi_1(S^3 \setminus N(U))$ has a single generator $a$ corresponding to arc $a$. This group is the infinite cyclic group $\mathbb Z$ which, as mentioned earlier, is also the fundamental group of the circle $S^1$. Hence, showing that the group of a knot is not homeomorphic to $\mathbb Z$ proves that this knot is not trivial.

Let us now compare the unknot to the trefoil knot. One can substitute $c = a^{-1}ba$ into the first two relations of $\pi_1(T)$ presented in previous section and see that the group of the trefoil has the simpler presentation $\pi_1(T) = (a,b| aba = bab)$. As we will show, this new presentation allows us to produce a surjective homeomorphism of $\pi_1(T)$ to the permutation group of three letters denoted by $S_{3}$. This proves that the trefoil knot is indeed knotted. Otherwise we would have a homeomorphism between commutative group $\mathbb Z$ and non-commutative group $S_{3}$.
 
To see the relation between $\pi_1(T)$ and $S_{3}$, consider permutations as represented by diagrams in Fig.~\ref{App_Funda_permdiag}~(1) where we indicate the mapping from a set $\{1,2,3\}$ to itself by drawing vertical arcs from a point to an image point. We compose two such diagrams by attaching the bottom row of one diagram to the top row of another. Two diagrams are equivalent if they represent the same permutation of the end points. In Fig.~\ref{App_Funda_permdiag}~(1) we have indicated a representation $R:\pi_1(T) \longrightarrow S_{3}$ of the trefoil group into the permutation group $\{1,2,3\}$ denoted by $S_{3}$, with $A= R(a)$ an interchange of the first two strands and $B = R(b)$ an interchange of the second two strands. In Fig.~\ref{App_Funda_permdiag}~(2), we also show the relationship $ABA = BAB$ that proves that this is indeed a representation of the trefoil fundamental group. 

In fact it is not hard to show that the permutation group $S_{3}$ has presentation $S_{3} = (\sigma_{1}, \sigma_{2} | \sigma_{1}\sigma_{2}\sigma_{1} = \sigma_{2}\sigma_{1}\sigma_{2},\sigma_{1}^{2} = 1, \sigma_{2}^{2} = 1)$ where $1$ denotes the identity element in the group. With $A=R(a) = \sigma_{1}$ and $B=R(b) = \sigma_{2}$, we see that $S_{3}$ is a quotient of $\pi_1(T)$ that is obtained by forcing $a$ and $b$ to have order $2$. In fact, one can prove that $a$ and $b$ have infinite order in $\pi_1(T)$ by a remarkable coincidence that the fundamental group of the trefoil knot is isomorphic to the braid group on three strands. We have not, in this paper, discussed the braid group and so we refer further information about this point to~\cite{PS}.

\subsubsection{Computing $\pi_1(\chi_{\mbox{\tiny K}}(S^3))$}\label{surgeryonK}
When the core curve $K$ of a non-trivial embedding $h(S_1^1 \times D^2)=N(K)$ is knotted, one cannot express $\lambda$ in terms of trivial longitudes and meridians, as was the case in Examples~\ref{ExFramedUnknot} and~\ref{ExFramedUnknot2}. In general, in order to compute the fundamental group of a $3$-manifold that is obtained by doing surgery on a blackboard framed knot $K$, we have to describe first how to write down a longitudinal element $\lambda$ in the fundamental group of the knot complement $S^3 \setminus N(K)$. 


\begin{wrapfigure}{L}{0.8\textwidth}  
\centering
\includegraphics[width=0.75\textwidth]{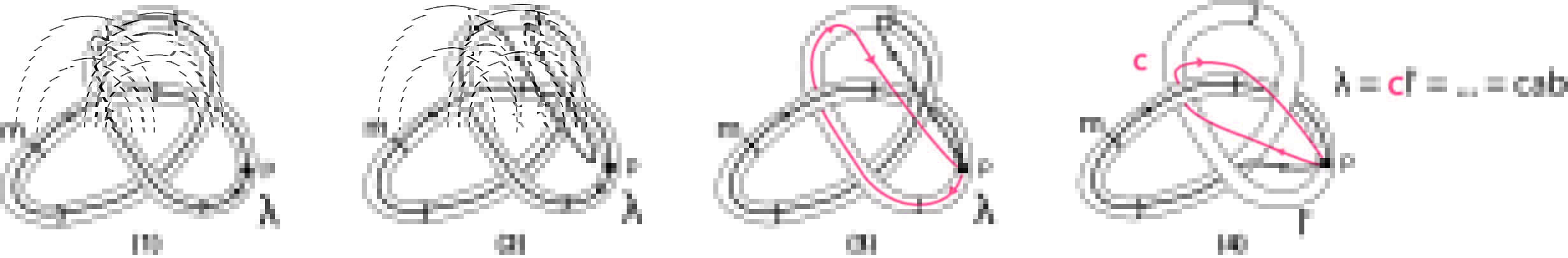}
\caption{\textbf{(1)} Longitude $\lambda$ in N(T) \textbf{(2)}\textbf{(3)}\textbf{(4)} Homotopy of $\lambda$ }
\label{App_Funda_solidtrefoil_Lambda}
\end{wrapfigure}

To do so, we homotope $\lambda$  to a product of the generators of $\pi_1(S^3 \setminus N(K))$ corresponding to the arcs that it underpasses. In this expression for the longitude, the elements $x$ that are passed underneath will appear either as $x$ or as $x^{-1}$   according to whether the knot is going to the right or to the left from the point of view of a traveler on the original longitude curve. Once the longitude $\lambda$ is expressed in terms of the generators of the fundamental group of  $S^3 \setminus N(K)$, we can calculate the fundamental group of $\chi_{\mbox{\tiny K}}(S^3)$ using Theorem~\ref{3d1long}.

\begin{wrapfigure}{R}{0.4\textwidth} 
\centering
\includegraphics[width=0.35\textwidth]{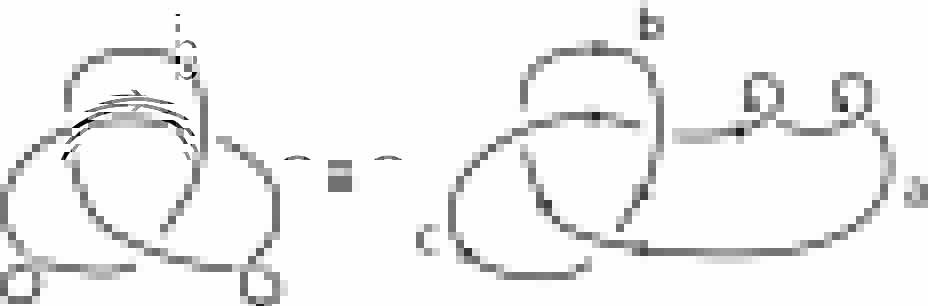}
\caption{Projection of the trefoil with total blackboard framing $1$}
\label{TrefoilP}
\end{wrapfigure}

For example, in Fig.~\ref{App_Funda_solidtrefoil_Lambda}~(1) we show a trefoil knot and the longitudinal element $\lambda$ in the fundamental group running parallel alongside it. Note that, for convenience, the basepoint $p$ is on the boundary of the torus but it could be anywhere in the complement $S^3 \setminus N(K)$. Each time that $\lambda$ goes under the knot we can run a line all the way back to the base point $p$ and then back to the point where $\lambda$ comes out from underneath the knot, see Fig.~\ref{App_Funda_solidtrefoil_Lambda}~(2),~(3) and~(4). By doing this, we have written, up to homotopy, the longitude as a product of the generators of the fundamental group that are passed under by the original longitude curve. Thus in the trefoil knot case, as shown in Fig.~\ref{App_Funda_solidtrefoil_Lambda}~(4), we see that the longitude is given by $\lambda= cab$.

\begin{exmp} \label{TrefoilSugeries}  \rm
We will now calculate the fundamental group of a $3$-manifold obtained by doing $3$-dimensional $1$-surgery on the trefoil knot  for two different projections. The first one is the simplest projection of the trefoil shown in Fig.~\ref{App_Funda_solidtrefoil_Lambda}~(1). It has three positive crossings yielding a blackboard framing number of $3$. The second one has two additional negative crossings thus having a blackboard framing number of $1$, see Fig.~\ref{TrefoilP}.


As mentioned in Section~\ref{Fundamen3d1}, surgery collapses the longitude $\lambda$, so the resulting fundamental group depends on how longitude $\lambda$ is expressed in the following relation: 

$$ \pi_1(\chi(S^3))=\frac{\pi_1(S^3\setminus N(T))}{<\lambda>} =\frac{\pi_1(T)}{<\lambda>}=(a,b,c \ | \ aba = bab, \lambda=1) $$

In the first case, by substituting $\lambda= cab$ and $c=a^{-1}ba$ to $\lambda=1$, we have $a^{-1}baab=1 \Leftrightarrow a=ba^{2}b$. Given that $aba = bab$, this implies that $a^2=baaba\Leftrightarrow a^2=babab \Leftrightarrow a^3=(ba)^3$. Notice now that  $(aba)^2=aba \cdot aba=bab \cdot bab=(ba)^3$. Thus by setting $A=a, B=ba$ and  $C=aba$ we have that  $A^3 = B^3=C^2$ and we only need to show that this is equal to $ABC$. Indeed, $ABC=a \cdot ba \cdot aba=(ba)^3$. Hence, the fundamental group of the  resulting manifold is isomorphic to the binary tetrahedral group $(A,B,C \ | \ A^3 = B^3=C^2=ABC)$ denoted $<3,3,2>$. It is also worth mentioning that the resulting manifold is isomorphic to $S^3/<3,3,2>$, the quotient of the $3$-sphere by an action of the binary tetrahedral group. For details on group actions the reader is  referred to~\cite{MilAct}. 


In the second case, the longitude $\lambda$ in the projection shown in Fig.~\ref{TrefoilP} is the same as the one in Fig.~\ref{App_Funda_solidtrefoil_Lambda}~(1) with two additional negative crossings along arc $a$. Hence, in this case $\lambda=caba^{-2}$. By substitution, we have $a^{-1}baaba^{-2}=1 \Leftrightarrow a^{3}=ba^{2}b$. Given that $aba = bab$, this implies that $a^3=baaba\Leftrightarrow a^4=babab \Leftrightarrow a^5=(ba)^3$.  Thus by setting $A=a, B=ba$ and  $C=aba$ we have that  $A^5 = B^3=C^2$ and   $ABC=a \cdot ba \cdot aba=(ba)^3$. The fundamental group of the  resulting manifold is isomorphic to the binary icosahedral group $(A,B,C \ | \ A^5 = B^3=C^2=ABC)$ denoted by $<5,3,2>$. The resulting manifold is isomorphic to $S^3/<5,3,2>$, the quotient of the $3$-sphere by an action of $<5,3,2>$. 
\end{exmp}

\begin{wrapfigure}{L}{0.25\textwidth} 
\centering
\includegraphics[width=0.2\textwidth]{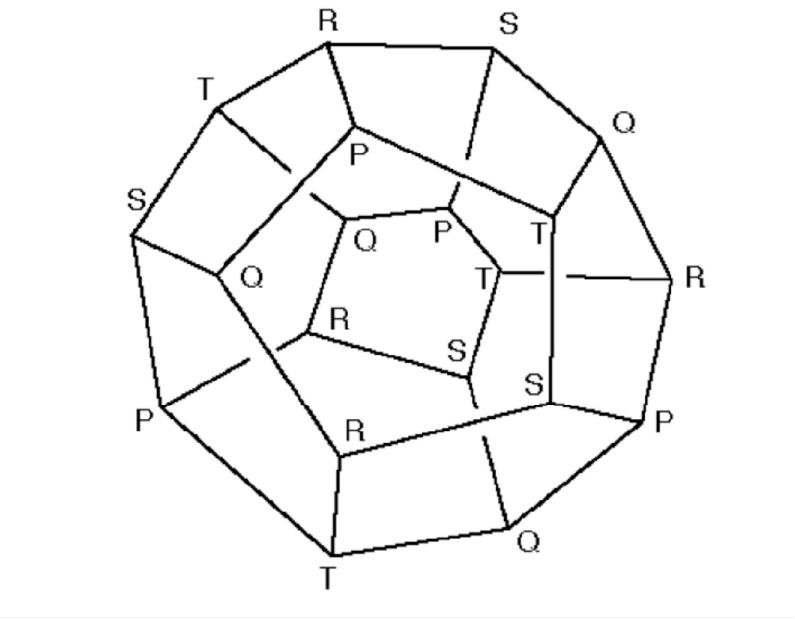}
\caption{Poincar\'{e} sphere}
\label{Poincare}
\end{wrapfigure}

This manifold is also known as the \textit{Poincar\'{e} homology sphere}, which can be described by identifying opposite faces of a dodecahedron according to the scheme shown in Fig.~\ref{Poincare} (for more details on this identification, see~\cite{SeTh}). It can be shown from this that the Poincar\'{e} homology sphere is diffeomorphic to the link of the variety $V((2,3,5))=\{ (z_1,z_2,z_3) \ | \ z_1{^2} + z_2{^3} + z_3{^5} = 0\}$, that is, the intersection of a small $5$-sphere around $0$ with $V((2,3,5))$. From this it is not hard to see that the Poincar\'{e} homology sphere can be also obtained as a $5$-fold cyclic branched covering of $S^3$ over the trefoil knot. For more details on the different descriptions of the Poincar\'{e} homology sphere, the reader is referred to~\cite{KirScha}.

This manifold has been of great interest, and is even thought by some physicists to be the shape of the geometric universe, see~\cite{Weeks,Luminet,Levin}. See also~\cite{SS4}.

\nolinenumbers


\newpage
\begin{thebibliography}{00}
					
\bibitem{SS0}	Lambropoulou	S.,	Samardzija	N.,	Diamantis	I.,	Antoniou	S.:	\textit{Topological	Surgery	and	Dynamics},	Mathematisches	Forschungsinstitut	Oberwolfach	Report	No.	26/2014,	Workshop:	Algebraic	Structures	in	Low-Dimensional	Topology.	DOI:	10.4171/OWR/2014/26	(2014).

\bibitem{SS1}	Antoniou	S.,	Lambropoulou	S.:	\textit{Extending	Topological	Surgery	to	Natural	Processes	and	Dynamical	Systems}.	PLOS	ONE 12(9).	DOI:	10.1371/journal.pone.0183993 (2017).	

\bibitem{SS2}	Lambropoulou	S.,	Antoniou	S.:	\textit{Topological	Surgery,	Dynamics	and	Applications	to	Natural	Processes}.	Journal	of	Knot	Theory	and	its	Ramifications 26(9).	DOI:	10.1142/S0218216517430027 (2016).
	
\bibitem{SS3}	Antoniou	S.,	Lambropoulou	S.:	\textit{Topological	Surgery	in	Nature}.	Book	`Algebraic	Modeling	of	Topological	and	Computational	Structures	and	Applications',	Springer	Proceedings	in	Mathematics	and	Statistics	Vol. 219, DOI: 10.1007/978-3-319-68103-0 (2017).

\bibitem{SS4}	Antoniou	S.:	\textit{Mathematical Modeling Through Topological Surgery and Applications}. Springer Theses Book Series. Springer International Publishing. DOI:  10.1007/978-3-319-97067-7 (2018).

\bibitem{Wal}	Wallace A.H.:	\textit{Modifications and cobounding manifolds}. Canad. J.  Math. 12, 503-528 (1960).

\bibitem{Milsur} Milnor J.:	\textit{A procedure for killing the homotopy groups of differentiable manifolds}. Symposia in Pure Math., Amer. Math. Soc. 3, 39-55 (1961).

\bibitem{Ra}	Ranicki	A.:\textit{	Algebraic	and	Geometric	Surgery}.	Oxford	Mathematical	Monographs,	Clarendon	Press	(2002).

\bibitem{Mil}	Milnor	J.:	\textit{Morse	Theory}.	Princeton	University	Press	(1963).		

\bibitem{Ki}	Kiehn	R.M.:	\textit{Non-equilibrium	systems	and	irreversible	processes	-	Adventures	in	applied	topology	vol.	1	-	Non	equilibrium	thermodynamics},	pp.	147,	150	University	of	Houston	Copyright	CSDC	Inc (2013).

\bibitem{PhyG}  Cowern D.: \textit{Fun with Vortex Rings in the Pool}. Available from:\url{https://www.youtube.com/watch?v=72LWr7BU8Ao}. 

\bibitem{Ro}	Rolfsen	D.:	\textit{Knots	and	links}.	Publish	or	Perish	Inc.	AMS	Chelsea	Publishing	(2003).	

\bibitem{Munkres} Munkres J.R.: \textit{Topology}, Second Edition. Prentice Hall, Incorporated (2000).	

\bibitem{LickTh} Lickorish W.B.R.: \textit{A representation of orientable combinatorial 3-manifolds}, Ann. of Math. (2) 76, pp. 531-540 (1962).
														
\bibitem{PS}	Prasolov	V.V.,	Sossinsky,	A.B.:	\textit{Knots,	links,	braids	and	3-manifolds}.	AMS	Translations	of	Mathematical	Monographs	154	(1997).

\bibitem{Kirby}	Kirby R.C.:	\textit{A calculus for framed links in $S^3$}. Invent. Math., 45, pp. 35-56 (1978).

\bibitem{DNA} Kauffman L.H., Lambropoulou S., Buck D.: \textit{DNA Topology}. (book in preparation). Chapter on ‘Three Dimensional Topology for DNA’.

\bibitem{BHsurg} Stathis A., Kauffman L.H., Lambropoulou S.: \textit{Black holes and topological surgery}. Preprint (2018). Available from: \url{https://arxiv.org/pdf/1808.00254}.

\bibitem{ER_EPR_} Maldacena J., Susskind L.: \textit{Cool horizons for entangled black holes}. Fortsch.Phys. 61, 781-811. DOI: 	10.1002/prop.201300020 (2013). Available from: \url{https://arxiv.org/abs/1306.0533}. 

\bibitem{ER_EPR} Susskind L.: \textit{Copenhagen vs Everett, Teleportation, and ER=EPR}. Fortschr. Phys. 64, No. 6–7, 551–564.  DOI: 10.1002/prop.201600036  (2016).




\bibitem{ERB} Einstein A., Rosen N.: \textit{The particle problem in the general theory of relativity}, Phys. Rev. 48, 73-77 (1935).

\bibitem{MisWhee} Misner C.W., Wheeler J.A.: \textit{Classical physics as geometry}. Ann. Phys. 2: 525 (1957).

\bibitem{EmWorm} Collas P., Klein D.: \textit{Embeddings and time evolution of the Schwarzschild wormhole}, American Journal of Physics 80, 203. DOI: 10.1119/1.3672848 (2012). 

\bibitem{Inter} James O., Tunzelmann E., Franklin P., Thorne K.S.: \textit{Visualizing Interstellar's Wormhole}. American Journal of Physics 83, 486. DOI:/10.1119/1.4916949 (2015). 

\bibitem{KiehnSmall} Kiehn R.M.: \textit{Falaco Solitons, Cosmic Strings in a Swimming Pool}. Preprint (2001). Available from: \url{http://arxiv.org/pdf/gr-qc/0101098v1}.

\bibitem{Thorne}  Thorne K.S.: \textit{Black Holes \& Time Warps: Einstein's Outrageous Legacy}, W.W. Norton \& Company (1994). 

\bibitem{Whee_QF} Wheeler J.A.:	\textit{On the Nature of Quantum Geometrodynamics}, Annals Phys. 2, 604-614. DOI: 10.1016/0003-4916(57)90050-7 (1957).



%
\bibitem{WHE} Baez J.C., Vicary J.: \textit{Wormholes and Entanglement}. Classical and Quantum Gravity 31. DOI: 10.1088/0264-9381/31/21/214007 (2014). Available from: \url{https://arxiv.org/abs/1401.3416}.

\bibitem{EMLK} Kauffman L.H., Mehrotra E.: \textit{Topological Aspects of Quantum Entanglement} (2016). Available from: \url{https://arxiv.org/abs/1611.08047}. 

\bibitem{WITT} Witten E.: \textit{Quantum Field Theory and the Jones Polynomial}. Comm. in Math. Phys. Vol. 121, 351-399 (1989).
%

\bibitem{Kibble} Kibble T.W.B.: \textit{Topology of cosmic domains and strings}. J. Phys. A: Math. Gen., Vol. 9, No. 8: 1387. DOI:10.1088/0305-4470/9/8/029 (1976).

\bibitem{Hawk}	Hawking	S.W.:	\textit{Black	Holes	from	cosmic	strings},	Physics	Letters	B	231	No.	237	DOI: 10.1016/0370-2693(89)90206-2	(1989).

\bibitem{Cald}	Caldwell	R.R.,	Casper	P.:	\textit{Formation	of	Black	Holes	from	Collapsed	Cosmic	String	Loops},	Phys.Rev.	D53.	DOI:	10.1103/PhysRevD.53.3002	(1996).

\bibitem{McGibb}	MacGibbon	J.H.,	Brandenberger	R.H.,	Wichoski	U.F.:	\textit{Limits	on	Black	Hole	Formation	from	Cosmic	String	Loops},	Phys.Rev.	D57	-1998	2158-2165.	DOI:	10.1103/PhysRevD.57.2158	(1998).	

%
\bibitem{Weeks}	Weeks  J.R.: \textit{The Shape of Space}. CRC Press (2001).

\bibitem{Luminet} Luminet J.-P., Weeks J.R, Riazuelo A., Lehoucq R.,Uzan J.-P.: \textit{Dodecahedral space topology as an explanation for weak wide-angle temperature correlations in the cosmic microwave background}. Nature 425, 593–595 (2003).



\bibitem{Levin}	Levin  J.: \textit{Topology and the Cosmic Microwave Background}. Physics Reports 365, 251–333 (2002).
%

\bibitem{JPrecent} Luminet J.-P.: \textit{The Status of Cosmic Topology after Planck Data}. Preprint (2016). Available from: \url{https://arxiv.org/abs/1601.03884}. 

\bibitem{Ad} Adams	C.:	\textit{The	Knot	Book,	An	Elementary	Introduction	to	the	Mathematical	Theory	of	Knots}.	American	Mathematical	Society 	(2004).


					
				
				
					
													

										
														
							
		


\bibitem{Rolfsen}  Rolfsen	D.:	\textit{Knots	and	links}.	Publish	or	Perish	Inc.	AMS	Chelsea	Publishing	(2003).

\bibitem{Sti}	Stillwell J.:	\textit{Classical Topology and Combinatorial Group Theory}. Graduate Texts in Mathematics, Springer-Verlag New York (1993).

											

			
		

						
				
														
													
																			



	


\bibitem{MagnusKarassSolitar} Wilhelm M., Karrass A., Solitar D.: \textit{Combinatorial Group Theory}, New York: Dover Publications (2004).

\bibitem{MilAct}	Milnor	J.:	\textit{On the 3-dimensional Brieskorn manifolds M(p,q,r). Knot Groups and 3-Manifolds - Papers dedicated to the memory of R.H.Fox}. Annals of Mathematics Studies 84. Princeton University Press, Princeton, NJ (1975).

\bibitem{SeTh}	Threlfall H., Seifert W. : \textit{A textbook of topology}. Academic Press  (1980).
 
\bibitem{KirScha}	 Kirby R.C., Scharlemann M.G.: \textit{Eight faces of the Poincaré homology 3-sphere}. Uspekhi Matematicheskikh Nauk, [N. S.]. DOI: 10.1016/B978-0-12-158860-1.50015-0 (1979).



\end{thebibliography}
\end{document}